\pdfoutput=1 %
\documentclass[12pt]{report}

\usepackage[utf8]{inputenc}
\def\final{1}
\usepackage{color}
\usepackage{subfigure}
\usepackage{url}
\usepackage{graphicx}
\usepackage{amsmath}
\usepackage{amsthm}
\usepackage{bm}
\usepackage[plainpages=false]{hyperref}
\usepackage{chngcntr}
\usepackage{titlesec}

\renewcommand{\thechapter}{\Roman{chapter}}
\titleformat{\chapter}
{\normalfont\LARGE\bfseries\filcenter}{\thechapter.}{1 em}{}

\counterwithout{section}{chapter}
\titleformat{\section}
{\normalfont\large\bfseries\filcenter}{\thesection.}{1 ex}{}
\titleformat{\subsection}[runin]
{\normalfont\normalsize\bfseries\filcenter}{\thesubsection.}{1 ex}{}
\usepackage[charter]{mathdesign}
\usepackage[mathcal]{eucal}
\usepackage[margin=1.2 in]{geometry}
\usepackage[square,numbers,sort&compress]{natbib}

\ifnum\final=0
\newcommand{\mynote}[1]{\marginpar{\tiny\sf #1}}
\else
\newcommand{\mynote}[1]{}
\fi

\usepackage{theomac}
\newtheoremWithMacro{theorem}{Theorem}
\newtheoremWithMacro{lemma}{Lemma}[section]

\newtheoremWithMacro{corollary}{Corollary}
\newtheoremWithMacro{prop}[lemma]{Proposition}

\newcommand{\figref}[1]{Figure \ref{fig:#1}}

\newcommand{\lemref}[1]{Lemma \ref{lemma:#1}}
\newcommand{\lemrefX}[1]{\ref{lemma:#1}}

\newcommand{\propref}[1]{Proposition \ref{prop:#1}}
\newcommand{\theoref}[1]{Theorem \ref{theo:#1}}
\newcommand{\secref}[1]{Section \ref{sec:#1}}
\newcommand{\secrefX}[1]{\ref{sec:#1}}
\newcommand{\chapref}[1]{Part \ref{chap:#1}}

\newcommand{\lemlab}[1]{\label{lemma:#1}}
\newcommand{\proplab}[1]{\label{prop:#1}}
\newcommand{\theolab}[1]{\label{theo:#1}}
\newcommand{\seclab}[1]{\label{sec:#1}}
\newcommand{\chaplab}[1]{\label{chap:#1}}

\renewcommand{\vec}[1]{\mathbf{#1}}

\newcommand{\iprod}[2]{\left\langle {#1},{#2}\right\rangle}

\newcommand{\R}{\mathbb{R}}
\newcommand{\Z}{\mathbb{Z}}
\newcommand{\Rep}{\operatorname{Rep}}
\newcommand{\Rad}{\operatorname{Rad}}

\newcommand{\Gal}{(\tilde{G},\varphi,\tilde{\bm{\ell}})}
\newcommand{\Gad}{(\tilde{G},\varphi,\tilde{\vec d})}
\newcommand{\Gtilde}{\tilde{G}}
\newcommand{\elltilde}{\tilde{\bm{\ell}}}

\newcommand{\bgamma}{\bm{\gamma}}
\newcommand{\Gpa}{G(\vec p,\Phi)}
\newcommand{\Vtilde}{\tilde{V}}
\newcommand{\Etilde}{\tilde{E}}
\newcommand{\Euc}{\operatorname{Euc}}
\newcommand{\into}{\hookrightarrow}

\newcommand{\Trans}{\Lambda}

\newcommand{\Teich}{\operatorname{Teich}}

\newcommand{\JMat}[1]{ \left( \begin{array}{rr} #1 \end{array} \right)}

\newcommand{\Id}{\operatorname{Id}}
\newcommand{\rep}{\operatorname{rep}}

\newcommand{\teich}{\operatorname{teich}}

\newcommand{\GammaTT}{$\Gamma$-$(2,2)$ }
\newcommand{\GammaOO}{$\Gamma$-$(1,1)$ }
\newcommand{\GammaCL}{$\Gamma$-colored-Laman }
\newcommand{\cent}{\operatorname{cent}}
\newcommand{\Cent}{\operatorname{Cent}}

\theoremstyle{remark}

\newcommand{\eop}{\hfill$\qed$}

\counterwithout{equation}{chapter}
\counterwithout{figure}{chapter}

\begin{document}
\title{Generic rigidity of frameworks with orientation-preserving crystallographic symmetry}
\author{Justin Malestein\thanks{Temple University, \url{justmale@temple.edu}}
\and Louis Theran\thanks{Freie Universität Berlin, \url{theran@math.fu-berlin.edu}}}
\date{}
\maketitle
\begin{abstract}
We extend our generic rigidity theory for periodic frameworks in the plane to
frameworks with a broader class of crystallographic symmetry.
Along the way we introduce a new class of
combinatorial matroids and associated linear representation results
that may be interesting in their own right.  The same techniques immediately yield a
Maxwell-Laman-type combinatorial characterization for frameworks embedded in 2-dimensional cones
that arise as quotients of the plane by a finite order rotation.
\end{abstract}

\section{Introduction} \seclab{intro}

A \emph{crystallographic framework} is an \emph{infinite} planar structure, \emph{symmetric} with respect to a crystallographic
group, made of fixed-length bars connected by universal joints with full rotational freedom.  The allowed continuous
motions preserve the lengths and connectivity of the bars (as in the finite framework case) and (this is the new addition)
\emph{symmetry with respect to the group $\Gamma$}.  However, the representation of $\Gamma$ is \emph{not} fixed and
may change. Figures \ref{fig:gam2graphtoperiodic} and \ref{fig:gam4graphtoperiodic} show examples.
\begin{figure}[htbp]
\centering
\subfigure[]{\includegraphics[width=.55\textwidth]{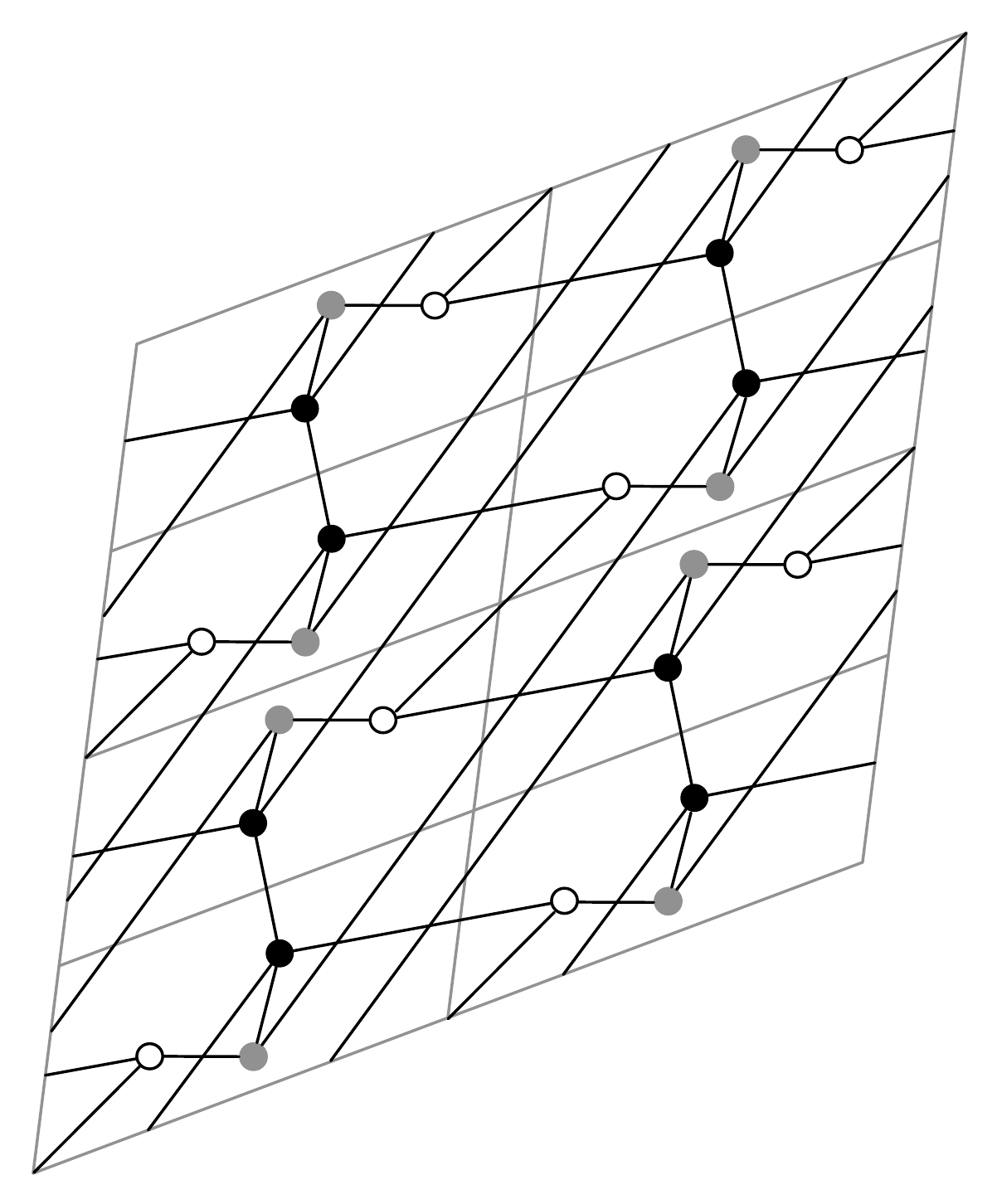}}
\subfigure[]{\includegraphics[width=.35\textwidth]{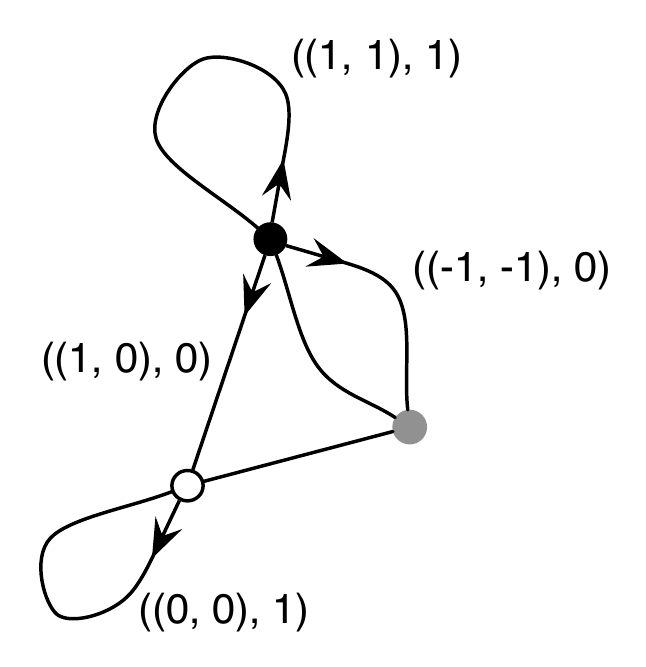}}
\caption{A $\Gamma_2$-crystallographic framework:
(a) A piece of an infinite crystallographic framework with $\Gamma_2$ symmetry.  The group
$\Gamma_2$ is generated by an order $2$ rotation and translations.  The origin, which is
a rotation center, is at the center of the diagram.  Each quadrilateral (with gray edges) is a
fundamental domain of the $\Gamma_4$-action on $\R^2$.
(b) The associated \emph{colored graph} capturing the underlying combinatorics.  Edges that
are not marked and oriented are colored with the identity element of $\Gamma_2$.  The vertices
in (b) are colored coded to show the fibers over each of them in (a).}
\label{fig:gam2graphtoperiodic}
\end{figure}
A crystallographic framework is
\emph{rigid} when the only allowed motions (that, additionally, must act on the
representation of $\Gamma$) are Euclidean isometries  and \emph{flexible} otherwise.

The topic of this paper is the following question: \emph{Which crystallographic frameworks are rigid
and which are flexible?}
In its most general form, this question doesn't seem computationally tractable: even for finite frameworks,
the best known algorithms rely on exponential-time Gröbner basis computations.  However, \emph{generically}---and almost
all crystallographic frameworks are generic---we  can say more with \theoref{main} (stated below in \secref{mainstatement}):
generic rigidity and flexibility depend on the \emph{combinatorial type} of the framework, given by a
\emph{colored graph}, which is a finite, directed graph with elements of a group on the edges.  Moreover,
\theoref{main} is a ``good characterization'' in that a polynomial time combinatorial algorithm can decide
whether a colored graph corresponds to generically rigid crystallographic frameworks.
\begin{figure}[htbp]
\centering
\subfigure[]{\includegraphics[width=0.55\textwidth]{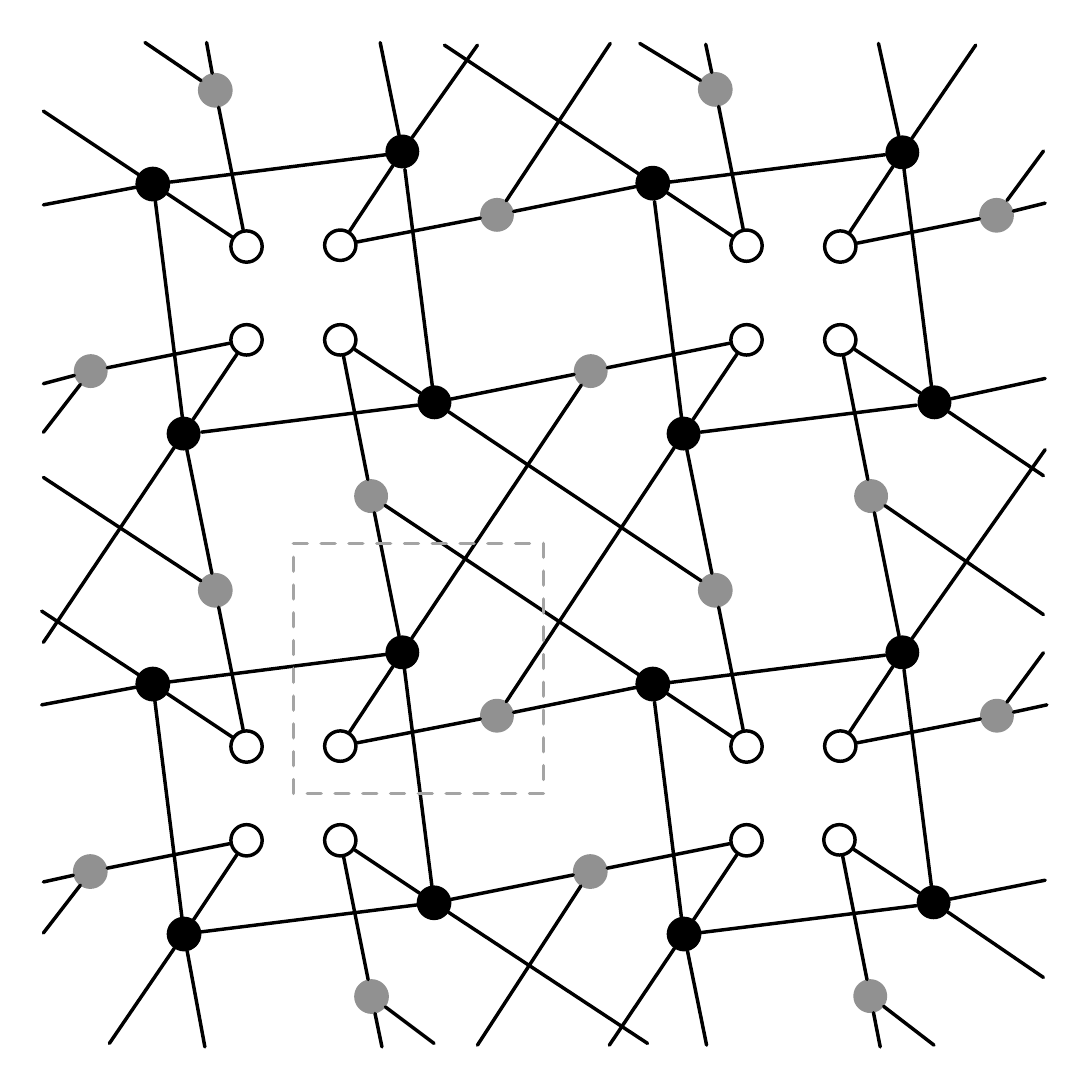}}
\subfigure[]{\includegraphics[width=0.35\textwidth]{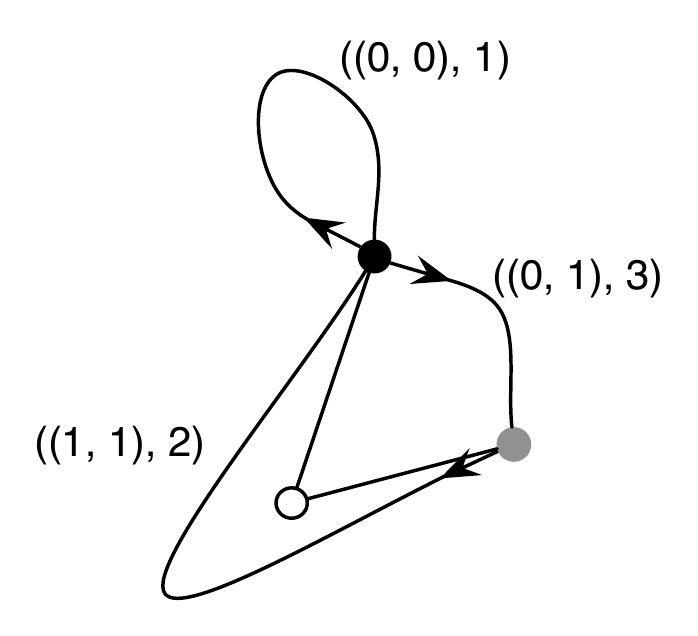}}
\caption{A $\Gamma_4$-crystallographic framework:
(a) A piece of an infinite crystallographic framework with $\Gamma_4$ symmetry.  The group
$\Gamma_4$ is generated by an order $4$ rotation and translations.  The fundamental domain of the
$\Gamma_4$-action on $\R^2$ is shown as a dashed box.
(b) The associated colored graph capturing the underlying combinatorics.  The color coding
conventions are as in \figref{gam2graphtoperiodic}.
}
\label{fig:gam4graphtoperiodic}
\end{figure}

Thus, \theoref{main} is a true analog of the landmark Maxwell-Laman Theorem \cite{M64,L70} from rigidity theory,
which characterizes generic rigidity and flexibility of finite frameworks in the plane.  We stress that
the genericity hypotheses made by \theoref{main} are \emph{on the geometry of the framework only}, which is
the same as genericity assumptions from the theory of finite frameworks.

\subsection{Algebraic definition of rigidity and flexibility}
A $\Gamma$-crystallographic framework is given by the data $\Gal$,
where $\Gtilde=(\Vtilde,\Etilde)$ is an infinite graph, $\Gamma$ is a crystallographic group,
$\varphi$ is a free $\Gamma$-action with finite quotient on $\Gtilde$, and $\elltilde$ is an
assignment of positive lengths to each edge $ij\in \Etilde$.  To keep the terminology in this framework manageable,
we will refer simply to \emph{frameworks} when the context is clear, with the understanding that the frameworks appearing in
the paper are crystallographic.

A \emph{realization} $\Gpa$ of the abstract framework $\Gal$ is defined to be an assignment
$\vec p=\left(\vec p_{i}\right)_{i\in\Vtilde}$ of points to the vertices of $\Gtilde$ and
a representation $\Phi$ of $\Gamma \into \Euc(2)$ by Euclidean isometries
acting discretely and co-compactly, such that
\begin{align}
||\vec p_i - \vec p_j||  =  \elltilde_{ij} & \text{\qquad for all edges $ij\in \Etilde$} \label{lengths} \\
\Phi(\gamma)\cdot \vec p_i  =  \vec p_{\gamma(i)} & \text{\qquad for all group elements $\gamma\in \Gamma$ and vertices $i\in\Vtilde$} \label{equivariant}
\end{align}
The condition \eqref{lengths} says that a realization respects the given edge lengths, which appears in the theory
of finite frameworks.  Equation \eqref{equivariant} says that, if we hold $\Phi$ fixed, regarded as a map
$\vec p : \Vtilde\to \R^2$, $\vec p$ is equivariant.  However, $\Phi$ is, in general, \emph{not} fixed.  This is a very important
feature of the model: the motions  available to the framework include those that deform the representation $\Phi$ of $\Gamma$,
provided this happens in a way compatible with the abstract $\Gamma$-action $\varphi$.

\subsection{Rigidity via realization and configuration spaces}
The \emph{realization space} $\mathcal{R}\Gal$  (shortly $\mathcal{R}$)
of an abstract framework is defined as the set of its realizations.  Motions of the framework are, then,
continuous paths in the realization space.  To factor out trivial motions, we define the \emph{configuration
space} $\mathcal{C}$ to be $\mathcal{C}=\mathcal{R}/\Euc(2)$.  With this definition, we can formally
define rigidity: a realization $\tilde{G}(\vec p, \Phi)$ is \emph{rigid} if it is isolated in $\mathcal{C}$; otherwise
the realization of the framework is \emph{flexible}, and there is a continuous path in $\mathcal{C}$ through $(\vec p,\Phi)$
giving a \emph{motion} of the framework.
(See \secref{continuous} for a detailed treatment of these spaces.)

We remark that the definition makes it clear that we are interested in what is sometimes called ``\emph{local rigidity}''
in the literature: the configuration space may have multiple connected components, each with a different dimension.  We
are not concerned with the stronger notion of ``\emph{global rigidity}'', which requires that $\mathcal{C}$ be a single
point.

\subsection{Main result: Crystallographic Maxwell-Laman}\seclab{mainstatement}
Our main result is the following ``Maxwell-Laman-type''
theorem for  crystallographic frameworks where the symmetry group is
generated by translations and a finite order rotation.
The ``\emph{$\Gamma$-colored-Laman graphs}'' appearing in the statement are defined in \secref{gamma-laman};
genericity is defined in detail in \secref{infinitesimal}, but the term is used in the standard sense of
algebraic geometry: generic frameworks are the (open, dense) complement of a proper algebraic
subset of the configuration space.
\begin{theorem}[\mainthm]\theolab{main}
Let $\Gamma$ be a crystallographic group generated by translations and rotations.
A generic $\Gamma$-crystallographic framework $(\tilde{G}, \varphi, \tilde{\bm{\ell}})$
is minimally rigid if and only if its colored quotient graph is $\Gamma$-colored-Laman.
\end{theorem}

\subsection{The Main Theorem for orbifolds} An alternative interpretation of
\theoref{main} is that it characterizes rigidity of finite frameworks in Euclidean
orbifolds with geodesic bars.  The orbifold is obtained by taking
the quotient $\R^2/\Gamma$, where $\Gamma$ is generated by translations and rotations.  This is
what is meant elsewhere in the literature when ``torus'' \cite{R09,R11} or ``cone'' \cite{W88}
frameworks are discussed. Since we don't work in this formalism, we leave the issue of an intrinsic
\theoref{main} aside.

\subsection{Cone frameworks}
A particularly interesting simplification---that we will see as a ``warm up'' for  \theoref{main}---is
when the symmetry is given by a rotation around the origin through angle $2\pi/k$.\footnote{The proof tells us more, namely
that the same theorems about cone frameworks are true for \emph{any} order $k$ rotation, but for simplicity, we
restrict ourselves to the case arising as part of the crystallographic setting.}
In this case, the quotient is a flat cone with opening angle $2\pi/k$, so we call such frameworks \emph{cone frameworks}.
For the purposes of cone frameworks, we will identify $\Z/k\Z$ with this subgroup of $SO(2)$.

The formalism is very similar to that for crystallographic frameworks, except everything is
finite. A cone framework is given by $(\Gtilde, \varphi, \elltilde)$,
where $\Gtilde=(\Vtilde,\Etilde)$ is a finite graph,  $\varphi$ is a free $\Z/k\Z$-action, and
$\elltilde$ is an assignment of positive lengths to each edge $ij\in \Etilde$.  Realizations
$\tilde{G}(\vec p)$ of the abstract framework $(\Gtilde, \varphi, k, \elltilde)$ are point sets
$\vec p=\left(\vec p_{i}\right)_{i\in\Vtilde}$ satisfying
\begin{align}
||\vec p_j - \vec p_i||  =  \elltilde_{ij} & \text{\qquad for all edges $ij\in \Etilde$} \label{cone-lengths} \\
\gamma \cdot \vec p_i  =  \vec p_{\gamma(i)} &
\text{\qquad for all group elements $\gamma\in \Z/k\Z$ and vertices $i\in\Vtilde$}
\label{cone-equivariant}
\end{align}
and the definitions of the realization and configurations spaces, and well as rigidity and flexibility are
similar to the crystallographic case.

We prove the following theorem in \secref{cone-rigidity}; \emph{cone-Laman} graphs are defined in
\secref{cone-sparse}.
\begin{theorem}[\conethm]\theolab{cone}
A  generic cone framework is minimally rigid if and
only if the associated colored graph $(G,\bgamma)$ is cone-Laman.
\end{theorem}

\subsection{Crystallographic direction networks}
In order to prove the rigidity Theorems \ref{theo:main} and \ref{theo:cone},
we will use crystallographic direction networks.
A \emph{$\Gamma$-crystallographic direction network} $\Gad$ consists of an infinite graph $\Gtilde$
with a free $\Gamma$-action $\varphi$ on the edges and vertices, and an assignment of a \emph{direction}
$\tilde{\vec d}_{ij}$ to each edge $ij\in \tilde{E}$.

We define a realization $\Gpa$ of $\Gad$ to be a mapping of $\Vtilde$ to a point set $\vec p$ and a
representation $\Phi$ of $\Gamma$ by Euclidean isometries such that
\begin{align}
\iprod{\vec p_i - \vec p_j}{\tilde{\vec d}_{ij}^\perp}  =  0 & \text{\qquad for all edges $ij\in \Etilde$} \label{directions} \\
\Phi(\gamma)\cdot\vec p_i  =  \vec p_{\gamma(i)} & \text{\qquad for all group elements $\gamma\in \Gamma$ and vertices $i\in\Vtilde$} \label{equivariantD}
\end{align}
Since setting all the $\vec p_i$ equal and $\Phi$ to be trivial produces a realization, the realization space is never empty.
For our purpose, though, such realizations are degenerate.  We define a realization of a crystallographic
direction network to be \emph{faithful} if none of the edges of $G$ are realized with coincident endpoints.

\subsection{Crystallographic Direction Network Theorem}
Our second main result is an exact characterization of when a generic direction network admits a faithful
realization, in the spirit of Whiteley's Parallel Redrawing Theorem \cite[Section 4]{W96}.
\begin{theorem}[\directionthm]\theolab{direction}
Let $\Gamma$ be a crystallographic group generated by translations and rotations.  A generic realization of a
$\Gamma$-crystallographic direction network $\Gad$ has a
faithful realization if and only if its associated colored graph is $\Gamma$-colored-Laman.
This realization is unique up to translation and scaling.
\end{theorem}

\subsection{Proof strategy for \theoref{main}}
The deduction of the rigidity \theoref{main} from \theoref{direction} uses the natural
extension of our \emph{periodic direction network method} from \cite{MT10}.  Briefly, the steps are:
\begin{itemize}
\item We reduce the problem of rigidity, as is standard in the field, to a linearization called
\emph{infinitesimal rigidity}.  (This is defined in \secref{infinitesimal}.)
\item We then show that minimal infinitesimal rigidity of a colored graph $(G,\bgamma)$
coincides with generic direction networks on $(G,\bgamma)$ having a faithful realization
up to translation and scaling. (This is done in \secref{main-proof}.)
\item \theoref{main} is then immediate from \theoref{direction}.
\end{itemize}

Although the steps in Sections \ref{sec:continuous}--\ref{sec:main-proof}
are, in light of \cite{ST10,MT10} somewhat routine, we remark at this point
that the translation between infinitesimal rigidity and faithful direction network
realizability \emph{does not} go through when the symmetry group contains reflections.
Thus, this additional hypothesis is forced by our proof method.  While, with some additional
effort, we might be able to extend the Direction Network \theoref{direction} to all
two-dimensional  crystallographic groups, this improvement would not, by itself, give a more
general rigidity theorem.

\subsection{Roadmap}
Most of the work in this paper is in the proof of \theoref{direction}, which proceeds
in three parts:
\begin{itemize}
\item \chapref{groups} studies the crystallographic groups $\Gamma_k$ for $k=2,3,4,6$,
giving convenient coordinates to their representation spaces (Sections \secrefX{repspace}--\secrefX{radical})
and developing a matroid on the $\Gamma_k$ (\propref{Mgammarank}).
\item \chapref{graphs} contains the combinatorial part of the proof of \theoref{direction},
developing $\Gamma$-graded sparse graphs (definitions are given in
Sections \secrefX{gamma22} and \secrefX{gamma-laman})
in terms of matroidal (\propref{gamma11}) and decomposition (\propref{gamma22-decomp})
properties.
\item \chapref{direction-networks} then develops the theory of direction networks
and links the combinatorics of colored graphs defined by sparsity conditions to the
geometry of direction networks.  The main result of \chapref{direction-networks} is
\theoref{direction}, which is deduced from \propref{crystal-collapse}.
\end{itemize}
Readers familiar with \cite{MT10} will notice that the broad strokes of the proof plan is
similar, but that there is no ``natural representation'' step, in which dependence and
independence in colored graph matroids are related to determinantal formulas.  The reason
for this is that, in the crystallographic case, the variables arising from direction network
realization problems do not separate out as cleanly.  Thus, an alternative viewpoint of
\chapref{direction-networks} is that it introduces new techniques for proving linear
representability of sparsity matroids.

\subsection{Related work}
The results of this paper are a direct extension of the theory we introduced in
\cite{MT10}, and they stand on a similar foundation.  Our paper \cite{MT10}
contains a detailed discussion from several historical perspectives.

The general area of rigidity with symmetry has been somewhat active in the past few years,
but the results here are independent of much of it.  For completeness, we review some work along
similar lines. A specialization of our \cite[Theorem A]{MT10} is
due to Ross \cite{R09,R11}. Schulze \cite{S10a,S10b} and Schulze and Whiteley \cite{SW10} studied the
question of when ``incidental'' symmetry induces non-generic behaviors in finite frameworks, which is
a different setting than the ``forced'' symmetry we consider here and in \cite{MT10}.
Ross, Schulze, and Whiteley \cite{RSW10} have studied the
present problems, but they do not give any combinatorial characterizations.  Borcea and Streinu
\cite{BS11} have proposed a kind of ``doubly generic'' periodic rigidity, where the
combinatorial model does not include the colors on the quotient graph.

\subsection{Acknowledgements} We thank Igor Rivin for encouraging us to take on this project and
many productive discussions on the topic.  This work is part of a larger effort to understand the
rigidity and flexibility of hypothetical zeolites, which is supported by CDI-I grant DMR 0835586 to
Rivin and M. M. J. Treacy.  LT's final preparation of this paper was funded by the European
Research Council under the European Union's Seventh Framework Programme (FP7/2007-2013) /
ERC grant agreement no 247029.

\chapter{Groups}\chaplab{groups}
\section{Crystallographic group preliminaries}\seclab{crystal-prelim}
In this section, we review some basic facts about crystallographic groups generated by translations and
rotations.

\subsection{Facts about the Euclidean group}
The Euclidean isometry group $\Euc(d)$ in any dimension
$d$ admits the following short exact sequence:
\[
1 \to \R^d \to \Euc(d) \to O(d) \to 1
\]
where $O(d)$ is the orthogonal group. The subgroup $\R^d < \Euc(d)$ is the subgroup of \emph{translations} and $\Euc(d) \to O(d)$
is the map that associates to an isometry $\psi$ its derivative at the origin $D\psi_0$.
This short exact sequence splits, since $O(d)$ is naturally isomorphic
to the subgroup of $\Euc(d)$ consisting of isometries fixing the origin.

Consequently, $\Euc(d)$ is isomorphic to the semidirect product $\R^d \rtimes O(d)$ with group operation:
\[
(v, r) \cdot (v', r') = (v + r \cdot v', r r')
\]
Since our setting is $2$-dimensional, from now on, we are interested in $\Euc(2)$.  In the two dimensional case,
we have the following simple lemma, which we state without proof.
\begin{lemma} \lemlab{TorR}
Any nontrivial orientation-preserving isometry of the Euclidean plane is either a rotation around
a point or a translation.
\end{lemma}
Thus, when we refer to orientation-preserving elements of $\Euc(2)$ we call them simply \emph{``rotations''}
or \emph{``translations''}.  We denote the counterclockwise rotation around the origin through angle $2\pi/k$
by $R_k$.

\subsection{Crystallographic groups}
A {\it $2$-dimensional crystallographic group} $\Gamma$ is a group admitting a discrete cocompact
faithful representation $\Gamma \to \Euc(2)$.  We will denote by $\Phi$ discrete faithful representations
of $\Gamma$.  In this paper we are interested in the case where all the group elements are represented
by rotations and translations (i.e., we disallow reflections and glides).

Bieberbach's Theorems \cite{B11,B12} classify all crystallographic groups,
and there are precisely five $2$-dimensional crystallographic groups containing only translations
and rotations.  The first group which we denote by $\Gamma_1$ is $\Z^2$.  The rest
are all semidirect products of $\Z^2$ with a cyclic group.  Namely, for $k = 2, 3, 4, 6$, we have
$\Gamma_k = \Z^2 \rtimes \Z/k\Z$.  The action on $\Z^2$ by the generator
of $\Z/k\Z$ is given by the following table.

\begin{center}
\begin{tabular}{|c|c|c|c|c|}
\hline  $k$ & $2$ & $3$ & $4$ & $6$ \\
\hline matrix & $\JMat{-1 & 0 \\ 0 & -1}$  & $\JMat{ 0 & -1 \\ 1 & -1} $ & $\JMat{0 & -1 \\ 1 & 0} $
&  $\JMat{0 & -1 \\ 1 & 1} $ \\
\hline
\end{tabular}
\end{center}

\noindent
We define the $\Z^2$ subgroup of $\Gamma_k$ to be the {\em translation subgroup} of $\Gamma_k$ and denote it by $\Trans(\Gamma_k)$.
We denote $\gamma\in \Gamma_k$, $k=2,3,4,6$ as $\gamma=(t,r)$ with $t\in \Z^2$ and $r\in \Z/k\Z$.

\subsection{Remark on groups considered}
Since we are only interested in crystallographic groups of this form, the rest of the paper will consider $\Gamma_k$ only (and not
more general crystallographic groups).  Moreover, since the main objective of this paper is \cite[Theorem A]{MT10} when $k=1$,
we will treat only $k=2,3,4,6$ in what follows.  However, the theory presented here specializes to $\Gamma_1$.

\subsection{Finitely generated subgroups} If $\gamma_1,\ldots, \gamma_t$ are element of $\Gamma_k$, we denote the subgroup
generated by the $\gamma_i$ as $\langle \gamma_1,\ldots, \gamma_t\rangle$.   If $\Gamma^1, \ldots, \Gamma^t$
are a sequence of finitely generated subgroups then $\langle \Gamma^1, \Gamma^2,\ldots, \Gamma^t \rangle$
denotes the subgroup generated by the union of some choice of generators for each $\Gamma^i$.

\section{Representation space}\seclab{repspace}

$\Gamma$-crystallographic frameworks and direction networks are required to be symmetric
with respect to the group $\Gamma$.  However, the representation is allowed to flex.  In this
section, we formalize this flexing.

\subsection{The representation space}
Let $\Gamma$ be a crystallographic group.  We define the \emph{representation space} $\Rep(\Gamma)$ of $\Gamma$ to be
\[
\Rep(\Gamma) = \{ \Phi : \Gamma \to \R^2 \rtimes O(2) \; | \; \Phi \text{ is a discrete faithful representation} \}
\]

\subsection{Motions in representation space}
For our purposes a $1$-parameter family of representations is a continuous motion if it is pointwise continuous.
More precisely, identify $\Euc(2) \cong \R^2 \times O(2)$ as topological spaces.  Suppose
$\Phi_t: \Gamma \to \Euc(2)$ is a family of representations defined for $t \in (-\epsilon, \epsilon)$ for some $\epsilon > 0$.
Then, $\Phi_t$ is a continuous motion through $\Phi_0$
if $\Phi_t(\gamma)$ is a continuous path in $\Euc(2)$ for all $\gamma \in \Gamma$.

\subsection{Coordinates for representations}
We now show how to give convenient coordinates for the representation space
for each $\Gamma_k$ for $k=2,3,4,6$; by the classification of
$2$-dimensional crystallographic groups, these are the only cases we need to check.
This next lemma follows readily from Bieberbach's Theorems, but we give a proof in in
\secref{repspace-proof} for completeness.
\begin{lemma}\lemlab{repspace}
The representation spaces of each of the $\Gamma_k$ can be given coordinates as follows:
\begin{itemize}
\item $\Rep(\Gamma_2)\cong \{v_1,v_2, w \in \R^2 : \text{$v_1$ and $v_2$ are linearly independent}\}$
\item $\Rep(\Gamma_k) \cong \{ v_1, w, \varepsilon \; | \; v_1 \neq 0, \varepsilon = \pm 1, v_1, w \in \R^2 \}$ for $k=3,4,6$
\end{itemize}
\end{lemma}
The vectors specify the ``$\R^2$-part'' of the image of a generator in $\Euc(2) \cong \R^2 \rtimes O(2)$.  The $v_i$ will be
the $\R^2$-part of translational generators, and $w$ the $\R^2$-part of a rotational generator.  The vector $w$ determines
the rotation center, but is \emph{not} the rotation center itself.

\subsection{Coordinates for finite-order rotations}
The following lemma describes the coordinates of an order $k$ rotation in $\Euc(2)$, and it makes the meaning
of the vector $w$ appearing in the statement of \lemref{repspace} precise: it determines how an order $k$ rotation
acts on the origin.
\begin{lemma}\lemlab{orderk}
Let $\psi$ be an orientation-preserving element of $\Euc(2)$.  Then $\psi$ has order
$k=2,3,4,6$ if and only if it is of the form $(w,R_k^{\pm 1})$, where $R_k$ is the order $k$
counterclockwise rotation through angle $2\pi/k$.
\end{lemma}
\begin{proof}
If $\psi$ has the required form, then $\psi^k$ is $(w+R_k\cdot w+ \cdots + R_k^{\pm (k-1)}\cdot w,R_k^{\pm k})$.
The first coordinate corresponds to walking along the boundary of a regular $k$-gon, so it is the identity, and the second
evidently is as well.  On the other hand, if $\psi$ has order $k$ then an arbitrary point is either fixed or its
iterated images under $\psi$ are the vertices of a regular polygon, but not necessarily visited in cyclic order.  More
specifically, a rotation though angle $j\frac{2\pi}{k}$ has order $k$ if and only if $j$ has order $k$ in $\Z/k\Z$.
For $k=2,3,4,6$, however, $1$ and $-1$ are the only such $j$.
\end{proof}

\subsection{Generators for $\Gamma_k$}
We also need a description of the generating sets for each of the $\Gamma_k$, which follows
from their descriptions as semi-direct products of $\Z^2\rtimes \Z/k\Z$.
\begin{lemma}\lemlab{gensets}
The following are generating sets for each of the $\Gamma_k$:
\begin{itemize}
\item $\Gamma_2$ is generated by the set $\{((1,0),0), ((0,1),0), ((0,0), 1)\}$.
\item $\Gamma_k$ is generated by the set $\{((1,0),0), ((0,0), 1)\}$ for $k=3,4,6$.
\end{itemize}
\end{lemma}
For convenience, we set the notation $r_k=((0,0),1)$, $t_1=((1,0),0)$, and $t_2=((0,1),0)$.
We now have the pieces in place to prove \lemref{repspace}.

\subsection{Proof of \lemref{repspace}}\seclab{repspace-proof}
We let $\Phi \in \Rep(\Gamma_k)$ be a discrete, faithful representation.  Thus $\Phi$ is determined by the images of the generators,
so \lemref{gensets} tells us we need only to check $t_1$, $t_2$, and $r_k$.

The generators $t_i$ must always be mapped to translations: since they are infinite order and $\Phi$ is faithful,
the only other possibility is an infinite order rotation.  This would contradict $\Phi$ being discrete.  Thus:
\begin{itemize}
\item For $k= 2$, $t_1$ and $t_2$ are mapped to translations $(v_1, \Id)$ and $(v_2, \Id)$.
\item For $k=3,4,6$, $t_1$ is mapped to a translation $(v_1, \Id)$.
\end{itemize}
Moreover, faithfulness and discreteness force:
\begin{itemize}
\item All the images $v_i$ to be non-zero.
\item The images $v_1$ and $v_2$ to be linearly independent for $k=1,2$.
\end{itemize}
By \lemref{orderk} we must have
$\Phi(r_k) = (w, R_k^{\varepsilon})$ for some $w \in \R^2$ and $\varepsilon \in \{-1, 1\}$.  Since $R_2$ is order $2$,
we have $\Phi(r_2) = (w, R_2)$ and $\varepsilon$ is unnecessary for $\Gamma_2$.

In the other direction, given the data described in the statement of the lemma, we simply define
$\Phi(t_i)$ and $\Phi(r_k)$ as above. When $k = 3, 4, 6$, we set $\Phi(t_2) = (R_k^{\varepsilon} v_1, \Id)$.
For arbitrary elements of $\Gamma$, we define $\Phi( (m_1, m_2), m_3) = \Phi(t_1)^{m_1} \Phi(t_2)^{m_2} \Phi(r_k)^{m_3}$.
It is straightforward to check $\Phi$ as defined is a homomorphism and is discrete and faithful. \eop

\subsection{Degenerate representations}
When we are dealing with  ``collapsed realizations'' of direction networks in \chapref{direction-networks}, we
will need to work with certain degenerate representations of $\Gamma_k$.  The space
$$\overline{\Rep}(\Gamma_k)$$
is defined to be representations of $\Gamma_k$ where we allow the $v_i$ to be any vectors.
Topologically this is the closure of  $\Rep(\Gamma_k)$ in the space of all (not necessarily discrete or faithful) representations
$\Gamma_k \to \Euc(2)$.

\subsection{Rotations and translations in crystallographic groups}
As we have defined them, $2$-dimensional crystallographic groups are abstract groups admitting a discrete
faithful representation to $\Euc(2)$.  However, as we
saw in the proof of \lemref{repspace}, all group elements in $\Trans(\Gamma_k)$ must be mapped to translations,
and all group elements outside $\Trans(\Gamma_k)$ must be mapped to rotations.  Consequently, we will henceforth call
elements of $\Trans(\Gamma)$ \emph{``translations''} and elements outside of $\Trans(\Gamma_k)$ \emph{``rotations''}
(even though technically they are elements of an abstract group).

\section{Subgroup structure}\seclab{subgroups}
This short section contains some useful structural lemmas about subgroups of $\Gamma_k$.

\subsection{The translation subgroup}
For a subgroup $\Gamma' < \Gamma_k$, we define its \emph{translation subgroup} $\Trans(\Gamma')$ to be $\Gamma' \cap \Trans(\Gamma_k)$.
(Recall that $\Trans(\Gamma_k)$ is the subgroup $\Z^2$ coming from the semidirect product decomposition of $\Gamma_k$.)

\subsection{Facts about subgroups}
With all the definitions in place, we state several lemmas about subgroups of $\Gamma_k$ that we need later.
\begin{lemma} \label{lemma:subgrpgen}
Let $\Gamma' < \Gamma_k$ be a subgroup of $\Gamma_k$, and
suppose $\Gamma' \neq \Trans(\Gamma')$.  Then $\Gamma'$ is generated by one rotation and $\Trans(\Gamma')$.
\end{lemma}
\begin{proof}
We need only observe that $\Gamma_k/\Trans(\Gamma_k)$ is finite cyclic
and contains $\Gamma'_k/\Trans(\Gamma'_k)$ as a subgroup.
\end{proof}

This next lemma is straightforward, but useful.  We omit the proof.
\begin{lemma} \label{lemma:rotscomm}
Let $r_1, r_2 \in \Gamma_k$ be rotations.  Then $\langle r_1, r_2 \rangle$ is a finite cyclic subgroup consisting of
rotations if and only if some nontrivial powers $r_1^p$ and $r_2^q$ commute.
\end{lemma}

\begin{lemma} \label{lemma:Gam2nice}
Let $r' \in \Gamma_2$ be a rotation and $\Gamma' < \Trans(\Gamma_2)$ a subgroup of the translation
subgroup of $\Gamma_2$.
Then $\Trans( \langle r', \Gamma' \rangle) = \Gamma'$; i.e., after adding the rotation $r'$,
the translation subgroup of the group generated by $r'$ and $\Gamma'$ is again $\Gamma'$.
\end{lemma}
\begin{proof}
All translation subgroups of $\Gamma_2$ are normal, and so the set
$\{ gh \;\; | \;\; g = r', \Id \;\; h \in \Gamma'\}$ is a subgroup and is equal to $\langle r', \Gamma' \rangle$.
Clearly, the only translations are those elements of $\Gamma'$.
\end{proof}

\section{The restricted representation space and its dimension}\seclab{radical}
To define our degree of freedom heuristics in \chapref{graphs},
we need to understand how representations of $\Gamma_k$ restrict to subgroups $\Gamma' < \Gamma_k$, or
equivalently, which representations of $\Gamma'$ extend to $\Gamma_k$.  For $\Gamma'<\Gamma_k$,
the \emph{restricted representation space} of $\Gamma'$ is the image of the restriction map
from $\Gamma_k$ to $\Gamma$, i.e.,
\[
\Rep_{\Gamma_k}(\Gamma') = \{ \Phi: \Gamma' \to \Euc(2) \; | \; \Phi
\text{ extends to a discrete faithful representation of $\Gamma_k$} \}
\]
We define the notation $\rep_{\Gamma_k}(\Gamma') := \dim \Rep_{\Gamma_k}(\Gamma')$, since the dimension of $\Rep_{\Gamma_k}(\Gamma')$ is
an important quantity in what follows.  Since it will be useful later, we also define:
\[
T(\Gamma') := \left\{ \begin{array}{rl} 0 & \text{$\Gamma'$ has a rotation} \\ 2 & \text{$\Gamma'$ has no rotations} \end{array} \right.
\]
Equivalently, we may define $T(\Gamma')$ as the dimension of the space of translations commuting
with $\Gamma'$.  In \secref{crystal-collapse}, we will show that $T(\Gamma')$ is the dimension of the space
of collapsed solutions of a direction network for a connected graph $G'$ satisfying $\rho(\pi_1(G')) = \Gamma'$.

The dimension $\rep_{\Gamma_k}(\Gamma')$
of the restricted representation space
$\Rep_{\Gamma_k}(\Gamma')$ is an important quantity for counting the degrees of freedom
in a direction network.  We now develop some properties of $\rep_{\Gamma_k}(\cdot)$ and
how it changes as new generators are added to a finitely generated subgroup.

\subsection{Translation subgroups}
For translation subgroups $\Gamma'< \Gamma$, we are interested in the dimension of
$\Rep_{\Gamma}(\Gamma')$.  The following lemma gives a characterization for translation subgroups
in terms of the rank of $\Gamma'$.
\begin{lemma} \label{lemma:trrep}
Let $\Gamma' < \Gamma_k$ be a {\em nontrivial} subgroup of translations.
\begin{itemize}
\item If $k = 3, 4, 6$, then $\rep_{\Gamma_k}(\Gamma') = 2$.
\item  If $k =1, 2$, 	then $\rep_{\Gamma_k}(\Gamma') = 2 \cdot r$, where $r$ is the
minimal number of generators of $\Gamma'$.
\end{itemize}
In particular, $\rep_{\Gamma_k}(\Gamma')$ is even.
\end{lemma}
\begin{proof}
Suppose $k = 3, 4,$ or $6$.  By Lemma \ref{lemma:repspace}, the space of representations of $\Gamma_k$
is $4$-dimensional and is uniquely determined by the
parameters $v_1, w$ and the sign $\varepsilon$.  The group $\Trans(\Gamma_k) \cong \Z^2$ is generated by $t_1$
and $r_k t_1 r_k^{-1}$, and so any $\gamma \in \Trans(\Gamma_k)$ can be written uniquely as
$t_1^{m_1} r_k t_2^{m_2} r_k^{-1}$ for integers $m_1, m_2$.
Thus, since $\Phi(\gamma)$ is a translation,
\begin{align*}
\Phi(\gamma) =  & \Phi(t_1)^{m_1} \Phi(r_k) \Phi(t_2)^{m_2} \Phi(r_k)^{-1} \\
& = \left(m_1 v_1, \Id\right)  \left(w, R_k^{\varepsilon}\right) \left(m_2 v_1, \Id \right)\left(w, R_k^{-\varepsilon} \right) \\
& = \left(m_1 v_1 + m_2 R_k^\varepsilon v_1, \Id\right)
\end{align*}
Hence, regardless of $w$, any representation with the same $v_1, \varepsilon$ parameters restricts to the
same representation on $\Trans(\Gamma_k)$ and thus also on $\Gamma'$.

Suppose $k = 1, 2$.  In this case by the proof of Lemma \ref{lemma:repspace}, any discrete faithful
representation $\Trans(\Gamma_k) \to \Euc(2)$ extends to a discrete faithful representation of $\Gamma_k$.
Since $\Trans(\Gamma_k) \cong \Z^2$, any discrete faithful representation of its subgroups to $\R^n$ extends to
$\Trans(\Gamma_k)$ and hence $\Gamma_k$.  Hence $\rep_{\Gamma_k}(\Gamma')$ is equal to the dimension of
representations $\Gamma' \to \R^2$ which is twice the size of a minimal generating set of $\Gamma'$, since it
is a free abelian group.
\end{proof}

\subsection{The radical of a subgroup}
In \secref{group-matroid}, we will introduce a matroid on the elements of a crystallographic group.  To prove
the required properties, we need to know how the translation subgroup
$\Trans(\cdot)$ changes as generators are added to a
subgroup of $\Gamma_k$.  The {\it radical} of a subgroup $\Gamma' < \Gamma$, which we now define and
develop, is the key tool for doing this.

We define the \emph{radical, $\Rad(\Gamma')$, of $\Gamma'$} to be the largest subgroup containing $\Gamma'$ such that
\begin{eqnarray}
\rep_\Gamma(\Trans(\Gamma')) = \rep_\Gamma(\Trans(\Rad(\Gamma'))) & \text{and} & T(\Gamma') = T(\Rad(\Gamma'))
\end{eqnarray}
It is called the radical since it contains at least all the roots of nontrivial elements of
$\Gamma'$, by \lemref{rad-roots} below.

\subsection{Properties of the radical}
The following sequence of lemmas enumerates the properties of the radical that we will use in the sequel.

\begin{lemma}\lemlab{rad-welldefined}
Let $\Gamma'<\Gamma_k$ be a subgroup of $\Gamma_k$.  Then the
radical $\Rad(\Gamma')$ is well-defined.
\end{lemma}
\begin{proof}
First let $k=2$.  There are two cases.
If $\Gamma'$ contains only translations, we set
\[
\Rad(\Gamma') = \{t\in \Trans(\Gamma_2) : \text{$t^i\in \Gamma'$ for some power $i$ of $t$} \}
\]
Any subgroup $\Gamma''<\Gamma_2$ containing $\Gamma'$ with $T(\Gamma')=T(\Gamma'')$ and
$\rep_\Gamma(\Gamma')=\rep(\Gamma'')$ must be a	translation group of the same rank as $\Gamma'$ and by
definition of $\Rad(\Gamma')$ is the largest such subgroup.  Also, note that $\Rad(\Gamma')$ and $\Gamma'$
necessarily have the same rank.

Otherwise $\Gamma'$ contains a rotation $r'$.  In this case, we set
\[
\Rad(\Gamma') = \langle r', \Rad(\Trans(\Gamma'))\rangle
\]
By \lemref{Gam2nice}, for $\Rad(\Gamma')$ defined this way, the translation subgroup
$\Trans(\Rad(\Gamma'))$ is just $\Rad(\Trans(\Gamma'))$
which by the previous paragraph is the largest translation subgroup containing $\Trans(\Gamma')$ and having the same rank.
Any subgroup $\Gamma''<\Gamma_2$ containing $\Gamma$ must
be of the form $\Gamma'' = \langle r',\Trans(\Gamma'')\rangle$ with $\Trans(\Gamma') < \Trans(\Gamma'')$.
If additionally $\rep_\Gamma(\Trans(\Gamma')) = \rep_\Gamma(\Trans(\Gamma''))$, then
$\Trans(\Gamma'') < \Rad(\Trans(\Gamma'))$ and $\Gamma'' < \Rad(\Gamma')$.

Now we suppose that $k=3,4,6$.  There are four possibilities for $\Gamma'$:
\begin{itemize}
\item If $\Gamma'$ is trivial, then we define $\Rad(\Gamma')$ to be trivial, and this choice is clearly canonical.
\item If $\Gamma'$ is a cyclic group of rotations, then \lemref{rotscomm} guarantees that there is a unique
largest cyclic subgroup containing it, and we define this to be $\Rad(\Gamma')$.
\item If $\Gamma'$ has only translations, then we define $\Rad(\Gamma') = \Trans(\Gamma_k)$.
\item If $\Gamma'$ has translations and rotations, then some power
of both standard generators for $\Gamma_k$ from \lemref{gensets} lies in $\Gamma'$.  It follows that
that defining $\Rad(\Gamma')=\Gamma_k$	is the canonical choice.
\end{itemize}
\end{proof}
The construction used to prove \lemref{rad-welldefined} gives us the following structural description
of the radical.
\begin{prop}\proplab{rad-struct}
Let $\Gamma'<\Gamma_k$ be a subgroup of $\Gamma_k$ for $k=2,3,4,6$. Then if $k=2$,
\begin{itemize}
\item If $\Gamma'$ is a translation subgroup, then $\Rad(\Gamma')$ is the subgroup of translations with
a non-trivial power in $\Gamma'$.
\item If $\Gamma'$ has translations and rotations, then $\Rad(\Gamma') = \langle r', \Rad(\Trans(\Gamma'))\rangle$.
\end{itemize}
If $k=2,3,4,6$, then there are four possibilities for the radical:
\begin{itemize}
\item If $\Gamma'$ is trivial, the radical is trivial.
\item If $\Gamma'$ is cyclic, the radical is a cyclic subgroup of order $k$.
\item If $\Gamma'$ is a translation subgroup, the radical is the translation subgroup of $\Gamma_k$.
\item If $\Gamma'$ has translations and rotations, the radical is all of $\Gamma_k$.
\end{itemize}
\end{prop}
Another immediate corollary of \lemref{rad-welldefined} is that we may ``pass to radicals'' if we are interested
in $\rep_{\Gamma_k}(\cdot)$ and $T(\cdot)$.
\begin{prop}\proplab{rad-rep-T-invariant}
Let $\Gamma'$ be a subgroup of $\Gamma_k$.  Then
\begin{eqnarray*}
\rep_{\Gamma_k}(\Gamma') & = & \rep_{\Gamma_k}(\Rad(\Gamma')) \\
T(\Gamma') & = & T(\Rad(\Gamma'))
\end{eqnarray*}
\end{prop}

The radical also has a monotonicity property.
\begin{lemma}\lemlab{rad-monotone}
Let $\Gamma'<\Gamma_k$ be a finitely-generated subgroup of $\Gamma_k$,
and let $\Gamma''< \Gamma'$ be a subgroup of $\Gamma'$.  Then $\Rad(\Gamma'') < \Rad(\Gamma')$.
\end{lemma}
\begin{proof}
Pick a generating set of $\Gamma''$ that extends to a generating set of $\Gamma'$.
Analyzing the cases in \propref{rad-struct} shows that
the radical cannot become smaller after adding generators.
\end{proof}

As mentioned above, this next lemma provides some justification for the terminology ``radical''.
\begin{lemma}\lemlab{rad-roots}
Let $\Gamma'<\Gamma_k$ be a subgroup of $\Gamma_k$.  If some power $\gamma^i$ of $\gamma$ is not the identity and
$\gamma^i\in \Gamma'$, then $\gamma\in \Rad(\Gamma')$.
\end{lemma}
\begin{proof}
If $\gamma$ is a translation, this is clear by \propref{rad-struct}.  Now let $\gamma$ be a
rotation with $\Id \neq \gamma^\ell\in \Gamma'$ and $\ell \neq 1$.  Together these hypotheses imply that $k$ is $3,4$ or $6$,
and so we see that $\Rad(\Gamma')$ is either all of $\Gamma_k$ or finite and cyclic or order $k$.
In the first case, we are clearly done, and the second follows from \lemref{rotscomm} and the
fact that $\Gamma'$ itself is finite and cyclic.
\end{proof}

\begin{lemma}\lemlab{rad-conj}
Let $\Gamma'<\Gamma_k$ be a translation subgroup of $\Gamma_k$, and let $\gamma\in \Gamma_k$.  Then
$\Rad(\gamma\Gamma'\gamma^{-1}) = \Rad(\Gamma')$; i.e., the radical of translation subgroups is
fixed under conjugation.
\end{lemma}
\begin{proof}
For $k=2$ this follows from the fact that all translation subgroups are normal.  For $k=3,4,6$ it
is immediate from the definitions.
\end{proof}

\begin{lemma}\lemlab{rad-push-trans}
Let $\Gamma'<\Gamma_k$ be a subgroup of $\Gamma_k$, and let $\Gamma''<\Trans(\Gamma_k)$ be a translation subgroup
of $\Gamma_k$.  Then $\Rad( \langle \Trans(\Gamma'), \Gamma''\rangle) = \Rad(\Trans(\langle \Gamma', \Gamma'' \rangle))$.
\end{lemma}
\begin{proof}
The proof is in cases based on $k$.  For $k=3,4,6$, either $\Gamma''$ is trivial or both sides of the desired equation are
$\Trans(\Gamma_k)$, by \propref{rad-struct}.  Either way, the lemma follows at once.

Now suppose that $k=2$.  If $\Gamma'$ is a translation subgroup, then the lemma follows immediately.  Otherwise,
we know that $\Gamma'$ is generated by a rotation $r'$ and the translation subgroup $\Trans(\Gamma')$.  Applying
\lemref{Gam2nice}, we see that
\[
\Trans(\langle\Gamma', \Gamma'' \rangle) =
\Trans(\langle r', \Trans(\Gamma'), \Gamma'' \rangle) = \langle \Trans(\Gamma'), \Gamma''\rangle
\]
from which the lemma follows.
\end{proof}

\subsection{The quantity $\rep_{\Gamma_k}(\Gamma') - T(\Gamma')$}
The following statement plays a key role in the matroidal construction of \secref{group-matroid}.

\begin{prop}\proplab{rad-rep-minus-T}
Let $\Gamma'<\Gamma_k$ be a subgroup of $\Gamma_k$, and let $\gamma\in \Gamma_k$ be an element of $\Gamma_k$.  Then,
\[
\rep_{\Gamma_k}(\Trans(\langle \Gamma', \gamma \rangle)) -
T(\langle \Gamma', \gamma \rangle) - (\rep_{\Gamma_k}(\Trans(\Gamma')) - T(\Gamma'))
=
\left\{ \begin{array}{rl} 2 & \text{ if } \gamma \notin \Rad(\Gamma') \\
0 & \text{ otherwise} \end{array} \right.
\]
i.e., the quantity $\rep_{\Gamma_k}(\cdot)-T(\cdot)$ increases by two after adding
$\gamma$ to $\Gamma'$ if and only if $\gamma\notin\Rad(\Gamma')$ and otherwise the
increase is zero.
\end{prop}
\begin{proof}
If $\gamma \in \Rad(\Gamma')$, this follows at once from the definition, since the
quantity $\rep_{\Gamma}(\Gamma')-T(\Gamma')$ depends only on the radical.

Now suppose that $\gamma \notin \Rad(\Gamma')$.  Since the radical is defined in terms of $\rep_{\Gamma_k}(\cdot)$
and $T(\cdot)$, \lemref{rad-monotone} implies that
at least one of $\rep_{\Gamma_k}(\cdot)$ or $-T(\cdot)$ increases, it is easy to see from the definition
that either type of increase is by at least $2$.  We will show that the increase is at most $2$, from which
the lemma follows.  The rest of the proof is in three cases, depending on $k$.

Now we let $k=3, 4, 6$.  The only way for the increase to be larger than $2$ is for
$\Gamma'$ to be trivial and $\Rad(\langle \gamma \rangle) = \Gamma_k$.  This is clearly
impossible given the description from the proof of \lemref{rad-welldefined}.

To finish, we address the case $k = 2$.
Suppose $\gamma$ is a translation.  Then $T(\langle \gamma, \Gamma' \rangle) = T(\langle \Gamma'\rangle)$, since
adding $\gamma$ as a generator doesn't give us a new rotation if one wasn't already present in $\Gamma'$.
Lemmas \ref{lemma:subgrpgen} and \ref{lemma:Gam2nice} imply that
$\Trans(\langle \gamma, \Gamma' \rangle) = \langle \gamma, \Trans(\Gamma') \rangle$.
Hence, the rank of the translation subgroup increases by at most $1$, and so, by \lemref{trrep},
$\rep_{\Gamma_k}(\cdot)$ by at most $2$.

Now suppose that $\gamma$ is a rotation.  If $\Gamma'$ has no rotations,
then Lemma \ref{lemma:Gam2nice} implies $\Trans(\langle \gamma, \Gamma' \rangle) = \Gamma'$,
and so $T(\cdot)$ decreases and $\rep_{\Gamma_2}(\cdot)$ is unchanged.  If
$\Gamma'$ has rotations, then $\Gamma' = \langle r', \Trans(\Gamma')\rangle$ for some rotation $r'\in \Gamma'$.
Since $k=2$, $r' \gamma$ is a translation and so
\[\Trans(\langle \gamma, \Gamma' \rangle) = \Trans(\langle \gamma, r', \Trans(\Gamma') \rangle) =
\Trans(\langle r', r' \gamma, \Trans(\Gamma') \rangle) = \langle r' \gamma, \Trans(\Gamma') \rangle\]
Thus, in this case, the the number of generators of the translation subgroup increases by at most one and
$T(\cdot)$ is unchanged.  By \lemref{trrep}, the proof is complete.
\end{proof}

\section{Teichmüller space and the centralizer}\seclab{teich}
The representation spaces defined in the previous two sections are closely related to the degrees
of freedom in the crystallographic direction networks we study in the sequel.  In this section, we
develop the \emph{Teichmüller space} and \emph{centralizer}, which play the same role for frameworks.

\subsection{Teichmüller space}
The \emph{Teichm\"uller space} of $\Gamma_k$ is defined to be
the space of discrete faithful representations
modulo conjugation by $\Euc(2)$, i.e. $\Teich(\Gamma_k) = \Rep(\Gamma_k)/\Euc(2)$.  For a subgroup $\Gamma' < \Gamma$,
we define its \emph{restricted Teichm\"uller space} to
\[
\Teich_{\Gamma_k}(\Gamma') = \Rep_{\Gamma_k}(\Gamma')/\Euc(2)
\]
Correspondingly, we define $\teich_{\Gamma_k}(\Gamma') = \dim(\Teich_\Gamma(\Gamma'))$.

\subsection{The centralizer}
For a subgroup $\Gamma' \leq \Gamma_k$ and a discrete faithful representation $\Phi: \Gamma \to \Euc(2)$,
the \emph{centralizer of $\Phi(\Gamma')$} which we denote $\Cent_{\Euc(2)}(\Phi(\Gamma'))$ is the set of
elements commuting with all elements in $\Phi(\Gamma')$.
We define $\cent_{\Gamma_k}(\Gamma')$ to be the dimension of the centralizer
$\Cent_{\Euc(2)}(\Phi(\Gamma'))$.  The quantity $\cent_{\Gamma_k}(\Gamma')$ is independent of $\Phi$,
and we can compute it.  Since we don't depend on \lemref{centtable} or \propref{reptoteich} for any of our main results,
we skip the proofs in the interest of space.
\begin{lemma} \lemlab{centtable}
Let notation be as above.  The dimension $\cent_{\Gamma_k}(\Gamma')$
of $\Cent_{\Euc(2)}(\Phi(\Gamma'))$ is independent of the representation $\Phi$.
Furthermore, $\cent(\Gamma') \geq T(\Gamma')$, and in particular,
\[ \cent(\Gamma') =
\begin{cases}
0 & \text{ if $\Gamma'$ contains rotations and translations} \\
1 & \text{ if  $\Gamma'$ contains only rotations} \\
2 & \text{ if  $\Gamma'$ contains only translations} \\
3 & \text{ if  $\Gamma'$ is trivial}
\end{cases}
\]
\end{lemma}
As a corollary, we get the following proposition relating $\rep_{\Gamma_k}(\cdot)$ and
$T(\cdot)$ to $\teich_{\Gamma_k}(\cdot)$ and $\cent(\cdot)$.
\begin{prop} \proplab{reptoteich}
Let $\Gamma' < \Gamma_k$.  Then:
\begin{itemize}
\item[\textbf{(A)}] If $\Gamma'$ contains a translation, then $T(\Gamma') = \cent_{\Gamma_k}(\Gamma')$.  Otherwise,
$T(\Gamma') = \cent_{\Gamma_k}(\Gamma') - 1$.
\item[\textbf{(B)}] If $\Gamma'$ is a non-trivial translation subgroup, then
$\teich_{\Gamma_k}(\Gamma') = \rep_{\Gamma_k}(\Gamma') -1$.
\item[\textbf{(C)}] If $\Gamma'$ is trivial, then $\teich_\Gamma(\Gamma') = \rep_\Gamma(\Gamma') = 0$.
\item[\textbf{(D)}] For any $\Gamma'<\Gamma_k$, $\rep_{\Gamma_k}(\Trans(\Gamma')) - T(\Gamma') =
\teich_{\Gamma_k}(\Trans(\Gamma')) - \cent_{\Gamma_k}(\Gamma')$.
\end{itemize}
\end{prop}

\section{Matroid preliminaries}\seclab{matroid-prelim}
The concepts of matroids and their linear representability play a key role in the
results of this paper.  In \secref{group-matroid}, we will define a matroidal structure
on, essentially, $\Gamma_k$.  In this section, we review the parts of matroid theory
we will need in the sequel.

\subsection{Matroids given by bases}\seclab{base-axioms}
A \emph{matroid} $M$ is a combinatorial structure that captures some essential features of linear
dependence and independence over a \emph{ground set} $E$.  Matroids have many equivalent definitions
(see, e.g., the monograph \cite{O06}),
but a convenient one for graph-theoretic matroids is by the \emph{bases} $\mathcal{B}(M)\subset 2^E$,
which must satisfy the following axioms:
\begin{itemize}
\item [\textbf{Non-triviality}] $\mathcal{B}(M)\neq \emptyset$.
\item [\textbf{Equal size}] If $A$ and $B$ are in $\mathcal{B}(M)$, then $|A| = |B|$.
\item [\textbf{Base exchange}] If $A$ and $B$ are in $\mathcal{B}(M)$, then there
are elements $a\in A\setminus B$ and $b\in B\setminus A$
such that $A+b-a\in \mathcal{B}(M)$.
\end{itemize}
The size of bases is defined to be the \emph{rank} of the matroid.

For readers new to matroids, we note that basis exchange corresponds to Steinitz exchange between
bases of a finite-dimensional vector space.  The canonical example of a matroid has the ground set the
edges of the complete graph $K_n$ on $n$ vertices and the bases the
spanning trees; this is usually called the \emph{graphic matroid} in the literature.
All the  axioms are readily verified in this case.

\subsection{Matroids given by rank functions}\seclab{rank-axioms}
Let $E$ be a set and $f$ a non-negative, integer-valued function defined on subsets of $E$.  We define $f$ to be \emph{monotone},
if for all $A\subset B\subset E$, $f(A)\le f(B)$.  We define $f$ to be \emph{submodular}, if for any subsets $A$ and $B$ of $E$:
\[
f(A\cup B) + f(A\cap B) \le f(A) + f(B)
\]
which is called the \emph{submodular inequality}.  Submodular functions are an important class in optimization theory, since they
capture a kind of ``combinatorial concavity''.  A more ``local'' characterization of submodularity, which will be easier
for us to work with is along these lines.  Let $A\subset B\subset E$, and let $e\in E\setminus B$.  Then, $f$ is
submodular if and only if for all such $A$, $B$ and $x$
\begin{equation}\label{local-submodular}
f(A\cup \{x\}) - f(A) \ge f(B\cup \{x\}) - f(B)
\end{equation}
An alternative characterization of a matroid $M$ on a ground set $E$
is by its \emph{rank function}, which we denote $f_M$.  An integer-valued set function $f_M$ on $E$
is the \emph{rank function of a matroid $M$} if:
\begin{itemize}
\item [\textbf{Non-negativity}]$f_M$ is non-negative and zero on $\emptyset$
\item [\textbf{Monotonicity}]$f_M$ is monotone
\item [\textbf{Submodularity}] $f_M$ is submodular
\item [\textbf{Normalization}] For all $A\subset E$ and $e\in E\setminus A$, $f_M$ increases by zero or one when $e$ is added to $A$.
\end{itemize}
For example, the rank function of the graphic matroid is given as follows: the rank of a subset $E'$ of
edges of $K_n$ spanning $n'$ vertices and $c'$ connected components is simply $n'-c'$.

\subsection{Connection between rank functions and bases}
The conversion between the characterization by rank function and bases is as follows.
Given the rank function $f_M$ of a matroid $M$, the bases are
\[
\mathcal{B}(M) =
\left\{
A\subset E : \text{$f_M(A) = |A|$ and $f_M(A)=f_M(E)$}
\right\}
\]
Given the bases $\mathcal{B}(M)$, the rank function is given by
\[
f_M(A) = \max_{B\in \mathcal{B}(M)} |A\cap B|
\]

\subsection{Infinite ground sets}
Readers familiar with matroids will notice that we have not required the ground set to be finite.
This is intentional: since we will be working with ground sets involving $\Gamma_k$, which is infinite, our ground
set will be as well. Because all the matroids we deal with are finite rank (depending on $n$ and $\Gamma$),
all the required theory goes through.

\subsection{Matroids from submodular functions}
A fundamental theorem of matroid theory, due to Edmonds and Rota \cite{ER66}, and extended to
the case where $E$ may be infinite by Pym and Perfect \cite{PP70}, in matroid theory gives a
recipe for moving from submodular functions to matroids.
\begin{theorem}[\edmonds][\cite{ER66}]\theolab{edmonds1}
Let $E$ be a set and $f$ be a non-negative, monotone, finite, integer-valued function on subsets of $E$.  Then
the collection of subsets
\[
\left\{
A\subset E : \text{$f(A) = |A|$ and for all subsets $A'\subset A$, $|A'|\le f(A')$}
\right\}
\]
gives the bases of a matroid.
\end{theorem}
We define the matroid arising from \theoref{edmonds1} to be $M_f$.

\subsection{Matroid union}
We will use as an essential tool, the following construction.
\begin{theorem}[\edmondss][\cite{ER66}]\theolab{edmonds2}
Let $M_1$ and $M_2$ be matroids on a common ground set $E$, and let $f_1$ and $f_2$ be submodular functions
such that $M_i$ is $M_{f_i}$ as in \theoref{edmonds1}.  Then the matroid $M_{f_1+f_2}$, obtained from
\theoref{edmonds1} has, as its bases, the subsets of $E$
\[
\left\{
A\subset E : \text{$A=A_1\cup A_2$, with $A_1\cap A_2 = \emptyset$ and $A_i\in \mathcal{B}(M_i)$}
\right\}
\]
\end{theorem}

\section{A matroid on crystallographic groups}\seclab{group-matroid}
We now define and study a matroid $M_{\Gamma_k,n}$ for $k=2,3,4,6$.

\subsection{Preview of $\Gamma$-$(1,1)$ graphs and $M_{\Gamma_k,n}$}
In \secref{sparse}, we will relate
$M_{\Gamma_k,n}$ to ``$\Gamma$-$(1,1)$ graphs'', which are defined in \secref{gamma11-def}.  The
results here, roughly speaking, are the group theoretic part of the proof of \propref{gamma11}
in \secref{sparse}.

To briefly motivative the definitions given next, $\Gamma$-$(1,1)$ graphs need
not be connected, and
each connected component is associated with a finitely generated subgroup of $\Gamma_k$.  The ground
set of $M_{\Gamma_k, n}$ and the $A_i$ defined below capture this situation.  The operations of
conjugating and fusing, defined here in Sections \secrefX{Mkn-conj} and \secrefX{Mkn-fuse} will
be interpreted graph theoretically in \secref{sparse}.

\subsection{The ground set} For the definition of the ground set, we fix $\Gamma$ and a natural number $n\ge 1$.
The ground set $E_{\Gamma_k,n}$ is defined to be:
\[
E_{\Gamma_k,n} =
\left\{
(\gamma,i) : 1\le i\le n
\right\}
\]
In other words the ground set is $n$ labeled copies of $\Gamma_k$.

Let $A\subset E_{\Gamma_k,n}$.  We define some notations:
\begin{itemize}
\item $A_i = \{\gamma : (\gamma,i)\in A\}$; i.e., $A_i$ is the group elements from copy $i$ of $\Gamma_k$ in $A$.  Some of the
$A_i$ may be empty and $A_i$ can be a multi-set.  $A$ may equivalently defined by the $A_i$.
\item $\Gamma_{A,i} = \langle \gamma : \gamma\in A_i\rangle$; i.e., the subgroup generated by the elements in $A_i$.
\item $\Trans(A) = \langle \Trans(\Gamma_{A,1}), \Trans(\Gamma_{A,2}), \dots, \Trans(\Gamma_{A,n}), \rangle$; the translation subgroup
generated by the translations in each of the $\Gamma_{A,i}$.
\item $c(A)$ is the number of $A_i$ that are not empty.
\end{itemize}

\subsection{The rank function}
We now define the function $g_1(A)$ for $A\subset E_{\Gamma_k,n}$ to be
\[
g_1(A) = n + \frac{1}{2}\rep_{\Gamma_k}(\Trans(A)) - \frac{1}{2}\sum_{i=1}^n T(\Gamma_{A,i})
\]

The meaning of the terms in $g_1(A)$ are as follows:
\begin{itemize}
\item The second term is a global adjustment for the representation space of the group generated by
the translations in each of the $\Gamma_{A,i}$.  We note that this is \emph{not} the same as
the translation group $\Trans(\langle\gamma : \gamma\in \cup_{i=1}^n A_i)$, which includes
translations arising as products of rotations in different $A_i$.
\item The quantity $n - \frac{1}{2} \sum_{i=1}^n T(\Gamma_{A, i}) = \sum_{i=1}^n (1 - \frac{1}{2} T(\Gamma_{A, i}))$
is a local adjustment based on whether $\Gamma_{A,i}$ contains a rotation:
each term in the latter sum is one if $\Gamma_{A,i}$ contains a rotation and otherwise it contributes
nothing.
\end{itemize}

\subsection{An analogy to uniform linear matroids}
To give some intuition about why the construction above might be matroidal, we observe that
\propref{rad-rep-minus-T}, interpreted in matroidal language gives us:
\begin{prop}\proplab{gammak-matroid}
Let $A$ be a finite subset of $\Gamma_k$ generating a subgroup $\Gamma_A$.  Then the function
\[
r(A) = \frac{1}{2}\left(\rep_{\Gamma_k}(\Gamma_A) - T(\Gamma_A)\right)
\]
is the rank function of a matroid on the ground set $\Gamma_k$.
\end{prop}
The matroid in the conclusion of \propref{gammak-matroid} is a kind of uniform linear matroid, with
$\Gamma_k$ playing the role of a vector space and $r$ the role of dimension of the linear span.  Since
the function $g_1$, defined above, builds on $r$, one might expect that it inherit a matroidal
structure.  We verify this next.

\subsection{$M_{\Gamma_k,n}$ is a matroid}
The following proposition is the main result of \chapref{groups}.
\begin{prop}\proplab{Mgammarank}
The function $g_1$ is the rank function of a matroid $M_{\Gamma_k,n}$.
\end{prop}
The proof depends on  Lemmas \lemrefX{reduce1} and \lemrefX{reduce2} below, so we defer it for the moment.
The strategy is based on the observation that when $n=1$, the ground set is essentially $\Gamma_k$.  In this case,
submodularity and normalization of $g_1$ (the most difficult properties to establish) follow
immediately from \propref{rad-rep-minus-T}.  The motivation of Lemmas \ref{lemma:reduce1} and \ref{lemma:reduce2} is
to reduce, as much as possible, the proof of the general case to $n=1$.

\begin{lemma}\lemlab{reduce1}
Let $A \subset E_{\Gamma_k, n}$, and set $\Gamma_{A, \ell}' = \langle \Gamma_{A, \ell}, \Trans(A) \rangle$.
Then, for all $1 \leq \ell \leq n$,
\begin{itemize}
\item[\textbf{(A)}] $\Rad(\Trans(A))  =  \Rad(\Trans(\Gamma_{A, \ell}'))$
\item[\textbf{(B)}] $T(\Gamma_{A, \ell})   =  T(\Gamma_{A, \ell}') $
\end{itemize}
\end{lemma}
\begin{proof}
The statement \textbf{(A)} is immediate from \lemref{rad-push-trans}. \textbf{(B)} follows from
the fact that $\Trans(A)$ is a translation subgroup of $\Gamma_k$, so $\Gamma_{A, \ell}'$ has a rotation
if and only if $\Gamma_{A,\ell}$ does.
\end{proof}
\begin{lemma}\lemlab{reduce2}
Let $A \subset E_{\Gamma_k, n}$, and set $\Gamma_{A, \ell}' = \langle \Gamma_{A, \ell}, \Trans(A) \rangle$.
If $B = A + (\gamma, \ell)$ and $\Gamma_{B, \ell}' = \langle \Gamma_{B, \ell}, \Trans(B) \rangle$,  Then,
\[
\Gamma_{B, \ell}' = \langle \gamma, \Gamma_{A, \ell}' \rangle
\]
\end{lemma}
\begin{proof}
First we observe that
\[
\Gamma_{B, \ell}' =
\langle \gamma, \Gamma_{A, \ell}, \Trans(B) \rangle \geq
\langle \gamma, \Gamma_{A, \ell}, \Trans(A) \rangle
\]
so to finish the proof we just have to show that
\[
\Trans(B) \leq \langle \gamma, \Gamma_{A, \ell}, \Trans(A) \rangle
\]
Since $\Gamma_{B, i} = \Gamma_{A, i}$ for all $i\neq \ell$, it follows that
\[
\Trans(B) = \langle \Trans(A), \Trans( \langle \gamma, \Gamma_{A, \ell} \rangle) \rangle \leq
\langle \gamma, \Gamma_{A, \ell}, \Trans(A) \rangle
\]
\end{proof}

\subsection{Proof of \propref{Mgammarank}}
We check the rank function axioms (from \secref{rank-axioms}).

\textbf{Non-negativity:} This follows from the fact that $\rep_{\Gamma_k}(\cdot)$ is non-negative, and the
sum of the $\frac{1}{2}T(\cdot)$ terms cannot exceed $n$.

\textbf{Monotonicity:} Immediate from \lemref{rad-monotone}.

\textbf{Normalization:} To prove that $g_1$ is normalized, let $A\subset E_{\Gamma_k,n}$ and $B = A + (\gamma,\ell)$.  Since all the
$T(\Gamma'_{\cdot,i})$ terms cancel except for the ones with $i=\ell$, the increase is given by
\[
g_1(B) - g_1(A) =
\frac{1}{2} \left(
\rep_{\Gamma_k}(\Trans(B)) - \rep_{\Gamma_k}(\Trans(A)) - T(\Gamma_{B, \ell}) + T(\Gamma_{A, \ell})
\right)
\]
Because the r.h.s. is an invariant of the radical by \propref{rad-rep-T-invariant}, we pass to radicals and apply
\lemref{reduce1} to see that the r.h.s. is equal to
\[
\frac{1}{2}
\left(
\rep_{\Gamma_k}(\Trans(\Gamma_{B, \ell}')) - \rep_{\Gamma_k}(\Trans(\Gamma_{A, \ell}'))  -
T(\Gamma_{B, \ell}') + T(\Gamma_{A, \ell}')
\right)
\]
Using \lemref{reduce2} then tells us that this can be simplified further to
\[
\frac{1}{2}
\left(
\rep_{\Gamma_k}(\Trans(\langle \gamma, \Gamma_{A, \ell}'\rangle)) - \rep_{\Gamma_k}(\Trans(\Gamma_{A, \ell}'))  -
T(\langle \gamma, \Gamma_{A, \ell}'\rangle) + T(\Gamma_{A, \ell}')
\right)
\]
at which point \propref{rad-rep-minus-T} applies, and we conclude that the increase is either zero or one.

\textbf{Submodularity:} Inspecting the argument for normalization and using \lemref{rad-monotone} one more time
gives submodularity, since, if $A'\subset A$ and $\gamma\notin\Rad(\Gamma_{A,\ell})$, then
$\gamma \notin \Rad(\Gamma_{A',\ell})$.  This gives us the the submodular inequality \eqref{local-submodular}.
\eop

\subsection{The bases and independent sets}\seclab{group-matroid-sets}
With the rank function of $M_{\Gamma,n}$ determined, we can give a structural characterization of its bases
and independent sets.  Let $A\subset E_{\Gamma_k,n}$.  We define $A$ to be \emph{independent} if
\[
|A| = g_1(A)
\]
If $A$ is independent and, in addition
\[
|A| = c(A) + \frac{1}{2} \rep_{\Gamma_k}(\Trans(\Gamma_k))
\]
we define $A$ to be \emph{tight}.  A (not-necessarily independent) set $A$ with $c(A)$ parts that contains a tight
subset on $c(A)$ is defined to be \emph{spanning}.

We define the classes
\[
\mathcal{B}(M_{\Gamma,n}) =
\left\{
B\subset E_{\Gamma,n} : \text{$B$ is independent and $|B| = n + \rep(\Gamma)$}
\right\}
\]

\[
\mathcal{I}(M_{\Gamma,n}) =
\left\{
B\subset E_{\Gamma,n} : \text{$B$ is independent}
\right\}
\]
It is now immediate from \propref{Mgammarank} that
\begin{lemma}\lemlab{Mgammasets}
The classes $\mathcal{I}(M_{\Gamma_k,n})$ and $\mathcal{B}(M_{\Gamma_k,n})$ are the independent sets and bases of
the matroid $M_{\Gamma_k,n}$.
\end{lemma}

\subsection{Structure of tight sets}
We also have a structural characterization of the tight independent sets in $M_{\Gamma,n}$.
\begin{lemma}\lemlab{Mgammatight}
An independent set $A\in \mathcal{I}(M_{\Gamma,n})$ is tight if and only if it is one of two types:
\begin{itemize}
\item [\textbf{(A)}]  Each of the non-empty $A_i$ contains a rotation.
One exceptional non-empty $A_i$
contains $\frac{1}{2}\rep_{\Gamma_k}(\Trans(\Gamma_k))$ additional elements,
and $\rep_{\Gamma_k}(\Trans(\Gamma_{A,i}))=\rep_{\Gamma_k}(\Trans(\Gamma_k))$,
and all the rest of the $A_i$ contain a single rotation only.
\item [\textbf{(B)}] Each of the $c(A)$ contains a rotation.  Two exceptional non-empty $A_i$
(w.l.o.g., $A_1$ and $A_2$)	contain, between them, $\frac{1}{2}\rep(\Trans(\Gamma_k))$ additional
elements and
$\rep_\Gamma(\langle\Trans(\Gamma_{A,1}), \Trans(\Gamma_{A,2})\rangle)=\rep_\Gamma(\Trans(\Gamma))$.
\end{itemize}
Type \textbf{(B)} is only possible when $\Gamma_k=\Gamma_2$.
\end{lemma}
\begin{proof}
One direction is straightforward: A set $A\subset E_{\Gamma,n}$ of either type
\textbf{(A)} or \textbf{(B)} satisfies, by hypothesis, $|A|=c(A)+\frac{1}{2}\rep_\Gamma(\Trans(\Gamma))$;
by construction $T(\Gamma_{A,i})$ is zero for all the non-empty
$A_i$ and $\rep_\Gamma(\Trans(A)) = \rep_\Gamma(\Trans(\Gamma))$.

On the otherhand, assuming that $A$ is tight, we see that each non-empty part has to contain a rotation,
and, since $A$ is independent there are only one (for $k=3,4,6$) or
two ($k=2$) additional elements in $A$.  Thus, the $A_i$ containing these extra elements need
to generate the translation subgroup of $\Gamma_k$.
\end{proof}

\subsection{Conjugation of independent sets}\seclab{Mkn-conj}
Let $A\in \mathcal{I}(M_{\Gamma_k})$ be an independent set, and suppose, w.l.o.g.,
that $A_1, A_2,\ldots, A_{c(A)}$ are the non-empty parts of $A$.  Let $\gamma_1,\gamma_2,\ldots,
\gamma_{c(A)}$ be elements of $\Gamma_k$.  The
\emph{conjugation of $A$ by $\gamma_1,\gamma_2,\ldots, \gamma_{c(A)}$} is defined to be
\[
\left\{(\gamma_i^{-1}A_i\gamma_i, i) : 1\le i\le c(A) \right\}
\]
Conjugation preserves independence in $M_{\Gamma_k,n}$.
\begin{lemma}\lemlab{Mgammaconj}
Let $A\in \mathcal{I}(M_{\Gamma_k})$ be an independent set.  Then the conjugation of $A$ by $c(A)$
elements $\gamma_1,\ldots, \gamma_{c(A)}$ is also independent.
\end{lemma}
\begin{proof}
\lemref{rad-conj} implies that the radical of translation subgroups is preserved under conjugation, and
whether or not $A_i$ contains a rotation is as well. Since the rank function $g_1$ is determined by these two
properties of the $A_i$, we are done.
\end{proof}

\subsection{Separating and fusing independent sets}\seclab{Mkn-fuse}
Let $A\in \mathcal{I}(M_{\Gamma_k})$ be an independent set.  A \emph{separation of $A$} is defined to be the following operation:
\begin{itemize}
\item Select $i$ and $j$ such that $A_j$ is empty.
\item Select a (potentially empty) subset $A'_i\subset A_i$ of $A_i$.
\item Replace elements $(\gamma,i)\in A'_i$ with $(\gamma,j)$.
\end{itemize}
Separation preserves independence in $M_{\Gamma_k,n}$.
\begin{lemma}\lemlab{Mgammasep}
Let $A\in \mathcal{I}(M_{\Gamma_k})$ be an independent set.  Then any separation of $A$ is also an independent set.
\end{lemma}
\begin{proof}
Let $B$ be a separation of $A$.  If the subset $A'_i$ in the definition of a separation is
empty, then $B$ is the same as $A$, and there is nothing to prove.

An independent set is either tight or a subset of a tight set.  (Bases in particular are tight.)
Consequently, by \lemref{Mgammatight}, either $B_i$ or $B_j$ consists of a single element.
Assume w.l.o.g., it is $B_j$.  Define $C \subset E_{\Gamma, n}$ as
$C_k = B_k$ for $k \neq j$ and $C_j$ empty; i.e. $C$ is $B$ with the single element in $B_j$ dropped.
Then $C$ is a subset of $A$ and hence independent.  If $B_j$ consists of a rotation, then adding it to $C$
clearly preserves independence.   If $B_j$ consists of a translation $\gamma$, then since $A$ is independent
we must have $\gamma \notin \Rad(\Trans(C))$.  Consequently $B = C + (\gamma, j)$ is independent
since $\Rad(\Trans(B)) > \Rad(\Trans(C))$ and hence $\rep_{\Gamma_k}(\Trans(B)) > \rep_{\Gamma_k}(\Trans(C))$.
\end{proof}

The reverse of separation is \emph{fusing a set $A$ on $A_i$ and $A_j$}.  This operation replaces $A_i$ with $A_i\cup A_j$ and
makes $A_j$ empty.  Fusing does not, in general, preserve independence, but it takes tight sets to spanning ones.
\begin{lemma}\lemlab{Mgammafuse}
Let $A$ be a tight independent set, and suppose that $A_i$ and $A_j$ are non-empty.  Then, after fusing $A$ on $A_i$ and $A_j$,
the result is a spanning set (with one less part).
\end{lemma}
\begin{proof}
Let $B$ be the set resulting from fusing $A$ on $A_i$ and $A_j$. By hypothesis, all the non-empty $A_\ell$ contain a rotation,
so this is true of the non-empty $B_\ell$ as well.
The lemma then follows by noting that $\Trans(A) \le \Trans(B)$, so the same is true of the radicals by \lemref{rad-monotone}.
Thus, $g_1(B)=c(B)+\rep(\Trans(\Gamma_k))$, and this implies $B$ is spanning.
\end{proof}

\chapter{Sparse graphs}\chaplab{graphs}
\section{Colored graphs and the map $\rho$}\seclab{colored-graphs}
We will use \emph{colored graphs}\footnote{This terminology comes from Igor Rivin \cite{R06}, and is consistent with \cite{MT10}.},
which are also known as ``gain graphs'' (e.g., \cite{R11}) or ``voltage graphs'' \cite{Z98} as the
combinatorial model for crystallographic frameworks and direction networks.  In this section we give the
definitions and explain the relationship between colored graphs and graphs with a free $\Gamma_k$-action.

\subsection{Colored graphs}
Let $G=(V,E)$ be a finite, directed graph, with $n$ vertices
and $m$ edges.  We allow multiple edges and self-loops, which are treated the same as other edges.
A \emph{$\Gamma_k$-colored-graph} (shortly, \emph{colored graph}) $(G,\bgamma)$ is a finite, directed multigraph $G$ and an assignment
$\bgamma=(\gamma_{ij})_{ij\in E(G)}$ of a group element $\gamma_{ij}\in \Gamma_k$ (the ``color'') to each edge $ij\in E(G)$.

\subsection{The covering map} Although we work with colored graphs because they are technically easier, crystallographic
frameworks were defined in terms of infinite graphs $\Gtilde$ with a free $\Gamma$-action $\varphi$ with finite
quotient.  In fact, the formalisms are equivalent.  The following is a straightforward specialization of covering space
theory (see, e.g., \cite[Section 1.3]{H02}), but we provide the dictionary for the convenience of the reader.

Let $(G,\bgamma)$ be a colored graph, we define its \emph{lift} $\Gtilde=(\Vtilde,\Etilde)$ by the following construction:
\begin{itemize}
\item For each vertex $i\in V(G)$, there is a subset of vertices
$\{i_{\gamma}\}_{\gamma\in \Gamma}\subset V(\Gtilde)$ (the fiber over $i$).
\item For each (directed) edge $ij\in E(G)$ with color $\gamma_{ij}$, and for each $\gamma\in \Gamma_k$,
there is an edge $i_{\gamma}j_{\gamma\cdot\gamma_{ij}}$
in $E(\Gtilde)$ (the fiber over $ij$).
\item The $\Gamma$-action on vertices is $\gamma \cdot i_{\gamma'} = i_{\gamma \gamma'}$.
The action on edges is that induced by the vertex action.
\end{itemize}

Now let $(\Gtilde,\varphi)$ be an infinite graph with a free $\Gamma_k$-action that has finite quotient.
We associate a colored graph $(G,\bgamma)$ to $(\Gtilde,\varphi)$ by the following construction, which we define to be
the \emph{colored quotient}:
\begin{itemize}
\item Let $G=\Gtilde/\Gamma$ be the quotient of $\Gtilde$ by $\Gamma$, and fix an (arbitrary) orientation
of the edges of $G$ to make it a directed graph.  By hypothesis, the vertices of $G$ correspond to the vertex orbits in $\tilde{G}$
and the edges to the edge orbits in $\tilde{G}$
\item For each vertex orbit under $\Gamma$ in $\Gtilde$, select a representative $\tilde{i}$.
\item For each edge orbit $\tilde{ij}$ in $\Gtilde$ there is a unique edge that has the representative
$\tilde{i}$ as its tail.  There is also a unique element $\gamma_{ij}\in \Gamma$ such that the
head of $\tilde{ij}$ is $\gamma_{ij}(\tilde{j})$.  We define this $\gamma_{ij}$ to be the color
on the edge $ij\in G$.
\end{itemize}
The \emph{projection map} from $(\Gtilde,\varphi)$ to its colored quotient is the function that sends
a vertex $\tilde{i}\in V(\Gtilde)$ its representative $i\in V(G)$.  Figures
\ref{fig:gam2graphtoperiodic} and \ref{fig:gam4graphtoperiodic} both show examples; the
color coding of the vertices in the infinite developments indicated the fibers over
vertices in the colored quotient.

The following lemma is straightforward:
\begin{lemma}\lemlab{lifts}
Let $(G,\bgamma)$ be a $\Gamma_k$-colored graph.  Then its lift is well defined, and is an infinite
graph with a free $\Gamma_k$-action.  If $(\Gtilde,\varphi)$ is an infinite graph with a free
$\Gamma_k$-action, then it is the lift of its colored quotient, and the projection map is well-defined
and a covering map.
\end{lemma}

\subsection{The map $\rho$} \seclab{rho}
Let $(G,\bgamma)$ be a colored graph, and let $P=\{e_1, e_2,\ldots, e_t\}$ be any \emph{closed path}
in $G$; i.e., $P$ is a not necessarily simple walk in $G$ that starts and ends at the same vertex crossing
the edges $e_i$ in order.  If we select a vertex $b$ as a \emph{base point}, then the closed paths are elements of the
\emph{fundamental group} $\pi_1(G,b)$.

We define the map
$\rho$ as:
\[
\rho(P) = \gamma_{e_1}^{\epsilon_1}\cdots \gamma_{e_t}^{\epsilon_t}
\]
where $\epsilon_i$ is $1$ if $P$ crosses $e_i$ in the forward direction (from tail to head) and $-1$ otherwise.  For a connected graph
$G$ and choice of base vertex $i$, the map $\rho$ induces a well-defined homomorphism $\rho: \pi_1(G, i) \to \Gamma$.

\subsection{Cyclic groups} The preceding development of colored graphs is in terms of a crystallographic group
$\Gamma$, but the construction is quite general, and it also works for any group such as e.g. $\Z/k\Z$.  Since $\Z/k\Z$
is abelian, it is easy to check that $\rho$ depends on its image on cycles in $G$ only, which makes the theory
simpler.  The following is \lemref{lifts} adapted for $\Z/k\Z$-colored graphs.
\begin{lemma}\lemlab{cone-lifts}
Let $(G,\bgamma)$ be a $\Z/k\Z$-colored graph.  Then its lift is well defined, and is a finite
graph with a free $\Z/k\Z$-action.  If $(\Gtilde,\varphi)$ is a finite graph with a free
$\Z/k\Z$-action, then it is the lift of its colored quotient, and the projection map is well-defined
and a covering map.
\end{lemma}

\section{The subgroup of a $\Gamma_k$-colored graph}\seclab{closed-paths}
The map $\rho$, defined in the previous section, is fundamental to the results of this paper.  In
this section, we develop properties of the $\rho$-image of a colored graph $(G,\bgamma)$ and
connect it with the matroid $M_{\Gamma_k,n}$ which was defined in \secref{group-matroid}.

\subsection{Colored graphs with base vertices}
Let $(G,\bgamma)$ be a colored graph with $n$ vertices and $c$ connected components $G_1,G_2,\ldots, G_c$.
We select a \emph{base vertex} $b_i$ in each connected component $G_i$, and denote the set of
base vertices by $B$.  The triple $(G,\bgamma,B)$ is then defined to be a \emph{marked colored graph}.

If $(G,\bgamma,B)$ is a marked colored graph then $\rho$ induces a homomorphism from $\pi_1(G_i,b_i)$
to $\Gamma_k$.  In the rest of this section, we show how to use these homomorphisms to
define a map from $(G,\bgamma)$ to $E_{\Gamma_k,n}$, the ground set of the matroid $M_{\Gamma_k,n}$.

\subsection{Fundamental closed paths generate the $\rho$-image}
Let $(G,\bgamma,B)$ be a marked colored graph with $n$ vertices and $c$ connected components.  Select
and fix a maximal forest $F$ of $G$, with connected components $T_1,T_2,\ldots, T_c$.  The $T_i$
are spanning trees of the connected components $G_i$ of $G$, with the convention that when a connected
component $G_i$ has no edges.

With this data, we define, for each edge $ij\in E(G)-E(F)$ the \emph{fundamental closed path of $ij$} to be the path that:
\begin{itemize}
\item Starts at the base vertex $b_\ell$ in the same connected component $G_\ell$ as $i$ and $j$.
\item Travels the unique path in $T_\ell$ to $i$.
\item Crosses $ij$.
\item Travels the unique path in $T_\ell$ back to $v_\ell$.
\end{itemize}
Fundamental closed paths with respect to $F$ in $G_i$ generate $\pi_1(G_i, b_i)$ \cite[Proposition 1A.2]{H02}.

\subsection{From colored graphs to sets in $E_{\Gamma_k,n}$}
We now let $(G,\bgamma,B)$ be a marked colored graph and fix a choice of spanning forest $F$.  We associate
with $(G,\bgamma,B,F)$ a subset $A(G,B,F)$ of $E_{\Gamma_k,n}$ (defined in \secref{group-matroid}) as follows:
\begin{itemize}
\item For each edge $ij\in E(G_\ell)-E(T_\ell)$, let $P_{ij}$ be the fundamental closed
path with respect to $T_i$ and $b_i$ of $ij$.
\item Add an element $(\rho(P_{ij}),\ell)$ to $A(G,B,F)$.
\end{itemize}
The following is immediate from the previous discussion.
\begin{lemma}\lemlab{GBF-eq-GammaA}
Adopting the notation from \secref{group-matroid}, $\Gamma_{A(G,F,B), \ell} = \rho(\pi_1(G_\ell, v_\ell))$.
\end{lemma}
Since we will show, in \secref{gamma22}, that the invariants we need are independent of $B$ and $F$,
we frequently suppress them from the notation when the context is clear.

\section{Map-graph preliminaries}\seclab{sparse-prelim}
The families of colored graphs we define in the next sections have, as their underlying (uncolored, undirected) multi-graphs,
a \emph{map-graph} structure.  In this short section, we define map graphs and review the properties we need.

\subsection{Map-graphs and sparsity}
A \emph{map-graph} is a graph in which every connected component has exactly one cycle.  In this definition, self-loops
correspond to cycles.
A \emph{$2$-map-graph} is a graph that is the edge-disjoint union of two spanning
map-graphs.  See \figref{map-graph-example} for an example; observe that map-graphs
do \emph{not} need to be connected.
\begin{figure}[htbp]
\centering
\includegraphics[width=.45\textwidth]{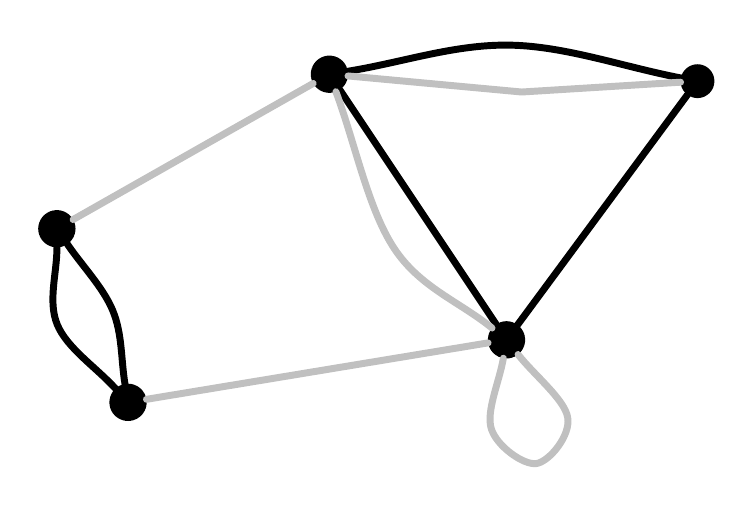}
\caption{A $2$-map-graph with its certifying decomposition into map-graphs indicated by edge color.}
\label{fig:map-graph-example}
\end{figure}

\subsection{The overlap graph}
Let $G$ be a $2$-map-graph and fix a decomposition
into two spanning map-graphs $X$ and $Y$.  Let  $X_i$ and $Y_i$
be the connected components of $X$ and $Y$, respectively.
Also select a base vertex $x_i$ and $y_i$ for each connected component of $X$ and $Y$,
with all base vertices on the cycle of their component. Denote the collection of base vertices by $B$.

We define the \emph{overlap graph} of $(G,X,Y,B)$ to be the directed graph with:
\begin{itemize}
\item Vertex set $B$.
\item A directed edge from $x_i$ to $y_i$ if $y_i$ is a vertex in $X_i$.
\item A directed edge from $y_i$ to $x_i$ if $x_i$ is a vertex in $Y_i$
\end{itemize}
\figref{overlap-example} gives an example.
\begin{figure}[htbp]
\centering
\subfigure[]{\includegraphics[width=.45\textwidth]{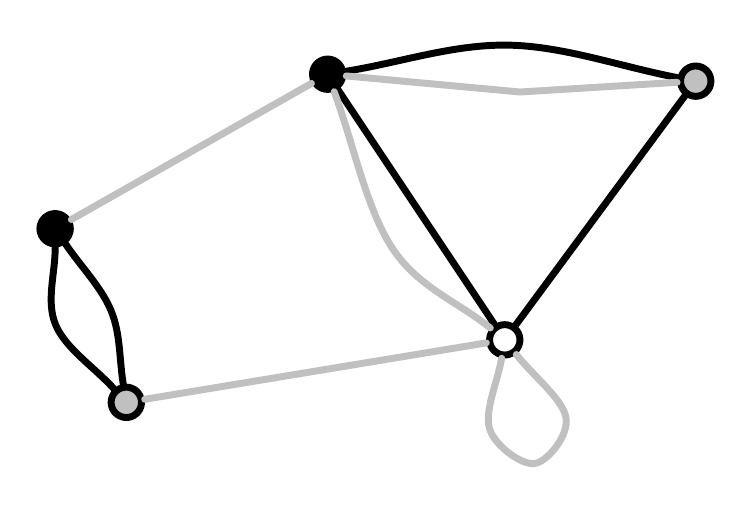}}
\subfigure[]{\includegraphics[width=.35\textwidth]{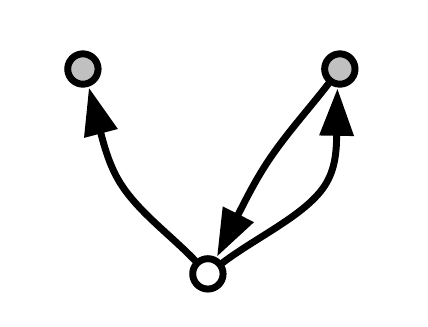}}
\caption{Example of the overlap graph: (a) a $2$-map-graph with a fixed decomposition and
base vertices (gray for black connected components and white for the gray connected component); (b)
the associated overlap graph.}
\label{fig:overlap-example}
\end{figure}
The property of the overlap graph we need is:
\begin{prop}\proplab{overlap-cycle}
Let $G$ be a $2$-map-graph with fixed decomposition and choice of base vertices.  The overlap graph of $(G,Y,R,B)$
has a directed cycle in each connected component.
\end{prop}
\begin{proof}
Every vertex has exactly one incoming edge, since each vertex is in exactly one connected component of each of $X$ and $Y$.
Thus, as an undirected graph, the overlap graph is a map-graph (see, e.g., \cite{ST09}).
\end{proof}

\section{$\Gamma$-$(2,2)$ graphs} \seclab{gamma22}
In this section we define \emph{$\Gamma$-$(2,2)$ graphs} which are the first of two key
families of colored graphs introduced in this paper (the second is \emph{$\Gamma$-colored-Laman
graphs}, defined in \secref{gamma-laman}).  We also state the main combinatorial results on
$\Gamma$-$(2,2)$ graphs, but defer the proof of a key technical result, \propref{gamma11}
to \secref{sparse}.

\subsection{The translation subgroup of a colored graph}\seclab{translation-subgroup-colored-graph}
Let $(G,\bgamma,B)$ be a marked colored graph, as in \secref{closed-paths}, with connected components
$G_1,G_2,\ldots, G_c$ and base vertices $b_1,b_2,\ldots, b_c$.  Recall from \secref{rho}
that, with this data, there is a homomorphism
\[
\rho : \pi_{1}(G_i,b_i)\to \Gamma_k
\]
We define $\Trans(G,B)$ to be
\[
\Trans(G,B) = \langle \Trans(\rho(G_i,b_i)) : i = 1,2,\ldots, c \rangle
\]
We define $\rep_{\Gamma_k}(G)=\rep_{\Gamma_k}(\Trans(G,B))$.  As the notation suggests,
$\rep_{\Gamma_k}(G)$ is independent of the choice of base vertices $B$.
\begin{lemma}\lemlab{rep-of-trans-G}
Let $(G,\bgamma,B)$ be a marked colored graph.
The quantity  $\rep_{\Gamma_k}(G)$ is independent of the choice of base vertices, and so is
a property of the underlying colored graph $(G,\bgamma)$.
\end{lemma}
\begin{proof}
Changing base vertices corresponds to conjugation.  \lemref{rad-conj} implies that
the radical of $\Trans(G,B)$ is preserved under conjugation.  Since $\rep_{\Gamma_k}(\cdot)$
depends only on the radical, the lemma follows.
\end{proof}

\subsection{The quantity $T$ for a colored graph}
Let $(G,\bgamma,B)$ be a marked colored graph, with $G$ connected (and so a single
base vertex $b$).  We define $T(G)$ to be $T(\rho(\pi_1(G,b)))$.  The proof of the
following lemma is entirely similar to that of \lemref{rep-of-trans-G}.
\begin{lemma}\lemlab{T-of-G}
Let $(G,\bgamma,B)$ be a marked colored graph.
The quantity  $T(G)$ is independent of the choice of base vertices, and so is
a property of the underlying colored graph $(G,\bgamma)$.
\end{lemma}
\begin{figure}[htbp]
\centering
\includegraphics[width=.8\textwidth]{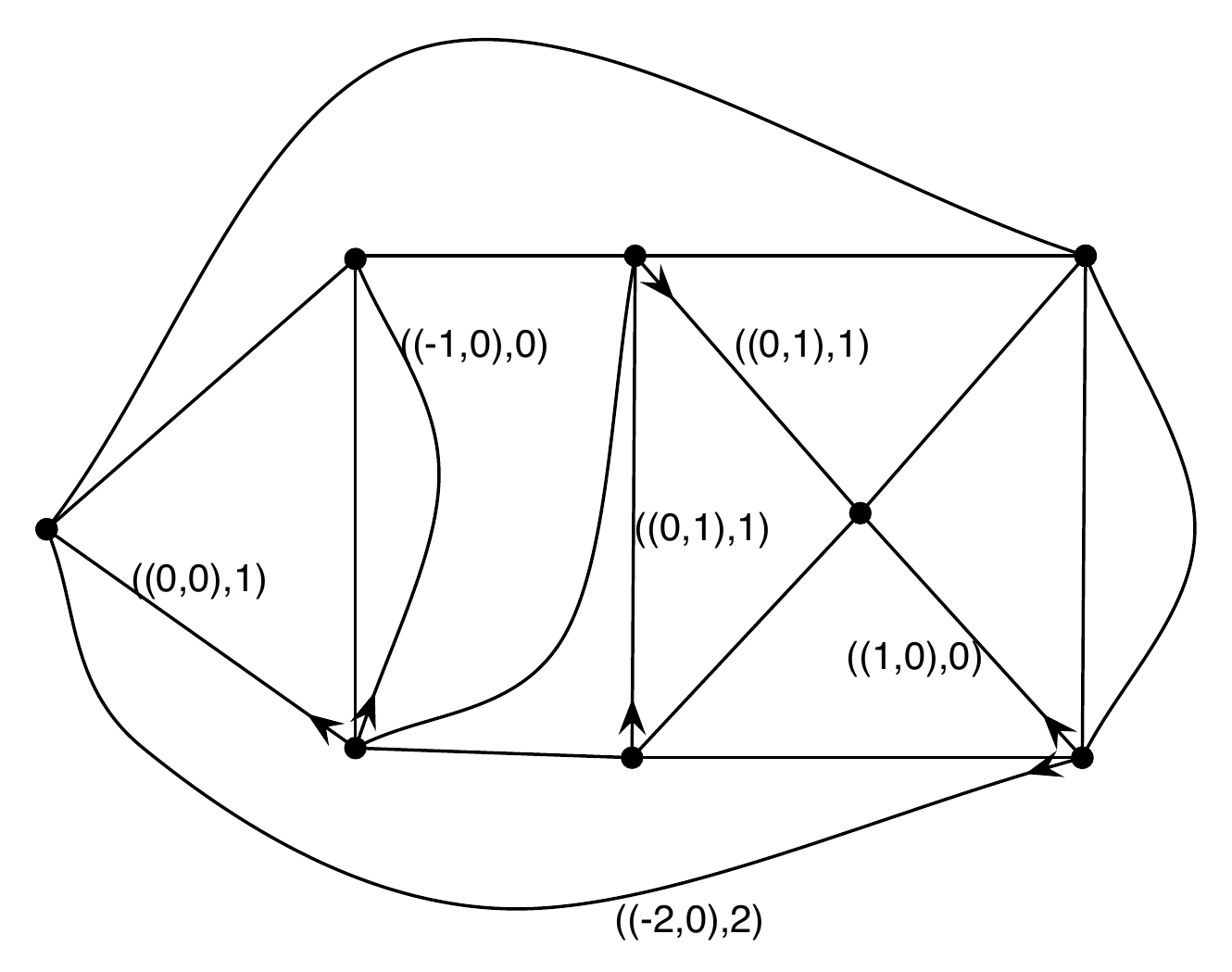}
\caption{An example of a \GammaTT graph when $\Gamma = \Gamma_3$.}
\label{fig:gam22}
\end{figure}
\subsection{$\Gamma$-$(2,2)$ graphs}
We are now ready to define $\Gamma$-$(2,2)$ graphs.  Let $(G,\bgamma)$ be a colored graph
with $n$ vertices and $c$ connected components $G_i$.  We define the function $f$ to be
\[
f(G) = 2n + \rep_{\Gamma_k}(G) - \sum_{i=1}^c T(G_i)
\]
A colored graph $(G,\bgamma)$ on $n$ vertices and $m$ edges is defined to be a \emph{$\Gamma$-$(2,2)$ graph} if:
\begin{itemize}
\item The number of edges $m$ is $2n+\rep(\Trans(\Gamma_k))$ (i.e., it is the maximum possible value for $f$).
\item For every subgraph $G'$ of $G$, with $m'$ edges, $m' \le f(G')$.
\end{itemize}
We note that it is essential that the definition is made
over \emph{all} subgraphs, and not just vertex-induced or connected ones.
\figref{gam22} shows an example of a \GammaTT-graph.

\subsection{Direction network derivation}
Before continuing with the development of the combinatorial theory, we quickly motivate
the definition of \GammaTT-graphs.  Readers who are not familiar with rigidity and
direction networks may want to either skip to \secref{gamma11-def} and revisit this,
purely informative, section after reading the definitions in
\chapref{direction-networks}.

\propref{crystal-collapse}, in \secref{cdns} below,
says that a generic direction network on a $\Gamma$-colored graph $(G,\bgamma)$ has
only \emph{collapsed} realizations (with all the points on top of each other and a trivial representation for the
$\Trans(\Gamma)$), if and only if $(G,\bgamma)$ is \GammaTT.

The definition of the function $f$ comes from analyzing the degree of freedom of collapsed realizations.
For any realization $G(\vec p, \Phi)$, we can translate it (which preserves directions),
so that $\Phi(r_k)$ has the origin as its rotation center.  Then, restricted to a subgraph $G'$ of $G$:
\begin{itemize}
\item The total number of variables involved in the equations giving the edge directions is $2n'+\rep_{\Gamma_k}(G')$.
Since we fix $\Phi(r_k)$ to rotate at the origin (see \secref{cdns} for an
explanation why we can do this), the only variability
left in $\Phi$ is $\Phi(\Trans(\Gamma_k))$.  Since $\rep_{\Gamma_k}(G')$
measures how much of $\Trans(\Gamma)$ is ``seen'' by $G'$, this is the term we add.
\item Each connected component $G_i$ has a $T(G'_i)$-dimensional space of collapsed realizations.  If $G_i'$ has a
rotation, then a collapsed realization of the lift $\tilde G_i'$ must lie on the corresponding rotation center
since a solution must be rotationally symmetric.  When $G_i'$ has no rotation, no such restriction exists, and
there are $2$-dimensions worth of places to put the collapsed $\tilde G_i'$.
Each collapsed connected component is independent of the others, so this term is additive
over connected components.
\end{itemize}
The heuristic above coincides with the definition of the function $f$.

\subsection{$\Gamma$-$(1,1)$ graphs}\seclab{gamma11-def}
We will characterize $\Gamma$-$(2,2)$ graphs in terms of decompositions into simpler
\emph{$\Gamma$-$(1,1)$ graphs}\footnote{The terminology of ``$(2,2)$'' and ``$(1,1)$'' comes from
the fact that spanning trees of finite graphs are ``$(1,1)$-tight'' in the sense of \cite{LS08}.  The
$\Gamma$-$(1,1)$ graphs defined here are, in a sense made more precise in \cite[Section 5.2]{MT10}, analogous to spanning
trees.  We don't go into details here in the interest of space, since the analogy isn't necessary for any of the proofs.}, which
we now define.

Let $(G,\bgamma)$ be a colored graph and select a base vertex $b_i$ for each connected component $G_i$ of $G$.
We define $(G,\bgamma)$ to be a \emph{$\Gamma$-$(1,1)$ graph} if:
\begin{itemize}
\item $G$ is a map-graph plus $\frac{1}{2}\rep_{\Gamma_k}(\Trans(\Gamma_k))$ additional edges.
\item For each connected component $G_i$ of $G$, $\rho(\pi_1(G_i,b_i))$ contains a rotation.
\item We have $\rep_{\Gamma_k}(G)=\rep_{\Trans(\Gamma_k)}(\Gamma_k)$, i.e., $\Rad(\Trans(G,B)) = \Trans(\Gamma_k)$.
\end{itemize}
Although we do not define \GammaOO graphs via sparsity counts, there is an alternative
characterization in these terms.  We define the function $g(G)$ to be
\[
g(G) = n + \frac{1}{2}\rep_{\Gamma_k}(G) - \frac{1}{2}\sum_{i=1}^c T(G_i)
\]
where $(G,\bgamma)$ is a colored graph and $n$ and $c$ are the number of vertices and
connected components.  Notice that $g=\frac{1}{2}f$.  In \secref{sparse} we will show:
\begin{prop}[\gammaOOmatroid] \proplab{gamma11}
The family of $\Gamma$-$(1,1)$ graphs gives the bases of a matroid, and the rank of
the $\Gamma$-$(1,1)$ matroid is given by the function:
\[
g(G) = n + \frac{1}{2}\rep_{\Gamma_k}(G) - \frac{1}{2}\sum_{i=1}^c T(G_i)
\]
In particular, this implies that $g$ is non-negative, submodular, and monotone.
\end{prop}

\subsection{Decomposition characterization of $\Gamma$-$(2,2)$ graphs}
The key combinatorial result about $\Gamma$-$(2,2)$ graphs, that is used in an essential way to
prove the ``collapsing lemma'' \propref{crystal-collapse}, is the following.
\begin{prop}\proplab{gamma22-decomp}
Let $(G,\bgamma)$ be a colored graph.  Then $(G,\bgamma)$ is a $\Gamma$-$(2,2)$ graph if and only
if it is the edge-disjoint union of two spanning $\Gamma$-$(1,1)$ graphs.
\end{prop}
\begin{proof}
Since $f=2g$, and \propref{gamma11} implies that $g$ meets the hypothesis of \theoref{edmonds1}, we conclude
that the $\Gamma$-$(2,2)$ graphs are also the bases of a matroid.  \theoref{edmonds2} then says that the
$\Gamma$-$(2,2)$ matroid must coincide with the class of colored graphs defined by the desired decomposition.
\end{proof}

\section{$\Gamma$-colored Laman graphs}\seclab{gamma-laman}
We are now ready to define \emph{$\Gamma$-colored-Laman graphs}, which are the
colored graphs characterizing minimally rigid generic frameworks in \theoref{main}.
Just as for \GammaTT graphs, we define them via sparsity counts.
\begin{figure}[htbp]
\centering
\subfigure[]{\includegraphics[width=.45\textwidth]{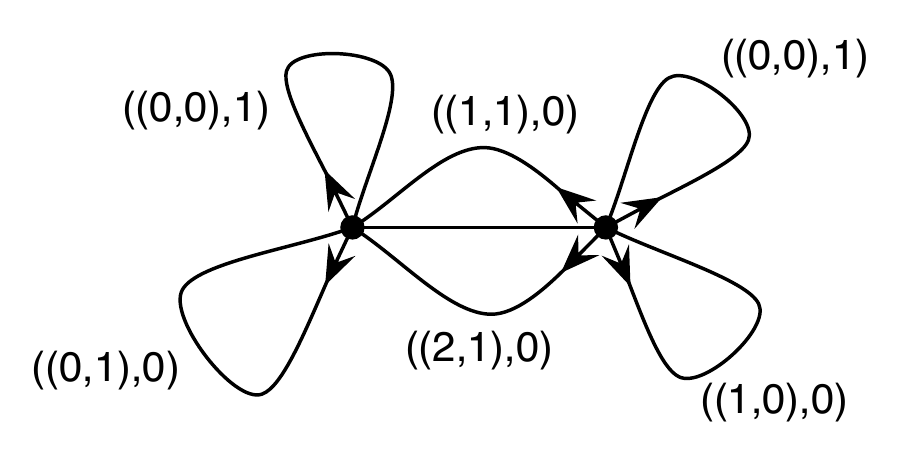}}
\subfigure[]{\includegraphics[width=.35\textwidth]{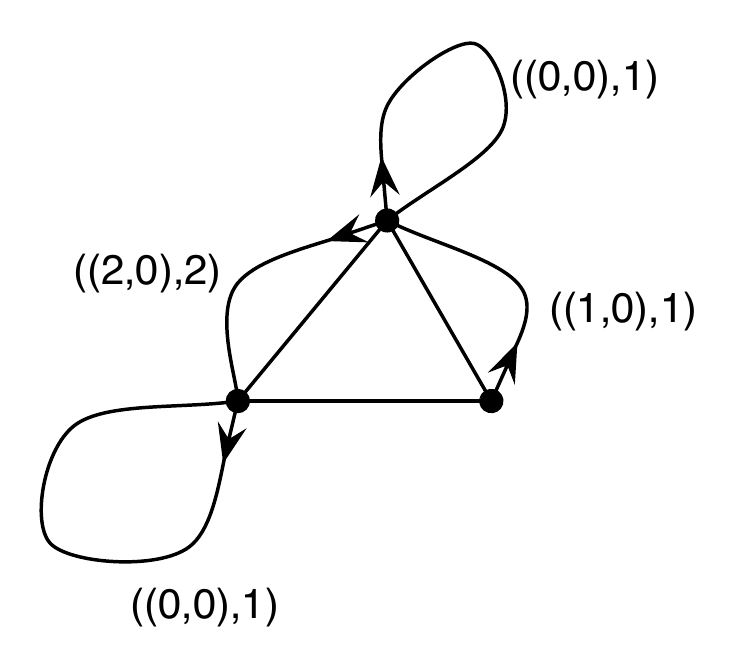}}
\caption{Examples of \GammaCL graphs:
(a) a $\Gamma_2$-colored-Laman graph;
(b) a $\Gamma_3$-colored-Laman graph
}
\label{fig:gamlaman}
\end{figure}
\subsection{Definition of $\Gamma$-colored-Laman graphs}
Let $(G,\bgamma)$ be a colored graph, and let $f$ be the sparsity function defined in \secref{gamma22}.
The most direct definition of the sparsity function $h$ for $\Gamma$-colored-Laman graphs is:
\[
h(G) = f(G) -1
\]

A colored graph $(G,\bgamma)$ is defined to be \emph{$\Gamma$-colored-Laman} if:
\begin{itemize}
\item $G$ has $n$ vertices and $m=2n + \rep_{\Gamma_k}(\Trans(\Gamma_k)) - T(\Gamma_k) - 1$ edges.
\item For all subgraphs $G'$ spanning $m'$ edges, $m'\le h(G')$
\end{itemize}
\figref{gamlaman} shows some examples of $\Gamma$-colored-Laman graphs.
If a colored graph is a subgraph of a $\Gamma$-colored-Laman graph, then it is defined to be
\emph{$\Gamma$-colored-Laman sparse}.  Equivalently, $(G,\bgamma)$ is $\Gamma$-colored-Laman sparse
if only the ``$m'\le h(G')$'' condition above holds.

\subsection{Alternate formulation of $\Gamma$-colored-Laman graphs}
While the definition of $h$ is all that is needed to prove
\theoref{main}, it does not give any motivation in terms of a degree-of-freedom count.  We now
give an alternate formulation of \GammaCL{} via the Teichmüller space and the centralizer, which
were defined in \secref{teich}, that will let us do this.

Let $(G,\bgamma,B)$ be a marked colored graph with connected components $G_1, \dots, G_c$ and $n$ vertices,
and let $\Trans(G,B)$ be its translation subgroup as defined in \secref{translation-subgroup-colored-graph}.
We define
\[
\teich_{\Gamma_k}(G) = \teich_{\Gamma_k}(\Trans(G, B))
\]
which, by a proof nearly identical to that of \lemref{rep-of-trans-G} is well-defined and independent of the
choice of base vertices.

For a component $G_\ell$ with base vertex $i_\ell$, we set
$\cent_{\Gamma_k}(G_\ell) = \cent_{\Gamma_k}(\rho(\pi_1(G_\ell, i_\ell)))$.  For similar reasons,
$\cent_{\Gamma_k}(G_\ell)$ is also independent of the base vertex.

We can now define a ``more natural'' sparsity function
\[
h'(G) = 2n + \teich_{\Gamma_k}(G) -
\left(
\sum_{i=1}^c \cent_{\Gamma_k}(G_i))
\right)
\]
The class of colored graphs defined by $h'$ is the same as that arising from $h$, giving a second definition
of $\Gamma$-colored-Laman graphs.  Since \lemref{gamma-cl-equiv} is not used to prove any further results, we
omit the proof.
\begin{lemma}\lemlab{gamma-cl-equiv}
A colored graph $(G,\bgamma)$ is \GammaCL if and only if:
\begin{itemize}
\item $G$ has $n$ vertices and $m=2n+\teich(\Gamma)-\cent(\Gamma)$ edges.
\item For all subgraphs $G'$ spanning $m'$ edges, $m'\le h'(G')$
\end{itemize}
\end{lemma}

\subsection{Degree of freedom heuristic}
The function $h'$ is amenable to an interpretation that allows us, by \lemref{gamma-cl-equiv}, to
give a  rigidity-theoretic ``degree of freedom'' derivation of $\Gamma$-colored-Laman graphs.  This
section is expository, and readers unfamiliar with rigidity theory may skip to \secref{doubling}
and return here after reading \chapref{rigidity}.

Given a framework with underlying colored graph $(G,\bgamma)$, with
the graph $G$ having $n$ vertices and $c$ connected components $G_1,G_2,\ldots, G_c$, we find that:
\begin{itemize}
\item We have $2n$ degrees of freedom from the points.  From the representation
$\Phi: \Gamma_k \to \Euc(2)$,
there are $\rep(\Gamma_k)$ degrees of freedom,
but if we mod out by trivial motions from $\Euc(2)$, we have $\teich_{\Gamma_k}(\Gamma)$ degrees of
freedom left.  	However, we have only $\teich_{\Gamma_k}(G)$ degrees of freedom that apply to $G$.
\item Each connected component has $\cent_{\Gamma_k}(G_i)$ trivial degrees of freedom.  Since elements in the
centralizer for $G_i$ commute with those in $\rho(\pi_1(G_i))$,
we may ``push the vertices of $G_i$ around'' with the centralizer
elements while preserving symmetry.  Since these motions always exist, they are trivial.
\end{itemize}
This heuristic corresponds to the function $h'$.

\subsection{Edge-doubling characterization of $\Gamma$-colored-Laman graphs}\seclab{doubling}
The main combinatorial fact about $\Gamma$-colored-Laman graphs we need is the following simple
characterization by edge-doubling (cf. \cite{LY82,R84}).

\begin{prop}\proplab{doubleedge}
Let $\Gamma=\Gamma_k$ for $k=2,3,4,6$ be a crystallographic group and let $(G,\bgamma)$ be a
$\Gamma$-colored graph.  Then $(G,\bgamma)$ is $\Gamma$-colored-Laman if and only if for any
edge $ij\in E(G)$, the colored graph $(G',\bgamma')$ obtained by adding a copy of $ij$ to $G$
with the same color results in  a $\Gamma$-$(2,2)$ graph.
\end{prop}
\begin{proof}
This is straightforward to check once we notice that $(G,\bgamma)$ is
$\Gamma$-colored-Laman if and only if no subgraph $G'$ with $m'$ edges has $m'=f(G')$.
\end{proof}

\subsection{$\Gamma$-colored-Laman circuits}
Let $(G,\bgamma)$ be a colored graph.  We define $(G,\bgamma)$ to be a \emph{$\Gamma$-colored-Laman circuit}
if it is edge-wise minimal with the property of being not $\Gamma$-colored-Laman sparse.  More formally,
$(G,\bgamma)$ is a $\Gamma$-colored-Laman circuit if:
\begin{itemize}
\item $(G,\bgamma)$ is not $\Gamma$-colored-Laman sparse
\item For all colored edges $ij\in E(G)$, $(G-ij,\bgamma)$ is $\Gamma$-colored-Laman sparse
\end{itemize}
As the terminology suggests, $\Gamma$-colored-Laman circuits are the circuits of the matroid that has, as its bases,
$\Gamma$-colored-Laman graphs.  The following lemmas are immediate from the definition.
\begin{lemma}\lemlab{not-gamma-laman-sparse-implies-circuit}
Let $(G,\bgamma)$ be a colored graph.  If $(G,\bgamma)$ is not $\Gamma$-colored-Laman sparse, then
it contains a $\Gamma$-colored-Laman circuit as a subgraph.
\end{lemma}
\begin{lemma}\lemlab{gamma-laman-circuit-sparsity}
Let $(G,\bgamma)$ be a colored graph with $n$ vertices and $m$ edges.
Then $(G,\bgamma)$ is a $\Gamma$-colored-Laman circuit if and only if:
\begin{itemize}
\item The number of edges $m = f(G)$
\item For all subgraphs $G'$ of $G$, on $m'$ edges, $m' < f(G')$
\end{itemize}
\end{lemma}
Here $f$ is the colored-$(2,2)$ sparsity function defined in \secref{gamma22}.

\section{\GammaOO{} graphs: proof of \propref{gamma11}} \seclab{sparse}
With the definitions and main properties of $\Gamma$-$(2,2)$ and $\Gamma$-colored-Laman
graphs developed, we prove:
\gammaOOmatroid
With this, the proof of \propref{gamma22-decomp} is also complete.  The rest of this section is organized
as follows: first we prove that the $\Gamma$-$(1,1)$ graphs give the bases of a matroid and then
we argue that the rank function of this matroid is, in fact, the function $g$, defined in
\secref{gamma22}.

We recall from \secref{closed-paths}
that, for a marked colored graph $(G,\bgamma,B)$ with a fixed spanning forest
$F$, the the map $\rho$, defined in \secref{colored-graphs} induced a map from
$(G,\bgamma,B,F)$ to $E_{\Gamma_k,n}$, the ground set of the matroid $M_{\Gamma_k,n}$
from \secref{group-matroid}.  We adopt the notation of \secref{closed-paths}, and
denote the image of this map by $A(G,B,F)$.

We start by studying $A(G,B,F)$ in more detail.

\subsection{Rank of $A(G,B,F)$}
As defined, the set $A(G,B,F)$ depends on a choice of base vertices for each connected
component and a spanning forest $F$ of $G$.  Since we are interested in constructing a
matroid on colored graphs without additional data, the first structural lemma is that
the rank of $A(G,B,F)$ in $M_{\Gamma_k,n}$ is independent of the choices for $B$ and $F$.
\begin{lemma}\lemlab{rebase}
Let $(G,\bgamma,B)$ be a marked colored graph with connected components
$G_1,G_2,\ldots,G_c$ and fix a spanning forest $F$ of $G$.  Then the rank of $A(G,B,F)$ in the matroid
$M_{\Gamma_k,n}$ is is invariant under changing the base vertices and spanning forest.
\end{lemma}
\begin{proof}
For convenience, shorten the notation $A(G,B,F)$ to $A$.
By \lemref{GBF-eq-GammaA} $\rho(\pi_1(G_\ell, v_\ell)) = \Gamma_{A, \ell}$.
Changing the spanning forest
$F$ just picks out a different set of generators for $\pi_1(G_\ell, v_\ell)$, and so does not
change $\Gamma_{A, \ell}$, and thus the rank in $M_{\Gamma_k,n}$, which does not
depend on the generating set, is unchanged.

To complete the proof, we show that changing the base vertices corresponds, in $E_{\Gamma_k,n}$,
to applying the conjugation operation defined in \secref{group-matroid} to $A$.
Suppose that $G$ is connected and fix a spanning tree $F$ and a base vertex $b$.
If $P$ is a closed path starting and ending at $b$, for any other vertex $b'$ there is a path $P'$ that:
starts at $b'$, goes to $b$ along a path $P_{bb'}$, follows $P$, and then returns from $b$ to $b'$
along $P_{bb'}$ in the other direction.  We have $\rho(P') = \rho(P_{bb'})\rho(P)\rho(P_{bb'})^{-1}$, so
$P$ and $P'$ have conjugate images.  Thus changing base vertices corresponds to conjugation, and by
\lemref{Mgammaconj} we are done after considering connected components one at a time.
\end{proof}
In light of \lemref{rebase}, when we are interested only in the rank of $A(G,B,F)$, we can freely
change $B$ and $F$.  Thus, we define the notation $A(G)$ to suppress the dependence on $B$ and $F$.

\subsection{Adding or deleting a colored edge and $A(G)$}
In the proof of the basis exchange property, we will need to start with a
$\Gamma$-$(1,1)$ graph, and add a colored edge to it.  There are two possibilities:
the edge $ij$ is in the span of some connected component $G_i$ of $G$ or it is not.  Each
of these has an interpretation in terms of how $A(G+ij)$  is different from $A(G)$.
\begin{lemma}\lemlab{add-edge-AG}
Let $(G,\bgamma)$ be a colored graph and let $ij$ be a colored edge.  Then:
\begin{itemize}
\item[\textbf{(A)}] If the edge $ij$ is in the span of a connected component, $G_\ell$ of $G$, then
$A(G+ij)$ is $A(G)+(\gamma,\ell)$, where $\gamma$ is the image of the fundamental closed
path of $ij$ with respect to some spanning tree and base vertex of $G_\ell$.
\item[\textbf{(B)}] If the edge $ij$ connects two connected components $G_\ell$ and $G_r$ of $G$, then
$A(G+ij)$ is a fusing operation (defined in \secref{group-matroid}) on $A(G)$ after a conjugation.  In
particular, in the notation of \secref{group-matroid} $A(G)_\ell$ and $A(G)_r$ are fused.  Conversely,
$A(G)$ is a conjugation of a separation of $A(G+ij)$.
\end{itemize}
\end{lemma}
\begin{proof}
Statement \textbf{(A)} follows from the fact that if we pick a base vertex and spanning tree of $G_\ell$, then adding the
colored edge $ij$ to $G_\ell$ induces exactly one new fundamental closed path.

For statement \textbf{(B)}, w.l.o.g., assume that $G$ has two connected components and $ij$ connects them.  Since
$ij$ is in any spanning tree of $G+ij$, it follows that every fundamental closed path in $G+ij$  has $\rho$-image
conjugate to a closed path in $G$, so $A(G+ij)$ consists of group elements conjugate to elements in $A(G)_\ell$ and
$A(G)_r$ as required.  The converse is clear since the inverse of a conjugation is a conjugation, and the inverse
of fusing is separating.
\end{proof}

\subsection{$\Gamma$-$(1,1)$ graphs and tight independent sets in $M_{\Gamma_k,n}$}
$\Gamma$-$(1,1)$ graphs $(G,\bgamma)$ have a simple characterization in terms of $A(G)$: they
correspond exactly to the situations in which $A(G)$ is tight and independent.

\begin{lemma}\lemlab{gamma11-tight}
Let $(G,\bgamma)$ be a colored graph.  Then $(G,\bgamma)$ is $\Gamma$-$(1,1)$ if and only if $A(G)$
is tight and independent in $M_{\Gamma_k,n}$.
\end{lemma}

\begin{proof}
We recall that \lemref{Mgammatight} gave a structural characterization of tight independent sets in $M_{\Gamma_k,n}$.
The proof follows by translating the definitions from \secref{group-matroid-sets}
into graph theoretic terms.  In this proof we adopt the notation of
\secref{group-matroid-sets}, and we remind the reader that a subset $A\subset E_{\Gamma_k,n}$ is tight if it is
independent in $M_{\Gamma_k,n}$ and has:
\[
|A| = c(A) + \frac{1}{2}\rep(\Trans(\Gamma_k))
\]
elements.

We first suppose that $A(G)$ is tight, and show that $(G,\bgamma)$ is a $\Gamma$-$(1,1)$ graph.
By construction $A(G)$ has an element $(\gamma,\ell)$ if and only if there is some edge $ij$ in the
connected component $G_\ell$ not in the spanning forest $F$ used to compute $A(G)$.  It then follows that,
if $A(G)$ is tight, each connected component of $G_i$ of $G$ has at least one more edge than
$G_i\cap F$.     %
This implies that $G$ contains a spanning map graph.  Because $|A(G)| = c(A) + \frac{1}{2}\rep(\Trans(\Gamma_k))$
it then follows that $G$ is a map-graph plus $\frac{1}{2}\rep(\Trans(\Gamma_k))$ additional edges,
which are the combinatorial hypotheses for being a $\Gamma$-$(1,1)$ graph.

Now we use the fact that $A(G)$ is independent in $M_{\Gamma_k,n}$.  Independence implies that,
if non-empty, $A(G_i)$ contains a rotation, from which it follows that, for each connected component $G_i$ of $G$,
$\rho(\pi_1(G_i,b_i))$ does as well.  Similarly, independence implies that
$\rep_{\Gamma_k}(\Trans(A(G))) = \rep_{\Gamma}(\Trans(\Gamma_k))$, so the same is true for $\rep_{\Gamma_k}(G)$.
We have now shown that $(G,\bgamma)$ is a $\Gamma$-$(1,1)$ graph.

The other direction is straightforward to check.
\end{proof}

\subsection{$\Gamma$-$(1,1)$ graphs form a matroid}
We now have the tools to prove that the $\Gamma$-$(1,1)$ graphs form the bases of a matroid.  We take as the ground
set the graph $K_{\Gamma_k,n}$ on $n$ vertices that has one copy of each possible directed edge $ij$
or self-loop $ij$ with color $\gamma\in \Gamma_k$.

\begin{lemma}\lemlab{gamma11-bases}
The set of $\Gamma$-$(1,1)$ graphs on $n$ vertices
form the bases of a matroid on $K_{\Gamma_k,n}$.
\end{lemma}
\begin{proof}
We check the basis axioms (defined in \secref{matroid-prelim}).

\noindent
\textbf{Non-triviality:}
There is some \GammaOO{} graph on $n$ vertices. An uncolored tree plus $\frac{1}{2}\rep(\Gamma) + 1$ edges,
each of which is colored by a standard generator for $\Gamma$ is clearly \GammaOO.  Thus the set of bases is not empty.

\noindent
\textbf{Equal size:}
By definition, all \GammaOO{} graphs have the same number of edges.

\noindent
\textbf{Base exchange:}
The more difficult step is checking basis exchange.  To do this we let $G$ be a \GammaOO{} graph and $ij$ a colored edge
of some other \GammaOO{} graph which is not in $G$.  It is sufficient to check that there is some colored edge $i'j'\in E(G)$
such that $G+ij-i'j'$ is also a \GammaOO{} graph.  Let $(G',\bgamma')$ be the colored graph $(G+ij,\bgamma)$.

Pick base vertices $B$ and a spanning forest $F$ of $G'$ that contains the new edge $ij$.
By \lemref{rebase} forcing $ij$ to be in $F$ does not change the rank of $A(G',B,F)$ in
$M_{\Gamma_k,n}$.  \lemref{add-edge-AG} then implies that $A(G',B,F)$ is spanning, but not
independent, in $M_{\Gamma_k,n}$.  Thus there is an element of
$A(G',B,F)$ that can be removed to leave a tight, independent set.  Since $ij$ is in $F$, this
element does not correspond to $ij$.  The basis exchange axiom then follows from the characterization
of $\Gamma$-$(1,1)$ graphs in \lemref{gamma11-tight}.
\end{proof}

\subsection{The rank function of the $\Gamma$-$(1,1)$ matroid}
Now we compute the rank function of the $\Gamma$-$(1,1)$ matroid.  The
following lemma is immediate from the definitions.
\begin{lemma}\lemlab{g-and-g1}
Let $(G,\bgamma)$ be a colored graph with $n$ vertices and $c$ connected components.
Then
\[
g(G) = n - c + g_1(A(G))
\]
where $g_1$ is the rank function of the matroid $M_{\Gamma_k,n}$.
\end{lemma}
We can use this to show:
\begin{lemma}\lemlab{gamma11-independent-g}
Let $(G,\bgamma)$ be a colored graph that is independent in the $\Gamma$-$(1,1)$ matroid with $m$ edges.
Then $m=g(G)$.
\end{lemma}
\begin{proof}
By definition $(G, \bgamma)$ is a subgraph of some \GammaOO{} graph $(G', \bgamma')$.  By \lemref{gamma11-tight},
$m' = g(G')$, where $m'$ is the number of edges of $G'$.  It suffices to show that deleting an edge preserves this equality
and independence of $A(G')$.  By \lemref{add-edge-AG}, deleting an edge is equivalent to either removing an element from
$A(G')$ or separating and conjugating $A(G')$ and these both preserve independence
of $A(G')$.  In the first case, $g_1(A(\cdot))$ drops by $1$ while $n'$ and $c'$ remain constant, and in the second
case $n'$ and $g_1(A(\cdot))$ remain constant while $c'$ increases by $1$.
\end{proof}
We can now compute the rank function of the $\Gamma$-$(1,1)$ matroid.
\begin{lemma}\lemlab{gamma11-rank}
The function $g$ is the rank function of the $\Gamma$-$(1,1)$ matroid.
\end{lemma}
\begin{proof}
Let $(G,\bgamma)$ be an arbitrary colored graph with $n$ vertices and $c$ connected components.  As discussed
in \secref{matroid-prelim}, the rank of $(G,\bgamma)$ in the $\Gamma$-$(1,1)$ matroid is the maximum size of the
intersection of $G$ with a $\Gamma$-$(1,1)$ graph.  \lemref{gamma11-independent-g} implies that what
we need to show is that a maximal independent subgraph $(G',\bgamma)$ of $(G,\bgamma)$ has $g(G)$ edges.

We construct $G'$ as follows.  First pick a base vertex for every connected component of $G$ and a
spanning forest $F$ of $G$.  Initially set $G'$ to be $F$.  Then add edges one at a time
to $G'$ from $G-F$ so that $A(G')$ remains independent in $M_{\Gamma_k,n}$
until the rank of $A(G')$ is equal to that of $A(G)$.  This is possible by the
matroidal property of $M_{\Gamma_k,n}$ and \lemref{rebase}, which says the rank of $A(G')$ is
invariant under the choices of spanning forest and base vertices.

When the process stops, $A(G')$ is independent in $M_{\Gamma_k,n}$, so $G'$ is
in the $\Gamma$-$(1,1)$ matroid by \lemref{gamma11-tight}.  By construction $G'$ has
\[
m' = n - c + g_1(A(G))
\]
edges, which is $g(G)$ by \lemref{g-and-g1}.
\end{proof}

\subsection{Proof of \propref{gamma11}}
The proposition is immediate from Lemmas \ref{lemma:gamma11-bases} and \ref{lemma:gamma11-rank}.
\eop

\section{Cone-$(2,2)$ and cone-Laman graphs}\seclab{cone-sparse}
We now develop the combinatorial language for cone frameworks and direction networks.
Since it runs parallel to that for crystallographic direction networks, but is simpler,
we will be somewhat brief. \figref{cone-graph-examples} shows some examples of colored
graphs defined in this section.
\begin{figure}[htbp]
\centering
\subfigure[]{\includegraphics[width=.3\textwidth]{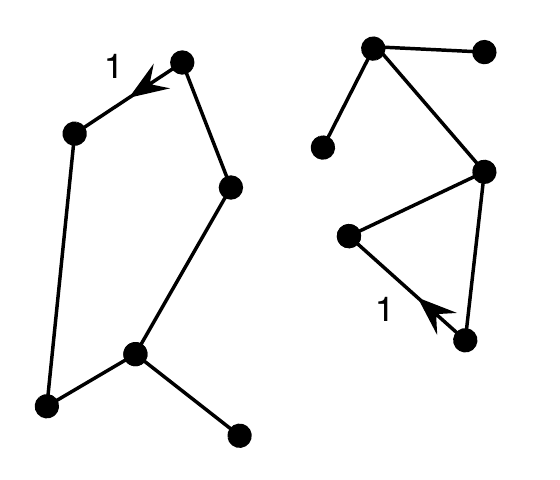}}
\subfigure[]{\includegraphics[width=.3\textwidth]{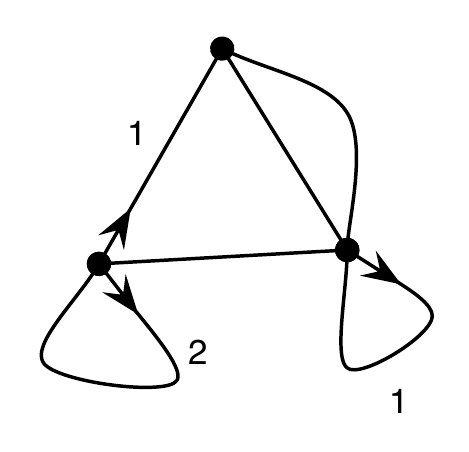}}
\subfigure[]{\includegraphics[width=.25\textwidth]{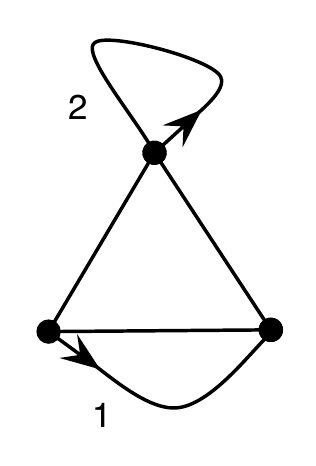}}
\caption{Examples of $\Z/3\Z$ colored graphs:
(a) a cone-$(1,1)$ graph;
(b) a cone-$(2,2)$ graph;
(c) a cone-Laman graph.  Edges without directions and colors have the
identity group element coloring them.
}
\label{fig:cone-graph-examples}
\end{figure}

\subsection{Cone-$(2,2)$ graphs}
Let $(G,\bgamma)$ be a $\Z/k\Z$ colored graph with $n$ vertices.  We define $(G,\bgamma)$
to be a cone-$(2,2)$ graph if:
\begin{itemize}
\item $G$ has $m = 2n$ edges.
\item For all subgraphs with $m'$ edges, $n'$ vertices, and connected components $G_1,G_2,\ldots, G_c$,
\[
m' \le 2n' - \sum_{i=1}^c T(G_i)
\]
\end{itemize}
The quantity $T$ is the same one defined in \secref{gamma22}, since all elements of $\Z/k\Z$ are represented
by rotations. If only the second condition holds, then $(G,\bgamma)$ is defined to be
\emph{cone-$(2,2)$ sparse}.

\subsection{Cone-$(1,1)$ graphs}
We define $(G,\bgamma)$ to be a cone-$(1,1)$ graph if $G$ is a map-graph and the cycle in each connected component
has non-trivial $\rho$-image.

The sparsity characterization of cone-$(1,1)$ graphs is:
\begin{lemma}\lemlab{cone-11-sparse}
The cone-$(1,1)$ graphs on $n$ vertices are the bases of a matroid that has as its rank function
\[
r(G') = n' - \frac{1}{2}\sum_{i=1}^c T(G_i)
\]
where $n'$ and $c'$ are the number of vertices and connected components in $G'$.
\end{lemma}
\lemref{cone-11-sparse} follows from a simplification of the arguments in
\secref{sparse}, but can also be obtained via \cite[``Matroid Theorem'']{Z82}.

\subsection{Characterization of cone-$(2,2)$ graphs}
From \theoref{edmonds2} and the matroidal \lemref{cone-11-sparse} we get a decomposition
characterization of cone-$(2,2)$ graphs.
\begin{lemma}\lemlab{cone-22-decomp}
Let $(G,\bgamma)$ be a colored graph.  The $(G,\bgamma)$ is cone-$(2,2)$ if and only if it
is the edge-disjoint union of two cone-$(1,1)$ graphs.
\end{lemma}
A corollary, by the Matroid Union \theoref{edmonds2} is:
\begin{lemma}\lemlab{cone-22-matroid}
The family of cone-$(2,2)$ graphs on $n$ vertices gives the bases of a matroid.
\end{lemma}

\subsection{Cone-Laman graphs}
Let $(G,\bgamma)$ be a $\Z/k\Z$ colored graph with $n$ vertices.  We define $(G,\bgamma)$
to be a cone-Laman graph if:
\begin{itemize}
\item $G$ has $m = 2n-1$ edges.
\item For all subgraphs with $m'$ edges, $n'$ vertices, and connected components $G_1,G_2,\ldots, G_c$,
\[
m' \le 2n' - 1 - \sum_{i=1}^c T(G_i)
\]
\end{itemize}
If only the second condition holds, we define $(G,\bgamma)$ to
be \emph{cone-Laman-sparse}.

The relationship between cone-Laman and cone-$(2,2)$ graphs is similar to that from the crystallographic case,
and has the same proof.
\begin{lemma}\lemlab{cone-laman-doubling}
Let $(G, \bgamma)$ be a $\Z/k\Z$-colored graph.  Then, $(G,\bgamma)$ is a cone-Laman graph
if and only if $G$ becomes a cone-$(2,2)$ graph after doubling any edge.
\end{lemma}

\subsection{Cone-Laman graphs are connected}
Although cone-$(2,2)$ graphs need not be connected, cone-Laman graphs are.
\begin{lemma}\lemlab{cone-laman-connected}
Let $(G,\bgamma)$ be a cone-Laman graph.  Then $G$ is connected.
\end{lemma}
\begin{proof}
Let $G$ have $n$ vertices.  By hypothesis, $G$ has $2n-1$ edges, and any subgraph on
$n'$ vertices and $m'$ edges satisfies $m'\le 2n' - 1$.  The lemma then follows from
\cite[Lemma 4]{LS08}.
\end{proof}

\subsection{Cone-Laman circuits}
Let $(G,\bgamma)$ be a $\Z/k\Z$ colored graph with $n$ vertices.  We define $(G,\bgamma)$
to be a \emph{cone-Laman circuit} if:
\begin{itemize}
\item $(G,\bgamma)$ is not a cone-Laman-sparse.
\item $(G-ij,\bgamma)$ is cone-Laman-sparse for any colored edge $ij\in E(G)$.
\end{itemize}
A fact we need is that cone-Laman circuits are always connected.
\begin{lemma}\lemlab{cone-laman-circuits-struct}
Let $(G,\bgamma)$ be a cone-Laman circuit.  Then $G$ is connected, and is either:
\begin{itemize}
\item A connected cone-$(2,2)$ graph, if $T(G)=0$.
\item A graph on $n$ vertices with $m'\le 2n'-2$ for all subgraphs, on $n'$ vertices and $m'$ edges,
if $T(G) = 2$.
\end{itemize}
\end{lemma}
\begin{proof}
Let $G$ have $n$ vertices, $m$ edges, and $c$ connected components $G_i$ with $n_i$
vertices and $m_i$ edges.  Because $(G,\bgamma)$ becomes
cone-Laman sparse after the removal of any edge, we must have
\[
m_i = 2n_i - T(G_i)
\]
for every connected component, since otherwise one of the $G_i$ would not be cone-Laman sparse
after removing one edge.  This then implies that none of the $G_i$ is cone-Laman sparse, so there
must be only one of them.

The structural statement then comes from noting that the cone-$(2,2)$ sparsity function bounds the
number of edges in any subgraph.
\end{proof}

\section{Generalized cone-$(2,2)$ graphs}\seclab{gen-cone-sparse}
As a technical tool in the proof of \theoref{direction}, we will use
\emph{generalized cone-$(2,2)$ graphs}.  These are $\Gamma_k$-colored
graphs, which we will define in terms of a decomposition property.

\subsection{Generalized cone-$(1,1)$ graphs}
Let $(G,\bgamma)$ be a $\Gamma_k$-colored graph.  We define $(G,\bgamma)$ to be a
\emph{generalized cone-$(1,1)$ graph} if, after considering the $\rho$-image modulo the translation subgroup,
the result is a cone-$(1,1)$ graph.  Equivalently, $(G,\bgamma)$ is a generalized cone-$(1,1)$ graph if:
\begin{itemize}
\item $G$ is a map graph
\item The $\rho$-image of the cycle in each connected component of $G$ is a rotation
\end{itemize}
The difference between cone-$(1,1)$ graphs and generalized cone-$(1,1)$ graphs is that
the rotations need not be around the same center.  Nonetheless, the proof of the following
lemma is nearly the same as that of \lemref{cone-11-sparse}.
\begin{lemma}\lemlab{gen-cone-11-sparse}
The generalized cone-$(1,1)$ graphs on $n$ vertices are the bases of a matroid that has as its rank function
\[
r(G') = n' - \frac{1}{2}\sum_{i=1}^c T(G_i)
\]
where $n'$ and $c'$ are the number of vertices and connected components in $G'$.
\end{lemma}

\subsection{Relation to $\Gamma$-$(1,1)$ graphs}
Generalized cone-$(1,1)$ graphs are related to $\Gamma$-$(1,1)$ graphs by this next sequence of
lemmas.
\begin{lemma}\lemlab{gamma11-is-gc11-spanning}
Let $(G,\bgamma)$ be a $\Gamma$-$(1,1)$ graph.  Then $(G,\bgamma)$ contains a
generalized cone-$(1,1)$ graph as a spanning subgraph.
\end{lemma}
\begin{proof}
This follows from the definition, since each connected component $G_i$ of $G$ has $T(G_i)=0$.  It follows
that $G_i$ has a spanning subgraph that is a connected map-graph with its cycle having a rotation as its
$\rho$-image.
\end{proof}

Let $(G,\bgamma)$ be a $\Gamma$-$(1,1)$ graph, and let $(G',\bgamma)$ be a spanning generalized cone-$(1,1)$
subgraph.  One exists by \lemref{gamma11-is-gc11-spanning}. We define $(G',\bgamma)$ to be a \emph{g.c.-basis}
of $(G,\bgamma)$.
\begin{lemma}\lemlab{gc11-circuits}
Let $(G,\bgamma)$ be a $\Gamma_k$-colored $\Gamma$-$(1,1)$ graph for $k=3,4,6$.  Let $(G',\bgamma)$
be a g.c.-basis of $(G,\bgamma)$, and let $ij$ be the (unique) edge in $E(G)-E(G')$.  Then either:
\begin{itemize}
\item The colored edge $ij$ is a self-loop and the color $\gamma_{ij}$ is a translation.
\item There is a unique minimal subgraph $G''$ of $G$, such that the $\rho$-image of $(G'',\bgamma)$
includes a translation, $ij$ is an edge of $G''$, and if $vw\in E(G'')$, then $(G'+ij-vw,\bgamma)$
is also a g.c.-basis of $(G,\bgamma)$.
\end{itemize}
\end{lemma}
\begin{proof}
If $ij$ is a self-loop colored by a translation, then it is a circuit in the matroid of
generalized cone-$(1,1)$ graphs on the ground set $(G,\bgamma)$.  Otherwise, the subgraph $G''$
the lemma requires is just the fundamental generalized-cone-$(1,1)$ circuit of
$ij$ in $(G',\bgamma)$.
\end{proof}

\subsection{Generalized cone-$(2,2)$ graphs}
Let $(G,\bgamma)$ be a $\Gamma_k$-colored graph.  We define $(G,\bgamma)$ to be a
\emph{generalized cone-$(2,2)$ graph} if it is the union of two generalized cone-$(1,1)$
graphs.  \theoref{edmonds2} implies that:
\begin{lemma}\lemlab{gen-cone22-matroid}
The generalized cone-$(2,2)$ graphs on $n$ vertices give the bases of a matroid.
\end{lemma}

The other fact about generalized cone-$(2,2)$ graphs is their relationship to
$\Gamma$-$(2,2)$ graphs.
\begin{lemma}\lemlab{gamma22-is-cone22-spanning}
Let $(G,\bgamma)$ be a $\Gamma$-$(2,2)$ graph.  Then $(G,\bgamma)$ contains a
generalized cone-$(2,2)$ graph as a spanning subgraph.
\end{lemma}
\begin{proof}
This is immediate from \lemref{cone-22-decomp} and \lemref{gamma11-is-gc11-spanning}.
\end{proof}

\chapter{Direction networks}\chaplab{direction-networks}
\section{Cone direction networks} \seclab{cones}
As a warm up for crystallographic direction networks, we will study \emph{cone direction networks}.  Our main
result on cone direction networks is the natural adaptation of \theoref{direction}.
Full definitions are given in \secref{cone-def}
below.  Genericity means that \theoref{cone-direction} is true for all but a proper algebraic subset of the
space of edge-direction assignments, and will be made precise in \secref{cone-proof}.
\begin{theorem}[\conedirection]\theolab{cone-direction}
A generic realization of a  cone direction network $\Gad$ has a
faithful realization if and only if its associated colored graph is cone-Laman.  This
realization is unique up to scaling.
\end{theorem}
In the rest of this section we give the required definitions and indicate the proof strategy.  The
proof is then carried out in Sections \ref{sec:geom}--\ref{sec:cone-proof}.

\subsection{Cone and colored direction networks}\seclab{cone-def}
A \emph{cone direction network} $(\Gtilde, \varphi, \tilde{\vec d})$, is given by a finite graph $\Gtilde$,
a free $\Z/k\Z$-action $\varphi$ on $\tilde G$, and an assignment $\tilde{\vec d}=(\tilde{\vec d}_{ij})_{ij\in E(G)}$
of a \emph{direction} to each edge, such that:
\begin{eqnarray*}
\tilde{\vec d}_{\gamma\cdot ij} = R_{k}^{\gamma}\cdot\tilde{\vec d}_{ ij} & \text{for all $\gamma\in \Z/k\Z$}
\end{eqnarray*}
Recall from \secref{crystal-prelim} that we let $\Z/k\Z$ act on $\R^2$ by mapping the generator to $R_k$,
the counter-clockwise rotation through angle $2\pi/k$ around the origin;
when the context is clear, we will simply write $\gamma\cdot\vec p_i$ for this action.
Note that $\tilde{\vec d}$ is completely defined by assigning a direction to one edge in each
$\Z/k\Z$-orbit of edges in $\tilde{G}$.

By \lemref{cone-lifts}, the combinatorial data of a cone direction network is contained in its colored
quotient graph.  We define a \emph{colored direction network} $(G,\bgamma,\vec d)$ to be a $\Z/k\Z$-colored
graph $(G,\bgamma)$ along with an assignment of a direction to every edge.

\subsection{The realization problem}
The \emph{realization problem} for a cone direction network is to find a point set
$\vec p_i$ for each $i \in V(\tilde G)$ so that each edge $ij\in E(\tilde G)$ is in the direction
$\tilde{\vec d}_{ij}$.  The \emph{realization space} of a cone direction network is defined to be:
\begin{eqnarray*}
\iprod{\vec p_j - \vec p_i}{\tilde{\vec d}^\perp_{ij}} =0 & \text{for all edges $ij\in E(\tilde{G})$} \\
\gamma\cdot \vec p_i = \vec p_{\gamma\cdot i} & \text{for all vertices $i\in V(\tilde{G})$}
\end{eqnarray*}
The unknowns are the points $\vec p_i$ and the given data are the directions $\tilde{\vec d}_{ij}$.  We
denote points in the realization space by $\tilde{G}(\vec p)$

Because the directions $\tilde{\vec d}_{ij}$ respect the $\Z/k\Z$-action $\varphi$ on $\tilde{G}$,
the realization space is identified with the following system (denoted $(G, \bgamma, \vec d)$) defined on the quotient graph
$(G,\bgamma)$:
\begin{equation}
\iprod{\gamma_{ij}^{-1}\vec d^\perp_{ij}}{\vec p_j} + \iprod{\vec d^\perp_{ij}}{-\vec p_i} = 0 \label{colored-cone-system}
\end{equation}
The points $\vec p_i$ for each $i \in V(G)$ are the unknowns and the directions $\vec d_{ij}$ are the given data.
We denote points in the realization space by $G(\vec p)$

The following is immediate from the definitions and \lemref{cone-lifts}.
\begin{lemma}\lemlab{cone-direction-network-lifts}
Given a colored direction network $(G,\bgamma,\vec d)$, its lift to a cone direction network
$(\Gtilde, \varphi, \tilde{\vec d})$ is well-defined, and the realization spaces of $(G,\bgamma,\vec d)$ and
$(\Gtilde, \varphi, \tilde{\vec d})$ are canonically isomorphic and hence of the same dimension.
\end{lemma}
In light of \lemref{cone-direction-network-lifts}, we can move back and forth between the two settings freely.
In our proofs, we will start with a colored direction network and study the dimension of its realization space
via geometric arguments involving the lift.  This next lemma is a corollary of \lemref{cone-direction-network-lifts},
but an explicit proof is instructive.
\begin{lemma}\lemlab{cone-direction-network-lifts2}
Let $(G,\bgamma)$ be a colored graph and $(\Gtilde, \varphi)$ its lift.  Assigning a direction
to one representative of each edge orbit under $\varphi$ in $\Gtilde$ gives a well defined
colored direction network $(G,\bgamma,\vec d)$.
\end{lemma}
\begin{proof}
Adopt the indexing scheme for the vertices and edges of $\tilde{G}$ from \secref{colored-graphs}.
If the direction $\vec d$ is assigned to an edge $i_{\gamma}j_{\gamma\cdot\gamma_{ij}}\in E(\tilde{G})$,
we assign the direction $\gamma_{ij}^{-1}\cdot \vec d$ to $ij\in E(G)$.  Since we assign a direction to
only on edge in the fiber over $ij$, this procedure gives a well-defined assignment of directions to
the edges of $G$, and it is easy to check that lifting these directions agrees with the assignments made to $\tilde{G}$.
\end{proof}

\subsection{Collapsed and faithful realizations}
The realization space of a cone direction network is never empty: it is always possible to
put all the points $\vec p_i$ at the origin, in which case the realization equations are
trivially satisfied.  We define such realizations to be \emph{collapsed}.  Similarly,
if $ij\in E(\tilde{G})$ is an edge and a realization sets $\vec p_i=\vec p_j$, we
define the edge $ij$ to be \emph{collapsed} in that realization.

For the purposes of rigidity theory, collapsed realizations are degenerate.  We define
a realization to be \emph{faithful} if it has no collapsed edges.  Thus, the content
of \theoref{cone-direction} is that cone-Laman graphs are, generically,  the maximal
colored graphs underlying direction networks with faithful realizations.

\subsection{Proof strategy for \theoref{cone-direction}}\seclab{cone-direction-sketch}
We deduce \theoref{cone-direction} from the following ``collapsing lemma''.
\begin{prop}[\conecollapseprop]\proplab{cone-collapse}
A generic cone direction network that has as its colored quotient graph a cone-$(2,2)$ graph
has only collapsed realizations.
\end{prop}
Given \propref{cone-collapse}, the proof of \theoref{cone-direction} uses an ``edge doubling trick''
employed to prove the analogous \cite[Theorem B]{MT10}:
\begin{itemize}
\item We start with a generic cone direction network with an underlying cone-Laman graph.  This has a one-dimensional
realization space.
\item We then observe that if there is a collapsed edge, the realization space is equivalent to that coming from a
generic direction network on the same graph with a doubled edge, which is cone-$(2,2)$.
\item \propref{cone-collapse} then says the realization space is, in fact, zero dimensional, which is a contradiction.
\end{itemize}
Although these steps, which are carried out in \secref{cone-proof},
require some technical care, they are straightforward. Most of the work is involved in proving
\propref{cone-collapse}.  Since the variables in the realization system for a
colored direction network with $\Z/k\Z$ symmetry, do not separate for $k=3,4,6$,
as in the finite \cite{ST10} or periodic \cite{MT10} cases, we make a geometric argument as opposed to using
the Laplace expansion as is done in \cite{ST10,MT10}.
The approach is as follows:
\begin{itemize}
\item We start with a cone-$(2,2)$ graph, and decompose it into two edge-disjoint cone-$(1,1)$
graphs $X$ and $Y$, which is allowed by the combinatorial \lemref{cone-22-decomp} and select base vertices.
\item We then assign a direction to each connected component of $X$ and $Y$ that forces any realization to
have a specific structure that is only possible in collapsed realizations.
\end{itemize}
These steps are carried out in Sections \ref{sec:geom} and \ref{sec:cone-collapse}.

\section{Generic linear projections} \seclab{geom}
For the proof of \propref{cone-collapse} in the next section, we will need several geometric lemmas.

\subsection{Affine lines}
Given a unit vector $\vec v\in \R^2$
and a scalar $s\in \R$, we denote by $\ell(\vec v,s)$ the affine line
\[
\iprod{\vec p}{\vec v^\perp} = s
\]

\subsection{An important linear equation}
The following is a key lemma which will determine where certain points must lie when solving a cone direction network.
\begin{lemma}\lemlab{rotation}  Suppose $R$ is a non-trivial rotation about the origin, $\vec v^*$ is a unit vector and $\vec p$ satisfies
$$(R-I) \vec p = \lambda \vec v^*$$
for some $\lambda \in \R$.  Then, for some $C \in \R$, we have $\vec p = C \vec v$
where $\vec v = R_{\pi/2} R^{-1/2} \vec v^*$, $R^{-1/2}$ is some square root of $R^{-1}$, and $R_{\pi/2}$ is the counter-clockwise rotation through angle $\pi/2$.
\end{lemma}
\begin{proof}
A computation shows that $(R - I) R^{-1/2} = R^{1/2}-R^{-1/2}$ is a multiple of $R_{\pi/2}$, from which the
Lemma follows.
\end{proof}

\subsection{The linear projection $T(\vec v,\vec w,R)$}
Let $k\in N$ be at least three,
$\vec v$ and $\vec w$ be unit vectors in $\R^2$, and $R$ some nontrivial rotation.  Denote
by $\vec v^*$ the vector $(R^{1/2} \cdot\vec v)^\perp$ for some choice of square root of $R$.

We define $T(\vec v,\vec w,R)$ to be the linear projection from $\ell(\vec v,0)$ to $\ell(\vec w,0)$
in the direction $\vec v^*$.  The following properties of $T(\vec v,\vec w,R)$ are straightforward.
\begin{lemma}\lemlab{Tvw-not-zero}
Let $\vec v$ and $\vec w$ be unit vectors, and $R$ a nontrivial rotation.  Then, the linear
map $T(\vec v,\vec w,R)$:
\begin{itemize}
\item Is defined if $\vec v^*$ is not in the same direction as $\vec w$.
\item Is identically zero if $\vec v^*$ and $\vec v$ are collinear.
\item Is otherwise never zero.
\end{itemize}
\end{lemma}

\subsection{The scale factor of $T(\vec v,\vec w,R)$}
The image $T(\vec v,\vec w,R)\cdot \vec v$ is equal to $\lambda\vec w$, for some scalar $\lambda$.
We define the \emph{scale factor} $\lambda(\vec v,\vec w,R)$ to be this $\lambda$.

We then need two elementary facts about the scaling factor of $T(\vec v,\vec w,R)$.  First, it is either
identically zero or depends rationally on $\vec v$ and $\vec w$.
\begin{lemma} \lemlab{scalefactor-poly}
Let $\vec v$ and $\vec w$ be unit vectors such that $\vec v^*$ and $\vec w$ are linearly independent.  Then
the scaling factor $\lambda(\vec v,\vec w,R)$ of the linear map $T(\vec v,\vec w, R)$ is
given by
\[
\frac{ \langle \vec v, (\vec v*)^\perp \rangle}{\langle \vec w, (\vec v*)^\perp\rangle}
\]
\end{lemma}
\begin{proof}
The map $T(\vec v, \vec w, R)$ is equivalent to the composition of:
\begin{itemize}
\item perpendicular projection from $\ell(\vec v, 0)$ to $\ell((\vec v*)^\perp, 0)$, followed by
\item the inverse of perpendicular projection $\ell(\vec w, 0) \to \ell((\vec v^*)^\perp, 0)$.
\end{itemize}
The first map scales the length of vectors by $\langle \vec v, (\vec v*)^\perp \rangle$ and the second by $\langle \vec w, (\vec v*)^\perp\rangle$.
\end{proof}

From \lemref{scalefactor-poly} it is immediate that
\begin{lemma}\lemlab{scalefactor-blowup}
The scaling factor $\lambda(v, w, R)$ is identically $0$ precisely when $R$ is an order two rotation.
If $R$ is not an order $2$ rotation, then $\lambda(\vec v, \vec w, R)$ approaches infinity as $\vec v^*$
approaches $\pm \vec w$.
\end{lemma}

\subsection{Generic sequences of the map $T(\vec v,\vec w,R)$}
Let $\vec v_1,\vec v_2,\ldots, \vec v_n$ be unit vectors, and $S_1,S_2,\ldots, S_n$ be rotations of the form $R^i_k$ where
$R_k$ is a rotation of order $k$.
We define the linear
map $T(\vec v_1,S_1,\vec v_2,S_2,\ldots, \vec v_n,S_n)$ to be
\[
T(\vec v_1,S_1,\vec v_2,S_2,\ldots, \vec v_n,S_n) =
T(\vec v_n,\vec v_1,S_n) \circ T(\vec v_{n-1},\vec v_n,S_{n-1}) \circ\cdots \circ T(\vec v_1,\vec v_2,S_1)
\]
This next proposition plays a key role in the next section, where it is interpreted
as providing a genericity condition for cone direction networks.
\begin{prop}\proplab{cone-genericity}
Let $\vec v_1,\vec v_2,\ldots, \vec v_n$ be pairwise linearly independent unit vectors, and
$S_1,S_2,\ldots, S_n$ be rotations of the form $R^i_k$.
Then if the $\vec v_i$ avoid a proper algebraic subset
of $\left(\mathbb{S}^1\right)^n$ (that depends on the $S_i$),
the map $T(\vec v_1,S_1,\vec v_2,S_2,\ldots, \vec v_n,S_n)$ scales the length of
vectors by a factor of $\lambda\neq 1$.
\end{prop}
\begin{proof}
If any of the $T(\vec v_i,\vec v_{i+1},S_i)$ are identically zero, we are done, so we assume none of them
are.  The map $T(\vec v_1,S_1,\vec v_2,S_2,\ldots, \vec v_n,S_n)$ then scales vectors by a factor of:
\[
\lambda(\vec v_1,\vec v_2,S_1)\cdot\lambda(\vec v_2,\vec v_3,S_2)\cdots\lambda(\vec v_{n-1},\vec v_n,S_{n-1})\cdot\lambda(\vec v_n,\vec v_1,S_n)
\]
which we denote $\lambda$.  That $\lambda$ is constantly one is a polynomial statement in the $\vec v_i$ by
\lemref{scalefactor-poly}, and so it is either always true or holds only on a proper algebraic subset of
$\left(\mathbb{S}^1\right)^n$.  This means it suffices to prove that there is one selection for the
$\vec v_i$ where $\lambda\neq 1$.  We will show that $|\lambda|$ can be made arbitrarily large, which implies that,
in particular, it is not a constant.

Select the $\vec v_i$ so that the projection $T(\vec v_i,\vec v_j,S_i)$ is defined for all $i$ and $j$.
We hold $\vec v_2,\ldots,\vec v_n$ fixed and consider what happens as we change $\vec v_1$.
As we change $\vec v_1$, the contributions to $\lambda$ from all the terms except
$\lambda(\vec v_1,\vec v_2,S_1)$ and $\lambda(\vec v_n,\vec v_1,S_n)$ are fixed,
so their contribution to $\lambda$ is a constant as $\vec v_1$ changes.

Here is the key observation:
in a small neighborhood of the direction that makes $\vec v^*_1=\vec v_2$, $\lambda(\vec v_n,\vec v_1)$
is uniformly bounded, since $\vec v_n^*$ is bounded away from $\vec v_1$ by our initial choice of the $\vec v_i$.
On the other hand, by \lemref{scalefactor-blowup}, $\lambda(\vec v_1,\vec v_2)$
is unbounded on the same neighborhood, and thus, $|\lambda|$ can be made arbitrarily large.
\end{proof}

\section{Direction networks on cone-$(2,2)$ graphs collapse}\seclab{cone-collapse}
In this section, we prove:
\conecollapseprop
The organization follows the outline given in \secref{cone-direction-sketch}.  We separate the proof
into two major cases: order $k\ge 3$ rotations (\secref{cone-collapse-ge3}) and order two rotations
(\secref{cone-collapse-2}).  Both make use of the results from \secref{cone-collapse-directions}
which relate the combinatorics of cone-$(1,1)$ graphs to the geometric lemmas of \secref{geom}.

\subsection{Genericity}\seclab{cone-collapse-genericity}
The meaning of \emph{generic} in the statement of \propref{cone-collapse} is the standard
one from algebraic geometry: the set of direction assignments for which the proposition
fails to hold is a proper algebraic subset of the space of direction assignments.

Because the system \eqref{colored-cone-system} is square and homogenous, the only
solutions are all zero if and only if \eqref{colored-cone-system} has full rank,
which is a polynomial statement in the given directions $\vec d_{ij}$.  It then
follows that if we can construct \emph{one} set of directions for which all
realizations are collapsed, the same statement is true generically.  Moreover,
in this case, it is easy to describe the non-generic set: it is the set of directions
for which the formal determinant of \eqref{colored-cone-system} vanishes.

The rest of this section, then, is occupied with showing such directions exist.

\subsection{Assigning directions for map-graphs}\seclab{cone-collapse-directions}
Let $(G,\bgamma)$ be a $\Z/k\Z$-colored graph that is a connected cone-$(1,1)$
graph.  Recall from \secref{cone-sparse} that this means that $G$ is a tree plus
one edge and that the unique cycle in $G$ has non-trivial image under the map $\rho$.
We also select and fix a base vertex $b\in V(G)$ that is on the cycle.

The next lemma shows that we can assign directions to the edges of $G$ so that
the realization of the resulting direction network all have the structure similar to that
shown in \figref{cone11-directions}.  This will be the main ``gadget'' that we
use in the proof of \propref{cone-collapse} below.
\begin{figure}[htbp]
\centering
\includegraphics[width=.35\textwidth]{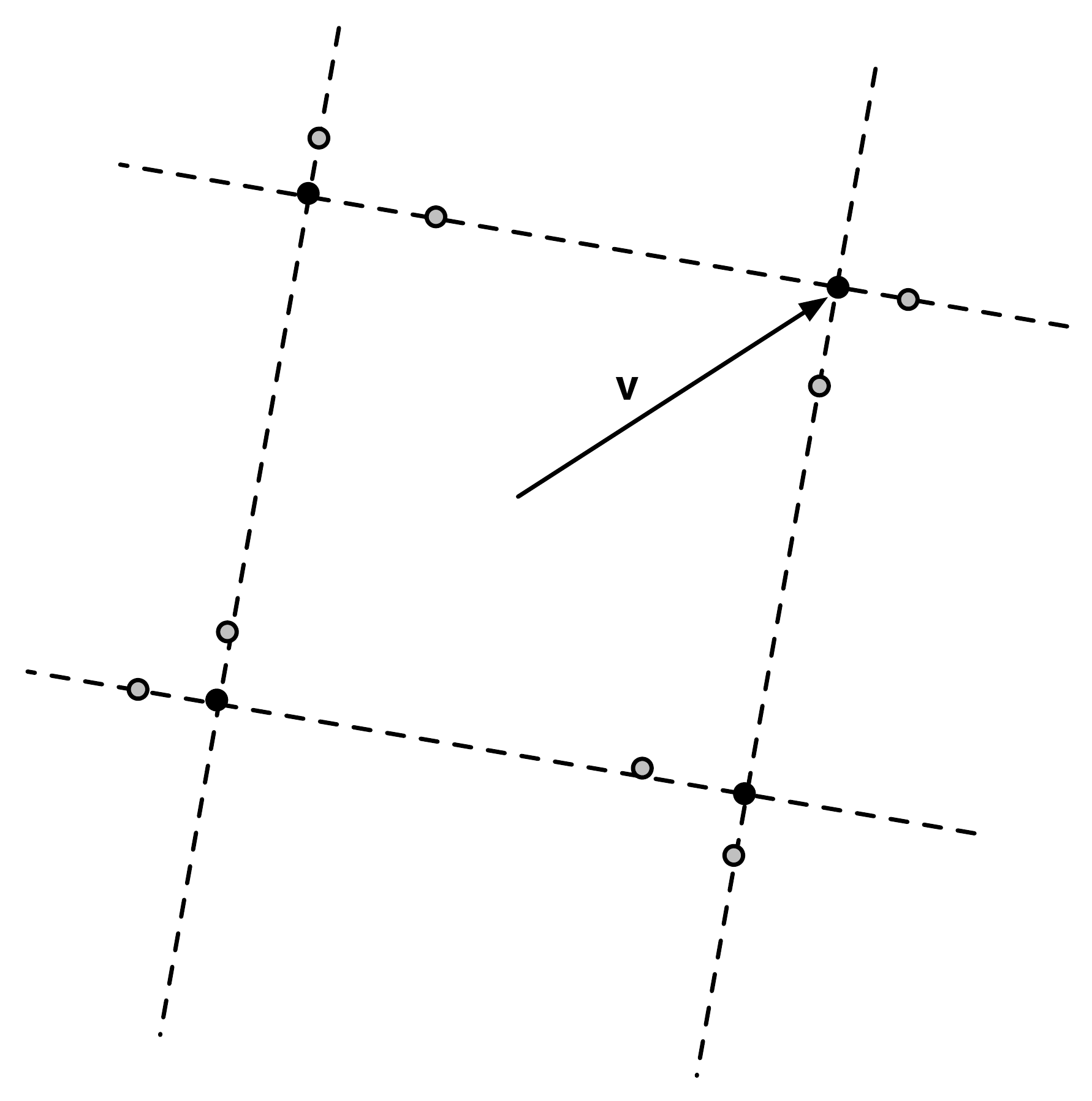}
\caption{
The structure of the realization of a cone-$(1,1)$ graph provided by \lemref{cone11-directions} when $k=4$
and the order of the rotation carried by the cycle is $4$ ($\gamma=1$ in the notation of \lemref{cone11-directions}).
Every vertex lies on one of the dashed lines, which are determined by the order $4$
rotational symmetry and the vector $\vec v$. The fiber over the base vertex (black) lies at the
crossings.
}
\label{fig:cone11-directions}
\end{figure}
\begin{lemma}\lemlab{cone11-directions}
Let $k=3,4,6$, let $(G,\bgamma)$ be a $\Z/k\Z$-colored graph that is a connected cone-$(1,1)$ graph with a base vertex $b$
on the unique cycle in $G$.  Let $\gamma\in \Z/k\Z$ be the $\rho$-image of the cycle in $G$, and let
$\vec v$ be a unit vector.  We can assign directions $\vec d$ to the edges of
$G$ so that, in all realizations of the resulting cone direction network that is the lift of $(G,\bgamma,\vec d)$:
\begin{itemize}
\item The directions from the origin to the points
realizing the fiber over the base vertex $b$ are in directions $R^j_k\cdot \vec v$
\item The rest of the points all lie on the lines between $\vec p_{i \cdot b}$ and
$\vec p_{(i+\gamma)\cdot b}$ as $i$ ranges over $\Z/k\Z$.
\end{itemize}
\end{lemma}
\begin{proof}
By \lemref{cone-direction-network-lifts} and \lemref{cone-direction-network-lifts2}, we can just assign
directions in the lift $\tilde{G}$ of $G$.  We start by selecting an edge $bi\in E(G)$ that is:
\begin{itemize}
\item Incident on the base vertex $b$
\item On the cycle in $G$
\end{itemize}
Such an edge exists because $G$ is a map-graph and $b$ is on the cycle.  $G-bi$ is a spanning tree $T$ of
$G$.

Since $T$ is contractible, it lifts to $k$ disjoint copies of itself in $\tilde{G}$.  Select
one of these copies and denote it $\tilde{T}$.  Note that $\tilde{T}$ hits the fiber
over every edge in $G$ except for $bi$ exactly one time and the fiber
over every vertex exactly one time.

Assign every edge in $\tilde{T}$ the direction $\vec v^*=(R^{\gamma/2}_k\cdot \vec v)^\perp$.  By
\lemref{cone-direction-network-lifts2} this assigns a direction to every edge in
$G$ except for $bi$. From the observations above, it now follows by the connectivity of $T$
that in any realization of the cone direction network induced on $\tilde{G}-\pi^{-1}(ij)$ any
point lies on the $R_k$-orbit of a single affine line in the direction of $\vec v^*$.

Now select the edge in the fiber over $bi$ incident on the copy of $i$ in $\tilde{T}$.  Assign this
edge the direction $\vec v^*$ as well.  Let the set of directions induced on $G$ be $\vec d$.

Denote by $\vec p_b$ the realization of the copy of $b$ in $\tilde{T}$ in a realization of
$(\tilde{G},\varphi,\tilde{\vec d})$.  As we have noted above, the realization of every vertex of
$\tilde{G}$ is on one of the lines $\ell(R^i_k\cdot \vec v^*, s)$ where $s$ is determined by $p_b$.
The selection of direction for the edge $bi$ further forces that if $\vec p_b$ is in the fiber over $b$,
that $R\cdot \vec p_b - \vec p_b$  is in the direction $\vec v$.

It now follows from \lemref{rotation}, applied to a rotation of the same order as $R^\gamma$,
that $\vec p_b$ is in the direction $\vec v$, which is what we wanted.
\end{proof}

\subsection{Proof of \propref{cone-collapse} for order $2$ rotations}\seclab{cone-collapse-2}
Decompose the cone-$(2,2)$ graph $(G,\bgamma)$ into two edge-disjoint cone-$(1,1)$ graphs $X$ and $Y$.
The order of the $\rho$-image of any cycle in either $X$ or $Y$ is always $2$, so the construction
of \lemref{cone11-directions} implies that by assigning the same direction $\vec v$ to every
edge in $X$ every vertex in any realization lies on a single line through the origin in the direction of $\vec v$.
Similarly for edges in $Y$ in a direction $\vec w$ different than $\vec v$.

Since every vertex is at the intersection of two skew lines through the origin, the proposition is proved.
\eop

\subsection{Proof of \propref{cone-collapse} for rotations of order $k\ge 3$}\seclab{cone-collapse-ge3}
In what follows we
let $(G,\bgamma)$ be a cone-$(2,2)$ graph on $n$ vertices.  We fix a decomposition of
$(G,\bgamma)$ into two cone-$(1,1)$ graphs.  This is possible by \lemref{cone-22-decomp}.

Let $G_i$ be the connected components of the two cone-$(1,1)$ graphs.  Which part of the
decomposition $G_i$ comes from is not important in what follows, so we suppress it in the notation.
All the information in the decomposition we need comes from the overlap graph, defined in  \secref{sparse-prelim}.
Select a base vertex $b_i$ on the cycle of each of the $G_i$.  Let $D$ be the overlap graph of the decomposition,
and index the vertex set of $D$ by $B$, the collection of $b_i$.

\paragraph{Assigning directions}
For each $G_i$, select a unit vector $\vec v_i$ such that:
\begin{itemize}
\item The $\vec v_i$ are all different.
\item Any subset of the $\vec v_i$ are generic in the sense of \propref{cone-genericity}.
\end{itemize}
This is possible, since \propref{cone-genericity} rules out only a measure zero
subset of directions for each sub-collection.

Now, for each $G_i$ we assign, in the notation of \lemref{cone11-directions}, the
direction $\vec v^*_i$ to the edges in $G_i$ as prescribed by \lemref{cone11-directions}. This is
well-defined, since the $G_i$ partition the edges of $G$.  (They clearly overlap on the vertices--we will exploit
this fact below---but it does not prevent us from assigning edge directions independently.)

We define the resulting colored direction network to be $(G,\bgamma,\vec d)$ and the lifted
cone direction network $(\tilde{G},\varphi,\tilde{\vec d})$.  We
also define, as a convenience, the rotation $S_i$ as the rotation such that $\vec v_i = (S_i^{1/2} \vec v^*_i)^\perp$ to be $S_i$.

\paragraph{Local structure of realizations}
Let $G_i$ and $G_j$ be distinct connected cone-$(1,1)$ components and suppose
that there is a directed edge $b_ib_j$ in the overlap graph $D$.  We have the following
relationship between $\vec p_{b_i}$ and $\vec p_{b_j}$ in realizations of
$(\tilde{G},\varphi,\tilde{\vec d})$.
\begin{lemma}\lemlab{cone-local-structure}
Let $\tilde{G}(\vec p)$ be a realization of the cone direction network $(\tilde{G},\varphi,\tilde{\vec d})$
defined above.  Let vertices $b_i$ and $b_j$ in $V(G)$ be the base vertices of $G_i$ and $G_j$, and
suppose that $b_ib_j$ is a directed edge in the overlap graph $D$.  Let $\vec p_{\gamma\cdot \tilde b_i}$
be the realization of any vertex in the fiber over $b_i$ in $V(\tilde{G})$.  Then there is a
vertex $\gamma'\cdot \tilde{b_j}$ in the fiber over $b_j$ such that
$\vec p_{\gamma'\cdot \tilde b_j} = T(\vec v_i,\vec v_j, S_i)\cdot \vec p_{\gamma \cdot \tilde b_i}$
\end{lemma}
The proof is illustrated in \figref{cone-local-structure-proof}.
\begin{figure}[htbp]
\centering
\includegraphics[width=.35\textwidth]{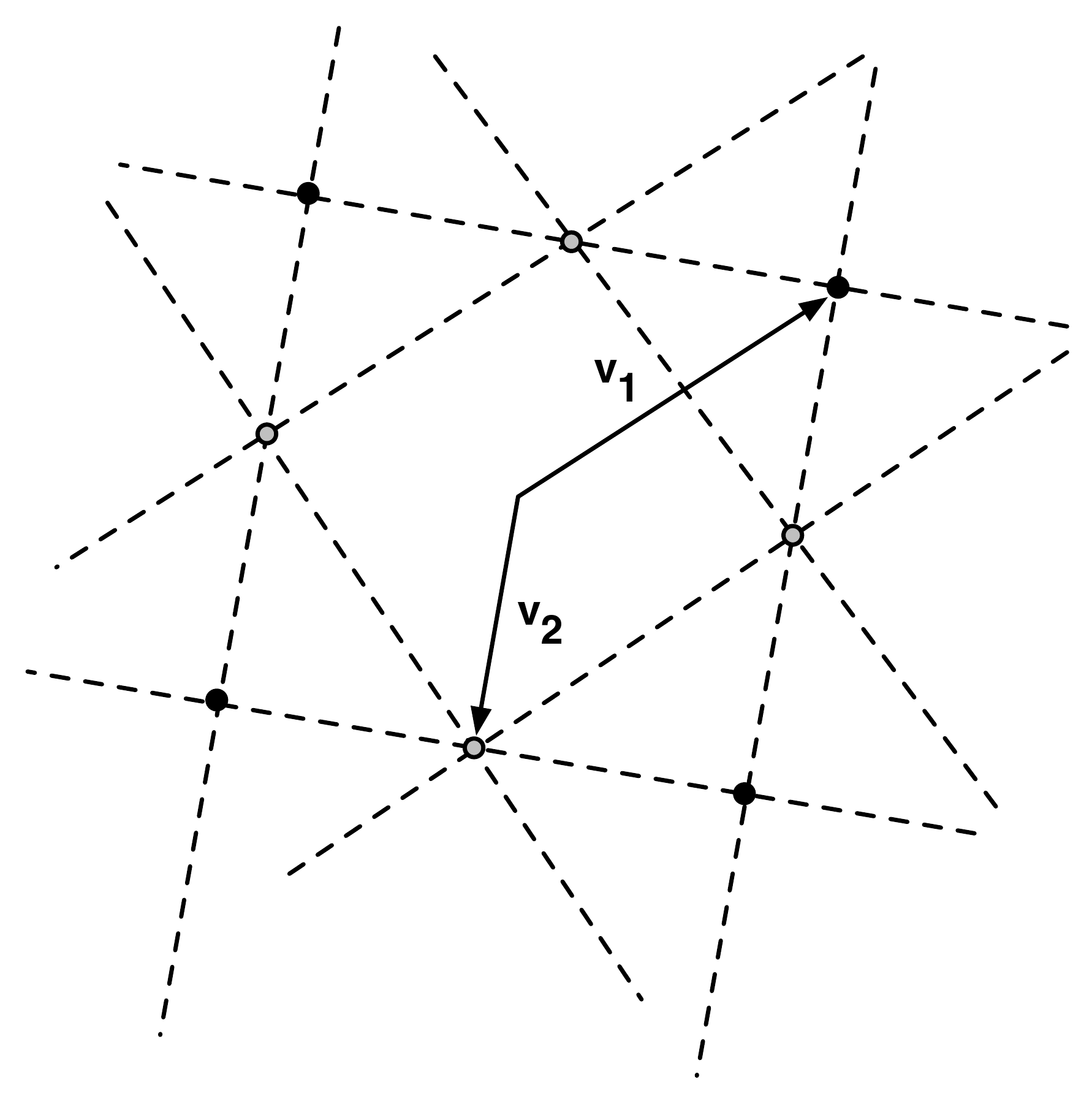}
\caption{
Example of the local structure of the proof of \propref{cone-collapse}: the directions we assign
imply that if $b_1b_2$ is an edge in the overlap graph, then the base vertex of $G_1$
can be obtained from the base vertex of $G_2$ via the linear projection $T(\vec v_1,\vec v_2, S_1)$.
}
\label{fig:cone-local-structure-proof}
\end{figure}
\begin{proof}
By \lemref{cone11-directions}, the fiber over every vertex in $G_i$ lies on a line $\ell(R^{t}_{k}\cdot \vec v_i^*,s)$
for some scalar $s_i$.  The scalar $s_i$ is determined by location of any $\vec p_{b_i}$ representing a
vertex in the fiber over $b_i$, since the $\vec p_{- \cdot b_i}$ are all equal to some multiple of $R_k^\gamma\cdot \vec v_i$.

In particular, the vertices in the fiber over $b_j$ are on these lines. Additionally, \lemref{cone11-directions},
applied to $G_j$, says that the vertices in the fiber over $b_j$ are all equal to some scalar multiple of
$R^\gamma_k\cdot \vec v_j$.  This is exactly the situation captured by the map $T(\vec v_i,\vec v_j, S_i)$.
\end{proof}

\paragraph{Base vertices on cycles in $D$ must be at the origin}
Let $b_i$ be the base vertex in $G_i$ that is also on a directed cycle in $D$.  The next step in the
proof is to show that all representatives in
$b_i$ must be mapped to the origin in any realization of $(\tilde{G},\varphi,\tilde{\vec d})$.
\begin{lemma}\lemlab{cone-cycle-vertices}
Let $\tilde{G}(\vec p)$ be a realization of the cone direction network $(\tilde{G},\varphi,\tilde{\vec d})$
defined above, and let $b_i\in V(G)$ be a base vertex that is also on a directed cycle in $D$ (one exists
by \propref{overlap-cycle}).  Then all vertices in the fiber over $b_i$ must be mapped to the origin.
\end{lemma}
\begin{proof}
Iterated application of \lemref{cone-local-structure} along the cycle in $D$ that $b_i$ is on tells
us that any vertex in the fiber over $b_i$ is related to another vertex in the same fiber by a linear
map meeting the hypothesis of \propref{cone-genericity}.  This implies that if any vertex in the
fiber over $b_i$ is mapped to a point not the origin, some other vertex in the same fiber is mapped to a
point at a different distance to the origin.  This is a contradiction, since all realizations $\tilde{G}(\vec p)$
are symmetric with respect to $R_k$, so, in fact the fiber over $b_i$ was mapped to the origin.
\end{proof}

\paragraph{All base vertices must be at the origin}
So far we have shown that every base vertex $b_i$ that is on a directed cycle in the overlap graph $D$
is mapped to the origin in any realization $\tilde{G}(\vec p)$ of $(\tilde{G},\varphi,\tilde{\vec d})$.
However, since every base vertex is connected to the cycle in its connected component by a directed path in $D$,
we can show all the base vertices are at the origin.

\begin{lemma}\lemlab{cone-non-cycle-vertices}
Let $\tilde{G}(\vec p)$ be a realization of the cone direction network $(\tilde{G},\varphi,\tilde{\vec d})$
defined above.  Then all vertices in the fiber over $b_i$ must be mapped to the origin.
\end{lemma}
\begin{proof}
The statement is already proved for base vertices on a directed cycle in \lemref{cone-cycle-vertices}.
Any base vertex not on a directed cycle, say $b_i$, is at the end of a directed path which starts
at a vertex on the directed cycle.  Thus $\vec p_{\gamma \cdot b_i}$ is the image of $0$ under
some linear map and hence is at the origin.
\end{proof}

\paragraph{All vertices must be at the origin}
The proof of \propref{cone-collapse} then follows from the observation that, if all the base vertices $b_i$
must be mapped to the origin in $\tilde{G}(\vec p)$, then \lemref{cone11-directions} implies that \emph{every}
vertex in the lift of $G_i$ lies on a family of $k$ lines intersecting at the origin.  (This is the
degenerate form of \figref{cone11-directions} where the base vertex is at the origin.)

Since every vertex is in the span of two of the $G_i$, and these families of lines
intersect only at the origin, we are done: $\tilde{G}(\vec p)$ must put all the points at the origin.
\eop

\section{Proof of \theoref{cone-direction}}\seclab{cone-proof}
This section gives the proof of:
\conedirection

\subsection{Generic rank of the colored realization system} \propref{cone-collapse} is a
geometric statement, but it has an algebraic interpretation.  In matroidal language,
this next lemma says that, in matrix form, the system \eqref{colored-cone-system} is a
generic linear representation for the cone-$(2,2)$ matroid.
\begin{lemma}\lemlab{cone22-rep}
Let $(G,\bgamma)$ be a $\Z/k\Z$-colored graph with $n$ vertices and $m$ edges.  Then,
if $(G,\bgamma)$ is cone-$(2,2)$ sparse,
the generic rank of the system \eqref{colored-cone-system} is $m$.
\end{lemma}
\begin{proof}
If $(G,\bgamma)$ is cone-$(2,2)$ sparse, the matroid property of cone-$(2,2)$
graphs (\lemref{cone-22-matroid}) implies that it can be extended to a cone-$(2,2)$ graph
$(G',\bgamma')$.
Form a generic direction network on $(G',\bgamma')$.  By \propref{cone-collapse}, the colored
realization system for this extended direction network has rank $2n$, so it follows that all $m$
of the equations in its restriction to $(G,\bgamma)$ are independent.
\end{proof}
In particular, since cone-Laman graphs are cone-$(2,2)$ sparse, we get:
\begin{lemma}\lemlab{cone22-rep-laman}
Let $(G,\bgamma)$ be a cone-Laman graph with $\Z/k\Z$ colors.  Then the
generic rank of the system \eqref{colored-cone-system} is $2n-1$.
\end{lemma}

\subsection{Proof of \theoref{cone-direction}}
We prove each direction of the statement in turn.  Since it is technically
easier, we prove the equivalent statement on colored direction networks.
The theorem then follows by \lemref{cone-direction-network-lifts}.

\paragraph{Cone-laman graphs generically have faithful realizations}
Let $(G,\bgamma)$ be a cone-Laman graph with $n$ vertices.  \lemref{cone-laman-doubling} implies that doubling any edge $ij$ of
$(G,\bgamma)$ results in a cone-$(2,2)$ graph $(G+ij,\bgamma)$.
Select edge directions for $G$ such that, for every $(G+ij,\bgamma)$ obtained from $G$ by doubling an edge,
these directions on the edges of $G$, plus some direction on the added copy of the edge $ij$ yield
directions on $(G+ij,\bgamma)$ that are generic in the sense of \propref{cone-collapse}.  This is possible,
since the desired directions lie in the intersection of a finite number of full measure subsets of the space of direction
assignments.

Define $(G,\bgamma,\vec d)$ to be the colored direction network with these directions.  By \lemref{cone22-rep}, the
realization space is $1$-dimensional, since the system \eqref{colored-cone-system} has rank $2n-1$.  There
must be some realization of $G$ that is not entirely collapsed: $G$ is connected by \lemref{cone-laman-connected}
and, since the realization space is $1$-dimensional, it allows non-trivial scalings.  Together these facts imply that
some edge is not realized with coincident endpoints.

Now we suppose, for a contradiction, that some edge $ij$ is collapsed in a non-collapsed realization $G(\vec p)$ of $(G,\bgamma,\vec d)$.
Because the realization space is one-dimensional, all other realizations are scalings of $G(\vec p)$, which implies that
$ij$ is collapsed in \emph{all} realizations of $(G,\bgamma,\vec d)$.  Adding a second copy of the colored edge $ij$ and
giving it a different direction forces $ij$ to be collapsed in all realizations, and so we see that the realization space
of $(G,\bgamma,\vec d)$ is exactly the same as that of $(G+ij,\bgamma,\vec d)$.

We are now at a contradiction. The directions $\vec d$ are chosen such that $(G+ij,\bgamma,\vec d)$ is generic in the sense
of \propref{cone-collapse}, and this implies that $(G+ij,\bgamma,\vec d)$, and so $(G,\bgamma,\vec d)$ has a zero-dimensional
realization space.  Since $(G,\bgamma,\vec d)$ always has at least a one-dimensional realization space, we are done.

\paragraph{Cone-laman circuits have collapsed edges}
For the other direction, we suppose that $(G,\bgamma)$ has $n$ vertices and is not cone-Laman.  If the number of edges $m$
is less than $2n-1$, then the realization space of any direction network is at least $2$-dimensional, so
it contains more than just rescalings.  Thus, we assume that $m\ge 2n - 1$.  In this case, $G$ is not
cone-Laman sparse, so it contains a cone-Laman circuit $(G',\bgamma)$.  Thus, we are reduced to showing that
any cone-Laman circuit has, generically, only realizations with collapsed edges, since these then force collapsed
edges in any realization of a  generic colored direction network on $(G,\bgamma)$.

We recall that \lemref{cone-laman-circuits-struct} says there are two types of cone-Laman circuits $(G',\bgamma)$:
\begin{itemize}
\item $(G',\bgamma)$ is a cone-$(2,2)$ graph.
\item $(G',\bgamma)$ has $T(G') = 2$, $n'$ vertices, $m' = 2n'-2$ edges, and is cone-$(2,2)$ sparse.
\end{itemize}
If $(G',\bgamma)$ is a cone-$(2,2)$ graph, then \propref{cone-collapse} applies to it, and we are done.  For the
other type of cone-$(2,2)$ circuit, \lemref{cone22-rep} implies that,
for generic directions, a direction network has rank $2n' - 2$, so the realization
space is two-dimensional.  Because $T(G') = 2$, the $\rho$-image of $(G',\bgamma)$ is trivial, so $G'$
lifts to $k$ disjoint copies of itself.  We can construct realizations of the lifted direction network
by picking one of these copies of $G'$ in the lift as representatives and putting the vertices
on top of each other at an arbitrary point in the plane.  Since this is a $2$-dimension family of realizations with
all edges collapsed, this family is, in fact the entire realization space, completing the proof.
\eop

\section{Crystallographic direction networks} \seclab{cdns}
Let $(\tilde{G},\varphi)$ be a graph with a $\Gamma_k$-action $\varphi$.  A \emph{crystallographic direction network}
$(\tilde{G},\varphi,\tilde{\vec d})$ is given by, $(\tilde{G},\varphi)$ and an assignment of
a direction $\tilde{\vec d}_{ij}$ to each edge $ij\in E(\tilde{G})$.

\subsection{The realization problem}
A \emph{realization} of a crystallographic direction network is given by a point set $\vec p=(\vec p_i)_{i\in V(\tilde{G})}$
and a representation $\Phi$ of $\Gamma_k$ such that:
\begin{eqnarray}
\iprod{\vec p_j - \vec p_i}{\tilde{\vec d}^\perp_{ij}} = 0 & \text{for all edges $ij\in E(\tilde{G})$} \\
\vec p_{\gamma\cdot i} = \Phi(\gamma)\cdot \vec p_i & \text{for all vertices $i\in V(\tilde{G})$}
\end{eqnarray}
We observe that for a crystallographic direction networks to be realizable, the $\Gamma_k$-orbit of any
edge needs to be given the $\phi$-equivariant directions; i.e. if $i'j' = \gamma \cdot ij$, then
$\vec d_{i'j'}$ is $\vec d_{ij}$ rotated by the rotational part of $\gamma$.  From now on we require $\phi$-equivariance
of directions.  We denote realizations by
$\tilde{G}(\vec p,\Phi)$, to indicate the dependence on $\Phi$.

The definition of \emph{collapsed} and \emph{faithful} realizations is similar to that for cone direction
networks.  An edge $ij$ is collapsed in a realization $\tilde{G}(\vec p,\Phi)$ if $\vec p_i = \vec p_j$.
A \emph{realization is collapsed} when all the edges are collapsed \emph{and} $\Phi$ is trivial.  A representation
is \emph{trivial} if it maps the translation generators of $\Gamma_k$ to the zero vector.  A realization
is \emph{faithful} if no edge is collapsed and $\Phi$ is not trivial.

Although our techniques don't require it in this section, for convenience, we will consider only realizations
that map the rotational generator $r_k$ of $\Gamma_k$ to the rotation around the origin $R_k$ through angle $2\pi/k$.
The dimension of the resulting realization space is two less than the dimension of the realization
space where the rotation center of $\Phi(r_k)$ is not pinned down.

\subsection{Direction Network Theorem}
As in the case of cone direction networks, all the information about a crystallographic direction
network is captured by its colored quotient graph.  We make this precise in \secref{colored-crystal-networks}
below.  Our main theorem on crystallographic direction networks is
\directionthm

\subsection{Proof of \theoref{direction}}  The key proposition,
which is proved in \secref{crystal-collapse} is:
\begin{prop}[\crystalcollapseprop]\proplab{crystal-collapse}
A generic crystallographic direction network that has as its colored quotient graph a $\Gamma$-$(2,2)$ graph
has only collapsed realizations.
\end{prop}
It then follows that:
\begin{prop}[\gammalamancircuitcollapse]\proplab{gamma-laman-circuit-collapse}
A generic crystallographic direction network that has, as its colored quotient graph, a
$\Gamma$-colored-Laman circuit has only realizations with collapsed edges.
\end{prop}
\propref{gamma-laman-circuit-collapse} is proved in \secref{gamma-laman-circuit-collapse}.
The proof of \theoref{cone-direction} then goes through with appropriate modifications. \eop

\subsection{Colored direction networks}\seclab{colored-crystal-networks}  As we did with cone direction networks, we will make use of
\emph{colored crystallographic direction networks} to study crystallographic direction networks.  Since
there is no chance of confusion, we simply call these ``colored direction networks''
in the next several sections.

A \emph{colored direction network} $(G,\bgamma,\vec d)$ is given by a $\Gamma_k$-colored graph $(G,\bgamma)$
and an assignment of a direction $\vec d_{ij}$ to every edge $ij$.  The realization system for
$(G,\bgamma,\vec d)$ is given by
\begin{eqnarray}
\iprod{\Phi(\gamma_{ij})\cdot\vec p_j - \vec p_i}{\vec d_{ij}^\perp} = 0
\label{colored-crystal-directions}
\end{eqnarray}
The unknowns are the representation $\Phi$ of $\Gamma_k$ and the points $\vec p_i$.  We denote points in the
realization space by $G(\vec p,\Phi)$.

The following two lemmas linking crystallographic direction networks and colored direction networks
have the same proofs as Lemmas \ref{lemma:cone-direction-network-lifts} and  \ref{lemma:cone-direction-network-lifts2}
\begin{lemma}\lemlab{crystal-direction-network-lifts}
Given a colored direction network $(G,\bgamma,\vec d)$, its lift to a crystallographic direction network
$(\Gtilde, \varphi, \tilde{\vec d})$ is well-defined and the realization spaces of $(G,\bgamma,\vec d)$ and
$(\Gtilde, \varphi, \tilde{\vec d})$ are isomorphic.  In particular, they have the same dimension.
\end{lemma}
\begin{lemma}\lemlab{crystal-direction-network-lifts2}
Let $(G,\bgamma)$ be a $\Gamma_k$-colored graph and $(\Gtilde, \varphi)$ its lift.  Assigning a direction
to one representative of each edge orbit under $\varphi$ in $\Gtilde$ gives a well defined
colored direction network $(G,\bgamma,\vec d)$.
\end{lemma}
This next lemma, which is also immediate from the definitions, describes collapsed edges in
terms of colored direction networks.
\begin{lemma}\lemlab{colored-collapsed-edge}
Let $(G,\bgamma,\vec d)$ be a colored direction network and let $G(\vec p,\Phi)$ be a realization of $(G,\bgamma,\vec d)$.
Let $(\tilde{G},\varphi,\tilde{\vec d})$ be the lift of $(G,\bgamma,\vec d)$ and $\tilde{G}(\vec p,\Phi)$ be the associated
lift of $G(\vec p,\Phi)$.  Then a colored edge $ij\in E(G)$ lifts to an orbit of collapsed edges
in $\tilde{G}(\vec p,\Phi)$ if and only if
\[
\vec p_i = \Phi(\gamma_{ij})\cdot\vec p_j
\]
in $G(\vec p,\Phi)$.
\end{lemma}
In light of Lemmas \ref{lemma:crystal-direction-network-lifts}--\ref{lemma:colored-collapsed-edge}, we
may switch freely between the formalisms, and we do so in subsequent sections.

\subsection{Proof strategy for \propref{crystal-collapse}}
The main difference between \propref{crystal-collapse} and \propref{cone-collapse} is that
we need to account for the flexibility of $\Phi$.  To do this, we bootstrap the proof using
generalized cone-$(2,2)$ graphs (from \secref{gen-cone-sparse}). The steps are:
\begin{itemize}
\item We show that, for fixed $\Phi$,
a generic direction network on a generalized cone-$(2,2)$ graph
has a unique solution (\propref{gen-cone-unique}).
\item Then we allow $\Phi$ to flex. We show that by adding $\rep(\Trans(\Gamma_k))$ edges
that extend a generalized cone-Laman graph to a $\Gamma$-$(2,2)$ graph, realizations of a generic
direction network are forced to collapse.
\end{itemize}
This is done in Sections \ref{sec:gen-cone22-direction-networks} and \ref{sec:crystal-collapse}.

\section{Direction networks on generalized cone-$(2,2)$ graphs} \seclab{gen-cone22-direction-networks}
Let $(G,\bgamma)$ be a generalized cone-$(2,2)$ graph.  In light of \propref{cone-collapse}, it is intuitive
that the realization system \eqref{colored-crystal-directions} should have generic rank $2n$ for a colored
direction network on $(G,\bgamma)$, since cone direction networks are a ``special case''.  In this section
we verify that intuition by giving the reduction to.
\begin{prop}\proplab{gen-cone-unique}
Fix a faithful representation $\Phi$ of $\Gamma_k$.  Holding $\Phi$ fixed, a generic crystallographic
direction network that has a generalized cone-$(2,2)$ graph as its colored quotient has a unique
solution.
\end{prop}
\propref{gen-cone-unique} is immediate from the following statement and \lemref{crystal-direction-network-lifts}.
\begin{prop}\proplab{gen-cone-rank}
Let $(G,\bgamma)$ be a generalized cone-$(2,2)$ graph with $n$ vertices.
Then the generic rank of the realization system \eqref{colored-crystal-directions} is $2n$.
\end{prop}

\subsection{Proof of \propref{gen-cone-rank}}%
Expanding \eqref{colored-crystal-directions} we get
\begin{equation}
\label{gen-cone-rank-1}
\iprod{\Phi(\gamma_{ij})\cdot\vec p_j - \vec p_i}{\vec d_{ij}^\perp}  =
\iprod{\Phi(\gamma_{ij})\cdot\vec p_j}{\vec d_{ij}^\perp} - \iprod{\vec p_i}{\vec d_{ij}^\perp}
\end{equation}
We define $\Phi(\gamma_{ij})_{r}\in SO(2)$ to be the rotational part of $\Phi(\gamma_{ij})$ and $\Phi(\gamma_{ij})_t\in \R^2$
to be the translational part, so that $\Phi(\gamma_{ij})\cdot \vec p = \Phi(\gamma_{ij})_{r}\cdot \vec p + \Phi(\gamma_{ij})_t$.
In this notation, \eqref{gen-cone-rank-1} becomes
\begin{equation}
\label{gen-cone-rank-2}
\iprod{\Phi(\gamma_{ij})_r\cdot\vec p_j}{\vec d_{ij}^\perp} + 	\iprod{\Phi(\gamma_{ij})_t}{\vec d_{ij}^\perp} -
\iprod{\vec p_i}{\vec d_{ij}^\perp}
\end{equation}
Finally, since the rotational part $\Phi(\gamma_{ij})_r$ preserves the inner product, we see that \eqref{colored-crystal-directions}
is equivalent to the inhomogeneous system
\begin{equation}
\label{gen-cone-rank-3}
\iprod{\vec p_j}{\Phi(\gamma_{ij}^{-1})_r\cdot\vec d_{ij}^\perp}  -
\iprod{\vec p_i}{\vec d_{ij}^\perp} = - 	\iprod{\Phi(\gamma_{ij})_t}{\vec d_{ij}^\perp}
\end{equation}
The l.h.s. of \eqref{gen-cone-rank-3} is the same as \eqref{colored-cone-system}, and thus the generic
rank of \eqref{gen-cone-rank-3} is at least as large as that of  \eqref{colored-cone-system}.  The proposition
then follows from \propref{cone-collapse}.
\eop

\section{Proof of \propref{crystal-collapse}}\seclab{crystal-collapse}
We now have the tools in place to prove:
\crystalcollapseprop
The proof is split into two cases, $\Gamma_2$ and $\Gamma_k$ for $k=3,4,6$.

\subsection{Proof for rotations of order $3$, $4$, and $6$}
Let $(G,\bgamma)$ be a $\Gamma$-$(2,2)$ graph.  We construct a direction network
on $(G,\bgamma)$ that has only collapsed solutions, from which the desired
generic statement follows.

\paragraph{Assigning directions}
We select directions $\vec d$ for each edge in $G$ such that they are generic in the sense of
\propref{gen-cone-unique} for every g.c.-basis of $(G,\bgamma)$.  Define the resulting
colored direction network to be $(G,\bgamma,\vec d)$.

\paragraph{The realization space of any spanning g.c.-basis}
With these direction assignments, we can compute the dimensions
of the realization space for the direction network
induced on any spanning g.c.-basis of $(G,\bgamma)$.  One
exists by \lemref{gamma22-is-cone22-spanning}.

\begin{lemma}\lemlab{collapse-proof-dimension}
Let $(G',\bgamma)$ be a spanning g.c-basis of $(G,\bgamma)$.  Then the
realization space of the induced direction network $(G',\bgamma,\vec d)$
is $2$-dimensional, and linearly depends on the representation $\Phi$.
\end{lemma}
\begin{proof}
The dimension comes from \propref{gen-cone-unique} and comparing the number of
variables to the number of equations in the realization system
\eqref{colored-crystal-directions}.  Moving the variables associated with $\Phi$
to the right completes the proof.
\end{proof}

\paragraph{A g.c.-basis with non-collapsed complement}
Let $(G',\bgamma)$ be a g.c.-basis of $(G,\bgamma)$.  By edge
counts, there are exactly two edges $ij$ and $vw$ in the
complement of $(G',\bgamma)$.

\begin{lemma}\lemlab{collapse-proof-good-basis}
There is a g.c.-basis $(G',\bgamma)$ of $(G,\bgamma)$ such that the edges $ij$ and $vw$
in its complement are non-collapsed in some realization of $(G',\bgamma,\vec d)$.
\end{lemma}
\begin{proof}
By \propref{gamma22-decomp}, we can decompose $(G,\bgamma)$ into two spanning
$\Gamma$-$(1,1)$ graphs $X$ and $Y$. Since $\Gamma$-$(1,1)$ graphs are all
generalized cone-$(1,1)$ graphs plus an edge for $k=3,4,6$, we can assume, w.l.o.g.,
that $X-ij$ and $Y-vw$ are generalized cone-$(1,1)$.  Define $X'$ to be $X-ij$ and
$Y'$ to be $Y-vw$. It follows that $X'\cup Y'$ is a g.c.-basis of $(G,\bgamma)$.

Let $X''$ be the fundamental g.c.-$(1,1)$ circuit of $ij$ in
$X'$.  By \lemref{gc11-circuits}, the $\rho$-image of $X''$
contains a translation.  If every edge in $X''$ is collapsed
in every realization of the direction network $(X'\cup Y',\bgamma, \vec d)$,
this implies that $\Phi$ must always be trivial in any realization.
\propref{gen-cone-unique} would then imply that the realization space
is $0$-dimensional, which is a contradiction to \lemref{collapse-proof-dimension}.

Here is the key step of the argument: since some edge $i'j'$ in $X''$
is not collapsed, we can do a basis exchange (on the g.c.-$(1,1)$ matroid)
to find a g.c.-basis with $i'j'$ and $vw$ in its complement, and $i'j'$
not collapsed in some realization.

We then repeat the argument on $Y'$ and $vw$.  Since this will not affect
$i'j'$, we are done.
\end{proof}
Now we select a g.c.-basis $(G',\bgamma)$ with the property of
\lemref{collapse-proof-good-basis}, and let $G(\vec p,\Phi)$
be a realization in which the edges $ij$ and $vw$ are both
non-collapsed.  The rest of the proof will be to show that, adding back $ij$ and
$vw$ forces all realizations of $(G,\bgamma,\vec d)$ to collapse.

\paragraph{The realization space of $(G'+ij,\bgamma, \vec d)$}
We first consider adding back $ij$.
\begin{lemma}\lemlab{collapse-proof-ij}
The realization space of $(G'+ij,\bgamma, \vec d)$ is $1$-dimensional.
\end{lemma}
\begin{proof}
We know that $ij$ is not collapsed, so the Lemma will follow provided that
the direction of $\vec v = \vec p_j - \vec p_i$ is non-constant in realizations of $(G',\bgamma, \vec d)$
as a function of $\Phi$.  In this case, simply setting $\vec d_{ij} = \vec v$ would impose a new linear
constraint, decreasing the dimension of the realization space by one.

To see that the direction of $\vec v$ is not constant as $\Phi$
varies, observe that assigning a direction $\vec d_{ij}$ other than $\vec v$ to
$ij$ would then force $ij$ to collapse in any realization of $(G'+ij,\bgamma, \vec d)$.
In turn, using the edge swapping argument from \lemref{collapse-proof-good-basis},
the entire g.c.-$(1,1)$ circuit of $ij$ in $X'$ collapses, resulting in
a zero-dimensional realization space.  This contradicts \lemref{collapse-proof-dimension} in
that it implies the realization space of $(G',\bgamma,\vec d)$ was $1$-dimensional.
\end{proof}
In light of \lemref{collapse-proof-ij}, we set the direction $\vec d_{ij}$ to be $\vec v$.
This is allowed, since it preserve the realization $G(\vec p,\Phi)$ we obtained from
\lemref{collapse-proof-good-basis} and $ij$ is, by definition, outside of the g.c-basis
$(G',\bgamma)$.

\paragraph{The representation $\Phi$ must be trivial}
Now we consider adding the edge $vw$ to $(G'+ij,\bgamma,\vec d)$.
\begin{lemma}
The representation $\Phi$ is trivial in any realization of $(G,\bgamma,\vec d)$.
\end{lemma}
\begin{proof}
The realization space of $(G'+ij,\bgamma, \vec d)$ is $1$-dimensional by
\lemref{collapse-proof-ij}, and so it consists only of rescalings of the realization
$G(\vec p,\Phi)$ in which $\vec p_v$ and $\vec p_w$ are distinct.  Setting the
direction $\vec d_{vw}$ to any direction other than that of $\vec p_w - \vec p_v$
then gives the Lemma: the new constraint then forces the edge $vw$ to collapse, and
with it, using the argument used to show \lemref{collapse-proof-good-basis} its
g.c.-$(1,1)$ circuit in $Y'$, and consequently $\Phi$.
\end{proof}

\paragraph{All realizations are collapsed}
The existence of a g.c.-basis of $(G,\bgamma)$ and \propref{gen-cone-unique} guarantee a unique realization
of  $(G,\bgamma, \vec d)$ depending on $\Phi$. When $\Phi$ is trivial, this is the completely
collapsed solution.
\eop

\subsection{Proof for rotations of order $2$}
Let $(G,\bgamma)$ be a $\Gamma$-$(2,2)$ graph.  Again, we will assign directions so that the
resulting direction network $(G,\bgamma, \vec d)$ has only collapsed solutions.  The
proof has a slightly different structure from the $k=3,4,6$ case.  The main
geometric lemma is the following.

\begin{lemma}\lemlab{collapse-proof-G2-gam11}
Let $(X,\bgamma)$ be a $\Gamma$-$(1,1)$ graph with $\Gamma_2$ colors.  Then any
realization $X(\vec p,\Phi)$ of a colored direction network $(X,\bgamma,\vec d)$ that assigns the
same direction $\vec v$ to every edge lifts to a realization $\tilde{X}(\vec p, \Phi)$
such that every vertex lies on a single line in the direction of $\vec v$.
\end{lemma}

\paragraph{\propref{crystal-collapse} for $\Gamma_2$ from \lemref{collapse-proof-G2-gam11}}
With \lemref{collapse-proof-G2-gam11}, the Proposition follows readily:
the combinatorial \propref{gamma22-decomp} says we may decompose $(G,\bgamma)$ into two spanning
$\Gamma$-$(1,1)$ graphs, which we define to be $X$ and $Y$.  We assign
the edges of $X$ a direction $\vec v_X$ and the edges of $Y$ a linearly independent
direction $\vec v_Y$.  Applying \lemref{collapse-proof-G2-gam11}, to $X$ and $Y$ separately
shows that every vertex of a lifted realization $\tilde{G}(\vec p,\Phi)$ must lie in
two skew lines.  This is possible only when they are all at the intersection of these
lines, implying only collapsed realizations. \eop

\paragraph{Proof of \lemref{collapse-proof-G2-gam11}}
Let $(X,\bgamma)$ be a $\Gamma$-$(1,1)$ graph with $\Gamma_2$ colors, and let $(X,\bgamma, \vec d)$
be a direction network that assigns all the edges the same direction. Let $(X',\bgamma)$  be a g.c.-$(1,1)$
basis of $(X,\bgamma)$; one exists by \lemref{gamma11-is-gc11-spanning}.

First we consider one connected component $X''$ of $X'$.
\begin{lemma}\lemlab{collapse-proof-G2-gc11-conn}
Let $(X'',\bgamma,\vec d)$ be a connected g.c.-$(1,1)$ graph, and let $\vec d$
assign the same direction $\vec v$ to every edge.  Then, in any realization of the lifted
crystallographic direction network $(\tilde{X},\varphi,\vec d)$, every vertex
and every edge lies on a line in the direction $\vec v$ through a rotation center.
\end{lemma}
\begin{proof}
We  reason similarly
to the way we did in \secref{cone-collapse-2}.  Because the $\rho$-image of $X''$ contains
an order $2$ rotation $r$, for some vertex $i\in V(X'')$, there is a vertex $\tilde{i}$ in the
fiber over $i$ such that
$\vec p_{\tilde{i}} - \vec p_{r\cdot \tilde{i}} = \vec p_{\tilde{i}} - \Phi(r)\cdot\vec p_{\tilde{i}}$
is in the direction $\vec v$.  Because $\Phi(r)$ is a rotation through angle $\pi$, this means that
$\vec p_{\tilde{i}}$ and $\vec p_{r\cdot \tilde{i}}$ lie on a line through the rotation center of $r$
in the direction $\vec v$.  Because $X''$ is connected, and edge directions are fixed under an
order $2$ rotation, the same is then true for every vertex in the
connected component $\tilde{X}_0''$ of the lifted realization $\tilde{X}(\vec p,\Phi)$ that contains $\vec p_{\tilde{i}}$.

The lemma then follows by considering translates of $\tilde{X}_0''$.
\end{proof}
Considering the connected components one at a time, \lemref{collapse-proof-G2-gc11-conn} readily
implies
\begin{lemma}\lemlab{collapse-proof-G2-gc11}
Let $(X',\bgamma,\vec d)$ be a g.c.-$(1,1)$ graph, and let $\vec d$
assign the same direction $\vec v$ to every edge.  Then, in any realization of the lifted
crystallographic direction network $(\tilde{X},\varphi,\vec d)$, every vertex
and every edge lies on a line in the direction $\vec v$ through a rotation center.
\end{lemma}
To complete the proof, we recall that the $\rho$-image of $(X,\bgamma)$ contains
two linearly independent translations $t$ and $t'$.  If $\Phi(t)$ or $\Phi(t')$ is not in the
direction $\vec v$, by \lemref{collapse-proof-G2-gc11} there is some edge in the lifted
realization $\tilde{X}(\vec p,\Phi)$ that has one endpoint on one line in the direction
$\vec v$ and the other endpoint on a translation of this line.  This is incompatible with all
edge edges of $X$ being assigned the direction $\vec v$, so we conclude that $\Phi(t)$
and $\Phi(t')$ are both in the direction $\vec v$, from which the Lemma follows.
\eop

\section{Proof of \propref{gamma-laman-circuit-collapse}}\seclab{gamma-laman-circuit-collapse}
We now prove the ``Maxwell direction'' of \theoref{direction}:
\gammalamancircuitcollapse
In the proof, we will use the following statement (cf. \cite[Lemma 14.2]{MT10} for the case when
the $\rho$-image is a translation subgroup)
\begin{lemma}\lemlab{collapsed-dimensions}
Let $(G,\bgamma,\vec d)$ be a colored direction network on a colored graph $(G,\bgamma)$ with connected components
$G_1,G_2,\ldots, G_c$.  Then $(G,\bgamma,\vec d)$ has at least
\[
\rep_{\Gamma_k}(\Trans(\Gamma_k)) - \rep_{\Gamma_k}(G) + \sum_{i=1}^c T(G_i)
\]
dimensions of solutions with all edges collapsed and the origin as a rotation center.
\end{lemma}
We defer the proof of \lemref{collapsed-dimensions} to \secref{proof-of-collapsed-dimensions} and first show how
\lemref{collapsed-dimensions} implies \propref{gamma-laman-circuit-collapse}.
\subsection{Proof of \propref{gamma-laman-circuit-collapse}}
Let $(G,\bgamma)$ be a $\Gamma$-colored-Laman circuit with $n$ vertices, $m$ edges,and $c$
connected components $G_1,G_2,\ldots G_c$.
By \lemref{gamma-laman-circuit-sparsity}, we have
\[
m = 2n + \rep_{\Gamma_k}(G) - \sum_{i=1}^c T(G_i)
\]
It follows from \propref{crystal-collapse} that for generic directions,
a colored direction network $(G,\bgamma,\vec d)$ has a
\[
2n + \rep_{\Gamma_k}(\Trans(\Gamma_k)) - m = \rep_{\Gamma_k}(\Trans(\Gamma_k)) - \rep_{\Gamma_k}(G) + \sum_{i=1}^c T(G_i)
\]
dimensional space of realizations with the origin as a rotation center.
Applying \lemref{collapsed-dimensions} shows that in all of them
every edge is collapsed.
\eop

\subsection{Proof of \lemref{collapsed-dimensions}}\seclab{proof-of-collapsed-dimensions}
For now, assume that the colored graph $(G,\bgamma)$ is connected. Select a
base vertex $b$.

\paragraph{Representations that are trivial on $\Trans(G,b)$}
Let $\Phi\in \overline{\Rep_{\Gamma_k}}(\Trans(\Gamma_k))$ be such that
\[
\Phi(t) = ((0,0),\Id)
\]
for all translations $t\in \Trans(G,b)$.
These representations form a  $(\rep_{\Gamma_k}(\Trans(\Gamma_k)) - \rep_{\Gamma_k}(G))$-dimensional
space.

\paragraph{Collapsed realizations for a fixed representation}
Now we show that there are $T(G)$ dimensions of realizations with all edges collapsed.  We
do this with an explicit construction.  There are two cases.

\noindent
\textbf{Case 1: $T(G) = 2$.}
In this case, we know that the subgroup generated by $\rho(\pi_1(G,b))$ is a
translation subgroup.  Fix a spanning tree $T$ of $G$ and a point $\vec p_b\in \R^2$.
We will construct a realization with vertex $b$ mapped to $\vec p_b$ and all edges
collapsed.

For any pair of vertices $i$ and $j$, define $Q_{ij}$ to be the path in $T$ from $i$ to $j$
and define $\eta_{ij}$ to be $\rho(Q_{ij})$.  We then set $\vec p_i = \Phi(\eta_{bi}^{-1})\cdot \vec p_b$
for all vertices $i\in V(G)$ other than $b$.  Thus all vertex locations are determined by $\vec p_b$,
giving a $2$-dimensional space of realizations for this $\Phi$. We need to
check that all edges are collapsed.

If $ij$ is an edge of $T$ with color $\gamma_{ij}$, then we have
\[
\gamma_{ij}^{-1} = \eta_{bj}^{-1}\cdot\eta_{bi}
\]
Using this relation, we see that
\[
\vec p_j = \Phi(\eta_{bj}^{-1})\cdot \vec p_b = \Phi(\gamma_{ij}^{-1}\cdot\eta_{bi}^{-1})\cdot \vec p_b =
\Phi(\gamma_{ij}^{-1})\cdot \vec p_i
\]
so the edge $ij$ is collapsed.  If $ij$ is not an edge in $T$, then the fundamental closed path $P_{ij}$
of $ij$ relative to $T$ and $b$ follows $Q_{bi}$, crosses $ij$, and returns to $b$ along $Q_{jb}$.  This
gives us the relation
\[
\gamma_{ij} = \eta_{bi}^{-1}\cdot \rho(P_{ij})\cdot \eta_{bj}
\]
We then compute
\[
\Phi(\gamma_{ij})\cdot\vec p_j = (\Phi(\eta_{bi}^{-1})\cdot \Phi(\rho(P_{ij}))\cdot \Phi(\eta_{bj}))\cdot \vec p_j
\]
Since $\Phi$ is trivial on the $\rho$-images of fundamental closed paths, the r.h.s. simplifies to
\[
\Phi(\eta_{bi}^{-1})\cdot\Phi(\eta_{bj})\cdot\vec p_j = 	\Phi(\eta_{bi}^{-1})\cdot\vec p_b = \vec p_i
\]
and we have shown that all edges are collapsed.
\vspace{1ex}

\noindent
\textbf{Case 2: $T(G) = 0$.} We adopt the notation from Case 1.  As before, we fix a spanning tree $T$
and a representation $\Phi$ that is trivial on the translation subgroup $\Trans(G,b)$.  By \lemref{subgrpgen},
$\rho(\pi_1(G,b))$ is generated by a translation subgroup $\Gamma'<\Trans(G,b)$ and a rotation $r\in \Gamma_k$.
We set $\vec p_b$ to be on the rotation center of $\Phi(r)$ and define the rest of the $\vec p_i$
as before: $\vec p_i = \Phi(\eta_{bi}^{-1})\cdot \vec p_b$.  Observe that $\Phi(r)$ then fixes $\vec p_b$.

For edges $ij$ in the tree $T$, the argument that $ij$ is collapsed from Case 1 applies verbatim.  For non-tree
edges $ij$, a similar argument relating the fundamental closed path $P_{ij}$ to $Q_{bi}$ and $Q_{bj}$ yields
the relation
\[
\gamma_{ij} = \eta_{bi}^{-1}\cdot \rho(P_{ij})\cdot \eta_{bj}
\]
Since $\Phi$ is trivial on translations $t\in \Gamma'$, we see that
\[
\Phi(\gamma_{ij}) = \Phi(\eta_{bi}^{-1})\cdot\Phi(r)\cdot\Phi(\eta_{bj})
\]
We then compute
\[
\Phi(\gamma_{ij})\vec p_j = \Phi(\eta_{bi}^{-1})\cdot\Phi(r)\cdot\Phi(\eta_{bj})\cdot \vec p_j =
\Phi(\eta_{bi}^{-1})\cdot\Phi(r)\cdot\vec p_b
\]
Because $\Phi(r)\cdot\vec p_b=\vec p_b$, the r.h.s. simplifies to $\vec p_i$, and so the edge $ij$ is
collapsed.

\paragraph{Multiple connected components}
The proof of the lemma is completed by considering connected components one at a time to remove the assumption that
$G$ is connected.
\eop

\chapter{Rigidity}\chaplab{rigidity}
\section{Crystallographic and colored frameworks}\seclab{continuous}
We now return to the setting of crystallographic frameworks, leading to the proof of \theoref{main}
in \secref{main-proof}.  The overall structure is very similar to \cite[Sections 16--18]{MT10},
but we give sufficient detail for completeness.  Here is the roadmap to the rest of the paper:
\begin{itemize}
\item In this section we give the \emph{continuous rigidity} theory for crystallographic frameworks and the
related \emph{colored crystallographic frameworks}.
\item \secref{infinitesimal} introduces \emph{infinitesimal rigidity} and defines \emph{genericity}
for crystallographic frameworks.
\item The proof of \theoref{main} is then in \secref{main-proof}.
\item We conclude with a discussion of cone frameworks and the proof of \theoref{cone} in
\secref{cone-rigidity}.
\end{itemize}

\subsection{Crystallographic frameworks}
We recall the following definition from the introduction: a \emph{crystallographic framework}
$\Gal$ is given by:
\begin{itemize}
\item An infinite graph $\tilde{G}$
\item A free action $\varphi$ on $\tilde{G}$ by a crystallographic group $\Gamma$ with finite quotient
\item An assignment of a \emph{length} $\ell_{ij}$ to each edge $ij\in E(\tilde{G})$
\end{itemize}
In what follows, $\Gamma$ will always be one of the groups $\Gamma_2$, $\Gamma_3$, $\Gamma_4$,
or $\Gamma_6$.

\subsection{The realization space}
A \emph{realization} $\tilde{G}(\vec p,\Phi)$ of a crystallographic framework $\Gal$ is given by
an assignment
$\vec p=\left(\vec p_{i}\right)_{i\in\Vtilde}$ of points to the vertices of $\Gtilde$ and
a representation $\Phi$ of $\Gamma \into \Euc(2)$ by Euclidean isometries
acting discretely and co-compactly, such that
\begin{align}
||\vec p_i - \vec p_j||  =  \elltilde_{ij} & \text{\qquad for all edges $ij\in \Etilde$} \label{lengthsX} \\
\Phi(\gamma)\cdot \vec p_i  =  \vec p_{\gamma(i)} &
\text{\qquad for all group elements $\gamma\in \Gamma$ and vertices $i\in\Vtilde$} \label{equivariantX}
\end{align}
We see that \eqref{equivariantX} implies that, to be realizable at all, the framework $\Gal$ must assign the same
length to each edge in every $\Gamma$-orbit of the action $\varphi$.  The condition \eqref{lengthsX} is the standard one
from rigidity theory that says the distances between endpoints of each edge realize the specified lengths.

We define the \emph{realization space} $\mathcal{R}\Gal$ (shortly $\mathcal{R}$)
of a crystallographic framework to be the set of all realizations
\[
\mathcal{R}\Gal = \left\{ (\vec p,\Phi) : \text{$\tilde{G}(\vec p,\Phi)$ is a realization of $\Gal$}
\right\}
\]

\subsection{The configuration space}
The group $\Euc(2)$ of Euclidean isometries acts naturally on the realization space.  Let $\psi\in \Euc(2)$ be an
isometry. For any point $(\vec p,\Phi)\in \mathcal{R}$,
\[
(\psi\circ \vec p, \Phi^\psi)
\]
is a point in $\mathcal{R}$ as well where $\Phi^\psi$ is the representation defined by
$$\Phi^\psi(\gamma) = \psi \Phi(\gamma) \psi^{-1}.$$
We define the \emph{configuration space} $\mathcal{C}\Gal$
(shortly $\mathcal{C}$) of a crystallographic framework to be the quotient $\mathcal{R}/\Euc(2)$ of
the realization space by Euclidean isometries.

Since the spaces $\mathcal{R}$ and $\mathcal{C}$ are subsets of an infinite-dimensional space, there are
some technical details to check that we omit in the interested of brevity.
Interested readers can find a development for the periodic setting
in \cite[Appendix A]{MT10a}\footnote{The reference \cite{MT10a}
is an earlier version of \cite{MT10}.}.  The present crystallographic case proceeds along the
same lines.

\subsection{Rigidity and flexibility}
A realization $\tilde{G}(\vec p,\Phi)$ is defined to be (continuously) \emph{rigid} if it is isolated in the configuration
space $\mathcal{C}$.  Otherwise it is \emph{flexible}.  As the definition makes clear, rigidity is a \emph{local} property
that depends on a realization.

A framework that is rigid, but ceases to be so if any orbit of bars is removes is defined to be \emph{minimally rigid}.

\subsection{Colored crystallographic frameworks}
In principle, the realization and configuration spaces $\mathcal{R}\Gal$ and $\mathcal{C}\Gal$ of crystallographic
frameworks could be complicated infinite dimensional objects.  In this section, we will show that they are, in fact,
equivalent to the finite-dimensional configuration spaces of \emph{colored crystallographic frameworks}, which will
be technically simpler to work with.

A \emph{colored crystallographic framework} (shortly a \emph{colored framework}) is a triple $(G,\bgamma,\bm{\ell})$,
where $(G,\bgamma)$ is a $\Gamma_k$-colored graph and $\bm{\ell}=(\ell_{ij})_{ij\in E(G)}$ is an assignment of a length
to each edge.

The relationship between crystallographic and colored frameworks is similar to that between their direction network
counterparts: using the arguments for  Lemmas \ref{lemma:cone-direction-network-lifts} and
\ref{lemma:cone-direction-network-lifts2} we see that each colored framework has a well-defined lift to
a crystallographic framework and each crystallographic framework has, as its quotient, a colored framework.

\subsection{The colored realization and configuration spaces}
A \emph{realization} $G(\vec p,\Phi)$ of a colored framework is an assignment of points
$\vec p = (\vec p_i)_{i\in V(G)}$ and a representation $\Phi$ of $\Gamma_k$ by Euclidean isometries acting discretely and
cocompactly such that
\[
||\Phi(\gamma_{ij})\cdot\vec p_j - \vec p_i||^2 = \ell_{ij}^2
\]
for all edges $ij\in E(G)$.  The \emph{realization space} $\mathcal{R}(G,\bgamma,\ell)$ is then defined to be
\[
\mathcal{R}(G,\bgamma,\ell) = \left\{ (\vec p,\Phi) :
\text{$G(\vec p,\Phi)$ is a realization of $(G,\bgamma,\bm{\ell})$}
\right\}
\]
The Euclidean group $\Euc(2)$ acts naturally on $\mathcal{R}(G,\bgamma,\ell)$ by
\[
\psi\cdot(\vec p,\Phi) = (\psi\cdot\vec p,\Phi^\psi)
\]
where $\psi$ is a Euclidean isometry.  Thus we define the
\emph{configuration space}  $\mathcal{C}(G,\bgamma,\ell)$ to be the quotient
$\mathcal{R}(G,\bgamma,\ell)/\Euc(2)$ of the realization space by the Euclidean group.

\subsection{The modified configuration space}
Because it is technically simpler, we will consider the modified realization space
$\mathcal{R'}(G,\bgamma,\ell)$, which we define to be:
\[
\mathcal{R'}(G,\bgamma,\ell) = \left\{ (\vec p,\Phi) :
\text{$G(\vec p,\Phi)$ is a realization of $(G,\bgamma,\bm{\ell})$ with $\Phi(r_k)$ fixing the origin}
\right\}
\]
Recall that $r_k$ is the rotation of order $k$ that is one of the generators of $\Gamma_k$.  The modified
configuration space $\mathcal{C'}(G,\bgamma,\ell)$ is then defined to be the quotient $\mathcal{R'}(G,\bgamma,\ell)/O(2)$
of the modified realization space by the orthogonal group $O(2)$.  Since every representation $\Phi\in \Rep(\Gamma_k)$
is conjugate to a representation $\Phi'$ that has the origin as a rotation center by a Euclidean translation,
this next lemma follows immediately.
\begin{lemma}\lemlab{modified-config-space}
Let $(G,\bgamma,\bm{\ell})$ be a colored framework.  Then the configuration space $\mathcal{C}(G,\bgamma,\ell)$
is homeomorphic to the modified configuration space $\mathcal{C'}(G,\bgamma,\ell)$.
\end{lemma}

From the definition and \lemref{repspace} we see that the modified configuration space is an
algebraic subset of $\R^{2n}\times \R^4$, for $\Gamma_2$ and of $\R^{2n}\times \R^2$ for
$\Gamma_k$ with $k=3,4,6$.

\subsection{Colored rigidity and flexibility}
We now can define rigidity and flexibility in terms of colored frameworks.  A realization $G(\vec p,\Phi)$ of a colored
framework is \emph{rigid} if it is isolated in the configuration space and otherwise \emph{flexible}.  \lemref{modified-config-space}
implies that a realization is rigid if and only if it is isolated in the modified configuration space.

\subsection{Equivalence of crystallographic and colored rigidity}
The connection between the rigidity of crystallographic and colored frameworks
is captured in the following proposition, which says that we can switch between the two models.
\begin{prop}\proplab{colored-and-crystallographic-frameworks}
Let $\Gal$ be a crystallographic framework and let $(G,\bgamma,\bm{\ell})$ be an associated
colored framework quotient.  Then the configuration spaces $\mathcal{C}\Gal$ and
$\mathcal{C'}(G,\bgamma,\ell)$ are homeomorphic.
\end{prop}
\begin{proof}
This follows from the definitions and a straightforward computation. %
\end{proof}

\section{Infinitesimal and generic rigidity}\seclab{infinitesimal}
As discussed above, the modified realization space $\mathcal{R}'(G,\bgamma,\bm{\ell})$ of a colored
framework is an algebraic subset of $\mathbb{R}^{2n+2r}$, where $r$ is the rank of the translation
subgroup $\Trans(\Gamma_k)$.  The coordinates are given as follows:
\begin{itemize}
\item The first $2n$ coordinates are the coordinates of the points $\vec p_1, \vec p_2,\ldots, \vec p_n$
\item The final $2r$ coordinates are the vectors $v_i$ specifying the representation of the translation
subgroup $\Trans(\Gamma_k)$.  (Since we have ``pinned'' a rotation center to the origin, the vector
$w$ from \lemref{repspace} is also fixed.)
\end{itemize}

\subsection{Infinitesimal rigidity}
As is typical in the derivation of Laman-type theorems, we relax the the condition of rigidity, we
linearize the problem by considering the tangent space of $\mathcal{R}'(G,\bgamma,\bm{\ell})$
near a realization $G(\vec p,\Phi)$.

The vectors in the tangent space are infinitesimal motions of the framework, and
they can be characterized as follows.  Let
$(\vec q, u_1, u_2) \in \R^{2n+4}$ for $k =2$ or $(\vec q, u_1) \in \R^{2n+2}$ for $k =3,4,6$.
To this vector there is an associated representation $\Phi'$ defined by
$\Phi'(r_k) = (0, R_k)$ and $\Phi'(t_i) = (u_i, \Id)$.
Then differentiation of the length equations yield this linear system ranging over all
edges $ij \in E(G)$:
\begin{equation}
\iprod{\Phi(\gamma_{ij}) \cdot \vec p_j - \vec p_i}{\Phi'(\gamma_{ij}) \cdot \vec q_j  -  \vec q_i}  \label{eq:infinitesimal}
\end{equation}
The given data are the $\vec p_i$ and $\Phi$, and then unknowns are the $\vec q_i$ and $\Phi'$.
A realization $G(\vec p, \Phi)$ of a colored framework
is defined to be  \emph{infinitesimally rigid} if the  system \eqref{eq:infinitesimal}
has a $1$-dimensional solution space.  A realization that is infinitesimally rigid but ceases to be
so when any colored edge is removed is minimally infinitesimally rigid.

\subsection{Infinitesimal rigidity implies rigidity}
A standard kind of result relating infinitesimal rigidity and rigidity for generic frameworks holds in our setting.
Since our realization space
is finite, adapting standard arguments (see e.g. \cite{AR78}) to our situation is not hard, so we omit a proof.
\begin{lemma}\lemlab{infinitesimal-rigidity-implies-rigidity}
If a realization $G(\vec p,\Phi)$ of a colored framework is infinitesimally rigid, then it is rigid.
\end{lemma}

\subsection{Generic rigidity}
The converse of \lemref{infinitesimal-rigidity-implies-rigidity} does not hold in general, but it does for
nearly all realizations.  Let $(G,\bgamma,\bm{\ell})$ be a colored framework.  A realization
$G(\vec p,\Phi)$ is defined to be \emph{regular} for $(G,\bgamma,\bm{\ell})$
if the rank of the system \eqref{eq:infinitesimal} is maximal over all realizations.

Whether a realization is regular depends on both the colored graph $(G,\bgamma)$ and the given
lengths $\bm{\ell}$.  Let $G(\vec p,\Phi)$ be a regular realization of a colored framework.
If, in addition, the rank of \eqref{eq:infinitesimal} at $G(\vec p,\Phi)$ is maximal over all
realizations of colored frameworks with the same colored graph $(G,\bgamma)$, we define
$G(\vec p,\Phi)$ to be \emph{generic}.

We define the rank of \eqref{eq:infinitesimal} at a generic realization to be its \emph{generic rank}.
Since it depends on formal minors of the matrix underlying \eqref{eq:infinitesimal} only, it is a
property of the colored graph $(G,\bgamma)$.

If $(G,\bgamma,\bm{\ell})$ is a framework with generic realizations, it is immediate that the set
of non-generic realizations is a proper algebraic subset of the realization space.  Alternatively,
if we consider frameworks as being induced by realizations, the set of non-generic realizations
is a proper algebraic subset of $\R^{2n+2r}$, where $r=1$ for $\Gamma_{3}$, $\Gamma_{4}$, and
$\Gamma_{6}$, and $r=2$ for $\Gamma_2$.

For generic realizations, a standard argument (again, along the lines of \cite{AR78}) shows that
rigidity and infinitesimal rigidity coincide.
\begin{prop}\proplab{generic-rigidity}
A generic realization of a colored framework $(G,\bgamma,\bm{\ell})$ is rigid if and only if
it is infinitesimally rigid.
\end{prop}

\section{Proof of \theoref{main}}\seclab{main-proof}
All the tools are in place to prove our main theorem:
\mainthm
The proof occupies the rest of this section.

\subsection{Reduction to colored frameworks}
By \propref{colored-and-crystallographic-frameworks}, it is sufficient to prove the statement
of \theoref{main} for colored frameworks.  \propref{generic-rigidity} then implies that the
Theorem will follow from a characterization of generic infinitesimal rigidity for colored
frameworks.

Thus, to prove the theorem, we show that, for a colored graph $(G,\bgamma)$ with
$n$ vertices and $m=2n + \rep_{\Gamma_k}(\Trans(\Gamma_k)) - 1$ edges,
the generic rank of the system \ref{eq:infinitesimal} is $m$ if and only if
$(G,\bgamma)$ is a $\Gamma$-colored-Laman graph.

\subsection{Necessity: the ``Maxwell direction''}
We recall the definition of the sparsity function $h(G)$ from \secref{gamma-laman},
which defines $\Gamma$-colored-Laman graphs.  We have, for a colored graph $(G,\bgamma)$
with $n$ vertices and $c$ connected components $G_1,G_2,\ldots,G_c$,
\[
h(G) = 2n + \rep_{\Gamma_k}(G) - 1 - \sum_{i=1}^c T(G_i)
\]
That colored-Laman-sparsity is necessary for the system \ref{eq:infinitesimal} to have independent
equations is captured in the following proposition.
\begin{prop}\proplab{maxwell-direction}
Let $(G,\bgamma)$ be a colored graph.  Then the generic rank of the system \eqref{eq:infinitesimal} is at most
$h(G)$.
\end{prop}
\begin{proof}
Let $G(\vec p,\Phi)$ be any realization of a colored framework on a colored graph $(G,\bgamma)$ with
no collapsed edges.  That is select a representation $\Phi$ of $\Gamma_k$ and points $\vec p_i$, such
that, $\Phi(\gamma_{ij})\cdot\vec p_j\neq \vec p_i$ for all edges $ij\in E(G)$.

We now define the direction $\vec d_{ij}$ to be $(\Phi(\gamma_{ij})\cdot\vec p_j - \vec p_i)^{\perp}$ for
each edge $ij\in E(G)$.  These directions define a colored direction network $(G,\bgamma,\vec d)$ with the
property that any solution to this direction network corresponds to an infinitesimal motion of the
colored framework realization $G(\vec p,\Phi)$.

\lemref{collapsed-dimensions} implies that there are
\[
\rep_{\Gamma_k}(\Trans(\Gamma_k)) - \rep_{\Gamma_k}(G) + \sum_{i=1}^c T(G_i)
\]
dimensions of  realizations with every edge collapsed.  By construction, there is a non-collapsed realization of this
direction network as well: it is simply $(\vec p,\Phi)$ rotated by $\pi/2$.  Since this is not obtained by taking linear combinations of
realizations where every edge is collapsed, the dimension of the space of infinitesimal motions is always at
least
\[
\rep_{\Gamma_k}(\Trans(\Gamma_k)) - \rep_{\Gamma_k}(G) + \sum_{i=1}^c T(G_i) + 1
\]
The proposition follows by subtracting from $2n + \rep_{\Gamma_k}(\Trans(\Gamma_k))$ and comparing to $h(G)$.
\end{proof}

\subsection{Sufficiency: the ``Laman direction''}
The other direction of the proof of \theoref{main} is this next proposition
\begin{prop}\proplab{laman-direction}
Let $(G,\bgamma)$ be a $\Gamma$-colored-Laman graph.  Then the generic rank of
the system \ref{eq:infinitesimal} is $h(G)$.
\end{prop}
\begin{proof}
It is sufficient to construct a single example at which this rank is attained, since the generic rank is
always at least the rank for any specific realization.  We will do this using direction networks.

Let $(G,\bgamma)$ be a $\Gamma$-colored-Laman graph, and select a direction $\vec d_{ij}$ for each edge
$ij\in E(G)$, such that both $\vec d$ and $\vec d^{\perp}=(\vec d_{ij}^\perp)$ are
generic in the sense of \theoref{direction}.
By \theoref{direction}, the colored direction network $(G,\bgamma,\vec d)$ has a unique,
faithful solution $(\vec p,\Phi)$, which implies that, for all edges $ij\in E(G)$
\[
\Phi(\gamma_{ij})\cdot \vec p_j - \vec p_i = \alpha_{ij}\vec d_{ij}
\]
for some non-zero scalar $\alpha_{ij}\in \R$.  It follows that, by replacing $\vec d_{ij}$ with
$\Phi(\gamma_{ij})\cdot \vec p_j - \vec p_i$ in the direction realization system
\eqref{colored-crystal-directions} we obtain \eqref{eq:infinitesimal}.  Since $\vec d^\perp$ is
also generic for \theoref{direction}, we conclude that \eqref{eq:infinitesimal} has full rank
as desired.
\end{proof}

\section{Cone frameworks}\seclab{cone-rigidity}
For the group $\Z/k\Z$, the counterpart of \theoref{main} is
\conethm
The theory for cone frameworks follows the same lines as that for $\Gamma_k$-crystallographic frameworks.
Since all the steps from Sections \ref{sec:continuous}--\ref{sec:main-proof} go through with appropriate
modifications (which are simplifications) we omit the details in the interest of space.

\bibliographystyle{plainnat}

\end{document}